\documentclass[final]{amsart} 
\usepackage[pdftex]{graphicx}
\usepackage{mathrsfs}
\usepackage{verbatim}
\usepackage{lmodern}
\usepackage{dsfont}

\usepackage{enumitem}
\setlist{nosep}

\usepackage[hidelinks]{hyperref}

\usepackage{amsmath,amssymb,amsthm,amscd,yhmath,mathtools}
\usepackage[all,cmtip]{xy}
\usepackage[pdftex]{graphicx}
\usepackage{tikz}
\usetikzlibrary{decorations.markings}
\usetikzlibrary{arrows}

\theoremstyle{plain}
\newtheorem{thm}{Theorem}[section]
\newtheorem{prop}[thm]{Proposition}
\newtheorem{cor}[thm]{Corollary}
\newtheorem{lemma}[thm]{Lemma}

\theoremstyle{definition}
\newtheorem{defn}[thm]{Definition}
\newtheorem*{rem}{Remark}

\newcommand{\e}{\varepsilon}
\newcommand{\sm}{\smallsetminus}
\newcommand{\es}{\ensuremath{\varnothing}}

\newcommand{\A}{\mathbb{A}}
\newcommand{\C}{\mathbb{C}}
\newcommand{\F}{\mathbb{F}}
\newcommand{\G}{\mathbb{G}}
\renewcommand{\L}{\mathbb{L}}
\newcommand{\N}{\mathbb{N}}
\renewcommand{\P}{\mathbb{P}}
\newcommand{\Q}{\mathbb{Q}}
\newcommand{\R}{\mathbb{R}}
\newcommand{\T}{\mathbb{T}}
\newcommand{\Z}{\mathbb{Z}}

\newcommand{\Qbar}{\overline{\Q}}

\newcommand{\I}{\mathcal{I}}
\renewcommand{\O}{\mathcal{O}}
\newcommand{\K}{\mathcal{K}}

\newcommand{\Cc}{\mathcal{C}}
\newcommand{\Dc}{\mathcal{D}}
\newcommand{\Lc}{\mathcal{L}}
\renewcommand{\Mc}{\mathcal{M}}
\newcommand{\Pc}{\mathcal{P}}
\newcommand{\Qc}{\mathcal{Q}}
\newcommand{\Rc}{\mathcal{R}}
\newcommand{\Sc}{\mathcal{S}}
\newcommand{\Vc}{\mathcal{V}}
\newcommand{\Xc}{\mathcal{X}}

\newcommand{\As}{\mathscr{A}}

\newcommand{\Ms}{\mathscr{M}}

\newcommand{\Rs}{\mathscr{R}}

\renewcommand{\a}{\mathfrak{a}}

\newcommand{\n}{\mathfrak{n}}
\newcommand{\p}{\mathfrak{p}}
\newcommand{\q}{\mathfrak{q}}
\newcommand{\D}{\mathfrak{D}}

\newcommand{\If}{\mathfrak{I}}

\newcommand{\Is}{\mathscr{I}}
\newcommand{\Ish}{\widehat{\Is}}

\newcommand{\Rhat}{\widehat{R}}

\newcommand{\Yhat}{\widehat{Y}}
\newcommand{\Chat}{\widehat{C}}
\newcommand{\Ihat}{\widehat{I}}
\newcommand{\Vhat}{\widehat{V}}
\newcommand{\Zhat}{\widehat{Z}}

\newcommand{\mhat}{\widehat{m}}
\newcommand{\Sh}{\widehat{S}}

\newcommand{\Ct}{\widetilde{C}}
\newcommand{\Et}{\widetilde{E}}
\newcommand{\Ut}{\widetilde{U}}
\newcommand{\pt}{\widetilde{p}}

\newcommand{\Cthat}{\widehat{\Ct}}

\newcommand{\Rt}{\widetilde{R}}

\newcommand{\uprod}[1]{\mathcal{U}\!\left(#1\right)}

\renewcommand{\k}{\Bbbk}
\newcommand{\injection}{\hookrightarrow}
\newcommand{\surjection}{\twoheadrightarrow}

\newcommand{\patch}{\mathscr{P}}

\newcommand{\Fbar}{\overline{\F}}
\newcommand{\Ebar}{\overline{E}}
\newcommand{\Sbar}{\overline{S}}
\newcommand{\Rbar}{\overline{R}_\infty}
\newcommand{\Ibar}{\overline{\I}}
\newcommand{\Mbar}{\overline{M}_\infty}
\newcommand{\Tbar}{\overline{\T}}
\newcommand{\ebar}{\overline{\e}}
\newcommand{\chibar}{\overline{\chi}}
\newcommand{\psibar}{\overline{\psi}}
\newcommand{\varpibar}{\overline{\varpi}}

\newcommand{\face}{\preceq}

\newcommand{\Ab}{\mathbf{Ab}}

\newcommand{\rbar}{\overline{r}}
\newcommand{\rhobar}{\overline{\rho}}

\newcommand{\uf}{\mathfrak{F}}
\newcommand{\prF}{\mathfrak{Z}}

\newcommand{\ds}{\displaystyle}

\newcommand{\half}{\mathcal{H}}

\DeclareMathOperator{\Hom}{Hom}
\DeclareMathOperator{\Aut}{Aut}
\DeclareMathOperator{\End}{End}
\DeclareMathOperator{\Gal}{Gal}
\DeclareMathOperator{\Ann}{Ann}
\DeclareMathOperator{\Frob}{Frob}

\DeclareMathOperator{\GL}{GL}

\DeclareMathOperator{\Spec}{Spec}
\newcommand{\rSpec}{\underline{\Spec}}
\DeclareMathOperator{\Proj}{Proj}
\DeclareMathOperator{\Pic}{Pic}
\DeclareMathOperator{\Cl}{Cl}
\DeclareMathOperator{\supp}{supp}
\DeclareMathOperator{\Cone}{Cone}

\DeclareMathOperator{\im}{im}
\DeclareMathOperator{\id}{id}

\DeclareMathOperator{\tr}{tr}
\DeclareMathOperator{\et}{\acute{e}t}
\DeclareMathOperator{\pr}{pr}
\DeclareMathOperator{\Nm}{Nm}

\DeclareMathOperator{\st}{st}
\DeclareMathOperator{\fl}{fl}
\DeclareMathOperator{\loc}{loc}
\DeclareMathOperator{\univ}{univ}
\renewcommand{\min}{\operatorname{min}}
\renewcommand{\mod}{\operatorname{mod}}

\DeclareMathOperator{\Div}{Div}
\renewcommand{\div}{\operatorname{div}}

\newcommand{\Set}{{\bf Set}}

\newcommand{\etale}{{\'etale}}

\DeclareMathOperator{\rank}{rank}
\DeclareMathOperator{\genrank}{g.rank}

\newcommand{\semi}{\mathsf{S}}

\newcommand{\invlim}{\varprojlim}
\newcommand{\dirlim}{\varinjlim}

\newcommand{\dual}{\vee}

\title{Patching and Multiplicity $2^k$ for Shimura Curves}
\author{Jeffrey Manning}

\begin{document}

\maketitle

\begin{abstract}
We use the Taylor--Wiles--Kisin patching method to investigate the multiplicities with which Galois representations occur in the mod $\ell$ cohomology of Shimura curves over totally real number fields. Our method relies on explicit computations of local deformation rings done by Shotton, which we use to compute the Weil class group of various deformation rings. Exploiting the natural self-duality of the cohomology groups, we use these class group computations to precisely determine the structure of a patched module in many new cases in which the patched module is not free (and so multiplicity one fails).

Our main result is a ``multiplicity $2^k$'' theorem in the minimal level case (which we prove under some mild technical hypotheses), where $k$ is a number that depends only on local Galois theoretic information at the primes dividing the discriminant of the Shimura curve. Our result generalizes Ribet's classical multiplicity 2 result and the results of Cheng, and provides progress towards the Buzzard--Diamond--Jarvis local-global compatibility conjecture. We also prove a statement about the endomorphism rings of certain modules over the Hecke algebra, which may have applications to the integral Eichler basis problem.
\end{abstract}

\section{Introduction}\label{sec:intro}

\subsection{Overview and statement of main result}

One of the most powerful tools in the study of the Langlands program is the \emph{Taylor--Wiles--Kisin patching method} which, famously, was originally introduced by Taylor and Wiles \cite{Wiles,TaylorWiles} to prove Fermat's Last Theorem, via proving a special case of Langlands reciprocity for $\GL_2$. 

In its modern formulation (due to Kisin \cite{Kisin} and others) this method considers a ring $R_\infty$, which can be determined explicitly from local Galois theoretic data, and constructs a maximal Cohen--Macaulay module $M_\infty$ over $R_\infty$ by gluing together various cohomology groups. Due to its construction, $M_\infty$ is closely related to certain automorphic representations, and so determining its structure has many applications in the Langlands program beyond simply proving reciprocity.

A few years after Wiles' proof, Diamond \cite{DiamondMult1} and Fujiwara \cite{Fujiwara} discovered that patching can also be used to prove mod $\ell$ multiplicity one statements in cases where the $q$-expansion principle does not apply. In this argument they consider a case when the ring $R_\infty$ is formally smooth, and so the Auslander-Buchsbaum formula allows him to show that $M_\infty$ is free over $R_\infty$, a fact which easily implies multiplicity one. There are however, many situations arising in practice in which $R_\infty$ is not formally smooth, and so this method cannot be used to determine multiplicity one statements. 

In this paper, we introduce a new method for determining the structure of a patched module $M_\infty$ arising from the middle degree cohomology of certain Shimura varieties, which applies in cases when $R_\infty$ is Cohen--Macaulay, but not necessarily formally smooth. Using this, we are able to compute the multiplicities for Shimura curves over totally real number fields in the minimal level case, under some technical hypotheses. Our main result is the following (which we state here, using some notation and terminology which we define later):

\begin{thm}\label{Mult 2^k}
Let $F$ be a totally real number field, and let $D/F$ be a quaternion algebra which is ramified at all but at most one infinite place of $F$. Take some irreducible Galois representation $\rhobar:G_F\to \GL_2(\Fbar_\ell)$, where $\ell>2$ is a prime which is unramified in $F$, and prime to the discriminant of $D$. Assume that:
\begin{enumerate}
	\item $\rhobar$ is automorphic for $D$. \label{automorphic}
	\item $\rhobar|_{G_v}$ is finite flat for all primes $v|\ell$ of $F$. \label{finite flat}
	\item If $v$ is an prime of $F$ at which $\rhobar$ ramifies and $\Nm(v)\equiv -1\pmod{\ell}$, then either $\rhobar|_{I_v}$ is  irreducible or $\rhobar|_{G_v}$ is absolutely reducible. \label{vexing}
	\item If $v$ is any prime of $F$ at which $D$ ramifies, then $\rhobar$ is Steinberg at $v$ and 
$\Nm(v)\not\equiv -1\pmod{\ell}$. \label{-1modl}
	\item The restriction $\rhobar|_{G_{F(\zeta_\ell)}}$ is absolutely irreducible. \label{TW conditions}
	\end{enumerate}
Let $K^{\min}\subseteq D^\times(\A_{F,f})$ be the \emph{minimal} level at which $\rhobar$ occurs, and let $X^D(K^{\min})$ be the Shimura variety (of dimension either 0 or 1) associated to $K^{\min}$. Then the \emph{multiplicity}\footnote{In the case when $D$ unramified at exactly one infinite place of $F$, $X^D(K^{\min})$ is an algebraic curve, and so this multiplicity is just the number of copies of $\rhobar$ which appear in the \etale\ cohomology group $H^1_{\et}(X^D(K^{\min}),\mu_\ell)$. In the case when $D$ is ramified at all infinite pales of $F$, $X^D(K^{\min})$ is just a discrete set of points and so $\rhobar$ does not actually appear in the cohomology. In this case, by the multiplicity we just mean the dimension of the eigenspace $H^0(X^D(K^{\min}),\F_\ell)[m]$ for $m$ the corresponding maximal ideal of the Hecke algebra.} with which $\rhobar$ occurs in the mod $\ell$ cohomology of $X^D(K^{\min})$ is $2^k$, where
\[k = \#\bigg\{v \bigg| D \text{ ramifies at } v, \rhobar \text{ is unramified at } v \text{ and, }
\rhobar(\Frob_v) \text{ is a scalar} \bigg\}.\]
\end{thm}

Note that the conditions that $D$ ramifies at $v$, $\rhobar$ is unramified at $v$ and, $\rhobar(\Frob_v)$ is a scalar imply that $\Nm(v)\equiv 1\pmod{\ell}$.

\subsection{Some history of higher multiplicity results}

A special case Theorem \ref{Mult 2^k} was first proven by Ribet \cite{RibetMult2} in the case when $F=\Q$, $D = D_{pq}$ is the indefinite quaternion algebra ramified at two primes, $p$ and $q$, and $K^{\min}$ is the group of units in a maximal order of $D_{pq}(\A_{F,f})$, i.e. the ``level one'' case (in which case $k$ is forced to be either $0$ or $1$, as $\rhobar$ is necessarily ramified at at least of of $p$ and $q$). He also proved a more general result in the case when $F=\Q$ and $D = D_p$ is the definite quaternion algebra ramified at one prime, $p$.

Yang \cite{Yang} gave a (non-sharp) upper bound on the multiplicity in the case where $F=\Q$ and $\rhobar$ was ramified in at least half of the primes in the discriminant of $D$, and also showed that multiplicities of at least 4 are achievable. Helm \cite{Helm} strengthened this result to prove the optimal upper bound of $2^k$ on the multiplicity, again in the case of $F = \Q$, but without the ramification condition for $\rhobar$.

Cheng and Fu \cite{ChengFu} generalized Ribet's multiplicity $2$ results to the case when $F$ was a totally real number field, and also generalized this somewhat to the higher weight case. More precisely, they consider two indefinite quaternion algebras $D_0$ and $D$ (called $D$ and $D'$ in their paper), where $D$ is ramified precisely at the primes where $D_0$ is ramified as well as two additional primes $\p$ and $\q$. Provided that $\rhobar$ is ramified at $\q$ and that the Shimura curve associated to $D_0$ satisfies an appropriate multiplicity result one at \emph{non-minimal} level (which holds by the $q$-expansion principle in the case when $F=\Q$ and $D_0 = M_2(\Q)$, and follows in many other cases by known cases of Ihara's Lemma), they prove Theorem \ref{Mult 2^k} for $D$. In this case $k$ is again forced to be either $0$ or $1$, as $\p$ is the only prime in the discriminant of $D'$ at which $\rhobar$ can possibly be unramified and scalar.

They further proved (see Remark 3.9 in their paper) cases of Theorem \ref{Mult 2^k} in which $k$ is allowed to be arbitrary, showing unconditionally that a multiplicity of $2^k$ is possible for any $k$. These results require that $\rhobar$ ramifies at at least $k$ primes in the discriminant of $D$, and still rely on certain non-minimal multiplicity one results.

Their arguments rely heavily on the assumption that $\rhobar$ ramifies at certain primes in the discriminant of $D$ (as well as on the non-minimal multiplicity one results) so it is very unlikely that their approach could be used to prove results in the same generality as this paper (even if the needed multiplicity one results were proven in full generality). In particular, their method cannot handle the case when $\rhobar$ is unramified at every prime in the discriminant, or even more so when $\rhobar$ is unramified and scalar at every prime in the discriminant, whereas the methods of this paper cover those cases with no added difficulty.

Based on the results over $\Q$, Buzzard, Diamond and Jarvis \cite{BDJ} formulated a mod $\ell$ local-global compatibility conjecture, which gives a conjectural description of the multiplicity for arbitrary $F$, $D$ and (prime to $\ell$) level. Theorem \ref{Mult 2^k} is a special case of this conjecture.

\subsection{Overview of our method}

The previous results relied heavily on facts about integral models of Shimura curves, as well as other results such as mod $\ell$ multiplicity one statements for modular curves (arising from the $q$-expansion principle) and Ihara's Lemma. Our approach is entirely different, and does not rely on any such statements about Shimura curves.

Our method relies on the natural self-duality of the module $M_\infty$, combined with an explicit calculation of the ring $R_\infty$ arising in the patching method, together with its Weil class group $\Cl(R_\infty)$ and dualizing module $\omega_{R_\infty}$. The fact that $M_\infty$ is self-dual, and of generic rank $1$, implies that $M_\infty$ corresponds to an element of $[M_\infty] \in \Cl(R_\infty)$ satisfying $2[M_\infty] = [\omega_{R_\infty}]$. Provided that $\Cl(R_\infty)$ is $2$-torsion free (which it is in our situation) this uniquely determines $M_\infty$ up to isomorphism. Thus determining the structure of $M_\infty$ is simply a matter of computing everything explicitly enough to determine the unique module $M_\infty$ satisfying $2[M_\infty] = [\omega_{R_\infty}]$.

While these computations may be quite difficult in higher dimensions, all of the relevant local deformation rings have been computed by Shotton \cite{Shotton} in the $\GL_2$ case, and moreover his computations show that the ring $R_\infty/\lambda$ is (the completion of) the ring of functions on a toric variety. This observation makes it fairly straightforward to apply our method in the $\GL_2$ case, and hence to precisely determine the structure of the patched module $M_\infty$.

Additionally, our explicit description of the patched module $M_\infty$ allows us to extract more refined data about the Hecke module structure of the cohomology groups, beyond just the multiplicity statements (see Theorem \ref{R=T}, below). This has potential applications to the integral Eichler basis problem and the study of congruence ideals.

\subsection{An alternate approach to the multiplicity $2$ result of \cite{CG18}}
Another approach to proving a multiplicity $2$ statement by analyzing the structure of $R_\infty$, in a slightly different context, was given by Calegari and Geraghty in \cite[Section 4]{CG18}. Their approach again relies on explicit computations of local deformation rings, done in \cite{Snowden} in their context, as well as a computation of the dualizing module $\omega_{R_\infty}$ (which is also necessary for our method). Unlike the method described in this paper however, their method also crucially uses the work of \cite{Gross} on the $p$-divisible group of $J_1(N)$, and its Hecke module structure, which was ultimately proved using the $q$-expansion principle. This makes it somewhat difficult to generalize their approach.

Our method can be used to provide an alternate proof of Calegari and Geraghty's multiplicity $2$ result, without the reliance on the results of \cite{Gross}. In fact one can check that (in the notation of \cite[Section 4]{CG18}) that $(\varpi,\beta)$ is a regular sequence for the ring $\Rt^\dagger$, and that $\Rt^\dagger/(\varpi,\beta)$ is isomorphic to the ring $\Sbar$ considered in Section \ref{sec:toric} of this paper, which proves their multiplicity $2$ result without needing any computations beyond the ones already carried out in this paper (and also gives an analogue of Theorem \ref{R=T}). Since our method does not rely on the work of \cite{Gross}, it would automatically give a generalization of Calegari and Geraghty's result to Shimura curves over $\Q$. It's also likely that this could be extended to totally real number fields by analyzing the appropriate local deformation rings considered in \cite{Snowden}.

\subsection{Explanations for the conditions in Theorem \ref{Mult 2^k}}

Many of the conditions in the statement of Theorem \ref{Mult 2^k} were included primarily to simplify the proof and exposition, and are not fundamental limitations on our method.
 
Condition (\ref{finite flat}) is essentially an assumption that the minimal level of $\rhobar$ is prime to $\ell$. It, together with the earlier assumption that $\ell$ does not ramify in $F$, is included to ensure that the local deformation rings $R_v^{\square,\fl,\psi}(\rhobar|_{G_v})$ considered in Section \ref{sec:deformation} are formally smooth. As the local deformation rings at $v|\ell$ are known to be formally smooth in more general situations, this condition can likely be relaxed somewhat with only minimal modifications to our method. Even more generally, it is likely that our techniques can be extended to certain other situations in which the local deformation rings at $v|\ell$ are \emph{not} formally smooth, provided we can still explicitly compute these rings.

Condition (\ref{vexing}) rules out the so-called `vexing' primes. It is mainly for convenience, to allow us to phrase our argument in terms of a `minimal level' $K^{\min}$ for $\rhobar$ and avoid a discussion of types. The treatment of minimally ramified deformation conditions in \cite{CHT} should allow this restriction to be removed without much added difficulty.
	
Condition (\ref{-1modl}) ensures that the Steinberg deformation ring, $R^{\square,\st,\psi}(\rhobar|_{G_v})$ from Section \ref{sec:deformation} is a domain. In the case when $R^{\square,\st,\psi}(\rhobar|_{G_v})$ fails to be a domain, Thorne \cite{ThorneDih} has defined new local deformation conditions which pick out the individual components. It's likely this could allow this condition to be removed as well, with only small modifications to our arguments.

The restriction to the minimal level is similarly intended to ensure that the deformation rings considered will be domains. It is possible this restriction can be relaxed in certain cases, particularly in cases when Ihara's Lemma is known.

We intend to explore the possibility of relaxing or removing some of these conditions in future work.
	
Lastly, condition (\ref{TW conditions})
is the classical ``Taylor--Wiles condition''\footnote{Experts will note that there is also another Taylor--Wiles condition one must assume in the case when $\ell=5$ and $\sqrt{5}\in F$. In our case however, this situation is already ruled out by the assumption that $\ell$ is unramified in $F$, and so we do not need to explicitly rule it out.}, which is a technical condition necessary for our construction in Section \ref{sec:patching}. It is unlikely that this condition can be removed without a significant breakthrough.

\subsection{Definitions and Notation}\label{ssec:notation}

Let $F$ be a totally real number field, with ring of integers $\O_F$. We will always use $v$ to denote a finite place $v\subseteq \O_F$. For any such $v$, let $F_v$ be the completion of $F$ and let $\O_{F,v}$ be its ring of integers. Let $\varpi_v$ be a uniformizer in $\O_{F,\p}$ and let $\k_v = \O_{F,v}/\varpi_v = \O_F/v$ be the residue field. Let $\Nm(v) = \#\k_v$ be the \emph{norm} of $v$.

Let $D$ be a quaternion algebra over $F$ with discriminant $\D$ (i.e. $\D$ is the product of all finite primes of $F$ at which $D$ is ramified). Assume that $D$ is either ramified at all infinite places of $F$ (the \emph{totally definite} case), or split at exactly one infinite place (the \emph{indefinite} case). Fix a maximal order of $D$, so that we may regard $D^\times$ as an algebraic group defined over $\O_F$.

Now fix a prime $\ell> 2$ which is relatively prime to $\D$ and does not ramify in $F$. For the rest of this paper we will fix a finite extension $E/\Q_\ell$. Let $\O$ be the ring of integers of $E$, $\lambda\in \O$ be a uniformizer and $\F = \O/\lambda$ be its residue field.

For any $\lambda$-torsion free $\O$-module $M$, we will write $M^\dual = \Hom_\O(M,\O)$ for its dual.

We define a \emph{level} to be a compact open subgroup 
\[K = \prod_{v\subseteq \O_F} K_v \subseteq \prod_{v\subseteq \O_F} D^\times(\O_{F,v})\subseteq D^\times(\A_{F,f})\]
where we have $K_v = D^\times(\O_{F,v})$ for each $v|\D$. We say that $K$ is \emph{unramified} at some $v\nmid \D$ if $K_v = \GL_2(\O_{F,v})$. Note that $K$ is necessarily unramified at all but finitely many $v$. Write $N_K$ for the product of all places $v\nmid \D$ where $K$ is ramified.

If $D$ is totally definite, let
\[S^D(K) = \left\{f: D^\times(F)\backslash D^\times(\A_{F,f})/K \to \O\right\}.\]
If $D$ is indefinite, let $X^D(K)$ be the Riemann surface $D^\times(F)\backslash \left(D^\times(\A_{F,f})\times \half\right)/K$ (where $\half$ is the complex upper half plane). Give $X^D(K)$ its canonical structure as an algebraic curve over $F$, 
and let $S^D(K) = H^1(X^D(K),\Z)\otimes_\Z\O$. Also define $\overline{S^D}(K) = S^D(K)\otimes_\O\F$.

For any finite prime ideal $v$ of $F$ for which $v\nmid\D N_K$, consider the double-coset operators $T_v,S_v:S^D(K)\to S^D(K)$ given by
\begin{align*}
T_v &= \left[K\begin{pmatrix}\varpi_v&0\\0&1\end{pmatrix}K\right],&
S_v &= \left[K\begin{pmatrix}\varpi_v&0\\0&\varpi_v\end{pmatrix}K\right].
\end{align*}

Let
\begin{align*}
\T^D(K) &= \O\bigg[ T_v,S_v, S_v^{-1}\bigg|v\subseteq \O_F, v\nmid\D N_K\bigg] \subseteq \End_{\O}(S^D(K))
\end{align*}
be the (anemic) Hecke algebra.

It will sometimes be useful to treat the $\T^D(K)$'s as quotients of a fixed ring $\T^{\univ}_S=\O\left[T_v,S_v^{\pm1}\right]_{v\not\in S}$, where $S$ is a finite set of primes, containing all primes dividing $\D N_K$ (here, $T_v$ and $S_v$ are treated as commuting indeterminants).  We can thus think of any maximal ideal $m\subseteq \T^D(K)$ as being a maximal ideal of $\T^{\univ}_S$, and hence as being a maximal ideal of $\T^D(K')$ for all $K'\subseteq K$.

Now let $G_F = \Gal(\Qbar/F)$ be the absolute Galois group of $F$. For any $v$, let $G_v = \Gal(\overline{F_v}/F_v)$ be the absolute Galois group of $F_v$, and let $I_v\unlhd G_v$ be the inertia group. Fix embeddings $\Qbar\injection \overline{F_v}$ for all $v$, and hence embeddings $G_v\injection G$. Let $\Frob_v\in G_v$ be a lift of (arithmetic) Frobenius.

Let $\e_\ell:G_F\to \O^\times$ be the cyclotomic character (given by $\sigma(\zeta) = \zeta^{\e_\ell(\sigma)}$ for any $\sigma\in G_F$ and $\zeta\in \mu_{\ell^\infty}$), and let $\ebar_\ell:G_F\to \F^\times$ be its mod $\ell$ reduction.

Now take a maximal ideal $m\subseteq \T^D(K)$, and note that $\T^D(K)/m$ is a finite extension of $\F$.

It is well known (see \cite{Carayol1}) that the ideal $m$ corresponds to a two-dimensional semisimple Galois representation $\rhobar_m:G_F\to \GL_2(\T^D(K)/m)\subseteq \GL_2(\Fbar_\ell)$ satisfying:

\begin{enumerate}
	\item $\rhobar_m$ is \emph{odd}.
	\item If $v\nmid \D,\ell,N_K$, then $\rhobar_m$ is unramified at $v$ and we have
	\begin{align*}
	\tr(\rhobar_m(\Frob_v)) &\equiv T_v\pmod{m}\\
	\det(\rhobar_m(\Frob_v)) &\equiv \Nm(v)S_v\pmod{m}.
	\end{align*}
	\item If $v|\ell$ and $v\nmid \D,N_K$, then $\rhobar_m$ is finite flat at $v$.
	\item If $v|\D$ then 
	\[\rhobar_m|_{G_v} \sim 
	\begin{pmatrix}
	\chibar\,\ebar_\ell&*\\
	0 & \chibar
	\end{pmatrix}.\]
	where $\chibar:G_v\to \Fbar_\ell^\times$ is an unramified character.
\end{enumerate}

We say that $m$ is \emph{non-Eisenstein} if $\rhobar_m$ is absolutely irreducible.

In keeping with property (4) above, we will say that a local representation $r:G_v\to \GL_2(\Fbar)$ (resp. $r:G_v\to \GL_2(\Ebar)$) is \emph{Steinberg} if it can be written (in some basis) as
\begin{align*}
	& r =	\begin{pmatrix}	\chibar\,\ebar_\ell&*\\	0 & \chibar	\end{pmatrix}
	&
&\left(\text{resp. } r = \begin{pmatrix}\chi \e_\ell& *\\0&\chi\end{pmatrix}\right)
\end{align*}
for some unramified character $\chibar:G_v\to \Fbar$ (resp. $\chi:G_v\to \Ebar$). We say that a global representation $r:G_F\to \GL_2(\Fbar)$ (resp. $r:G_F\to \GL_2(\Ebar)$) is \emph{Steinberg at $v$} if $r|_{G_v}$ is Steinberg.

Now if $\psi:G_F\to \O^\times$ is a character for which $\psi\e_\ell^{-1}$ has finite image, define the \emph{fixed determinant} Hecke algebra $\T^D_{\psi}(K)$ to be quotient of $\T^D(K)$ on which $\Nm(v)S_v=\psi(\Frob_v)$ for all $v\nmid \D,\ell$ at which $K$ is unramified.

Note that by Chebotarev density, a maximal ideal $m\subseteq \T^D(K)$ is in the support of $\T^D_{\psi}(K)$ if and only if $\rhobar_m$ has a lift $\rho:G_F\to \GL_2(\O)$ which is modular of level $K$ with $\det\rho = \psi$ (which in particular implies that $\det\rhobar_m\equiv \psi\pmod{\lambda}$).

Now for any continuous \emph{absolutely irreducible} representation $\rhobar:G_F\to \GL_2(\Fbar_\ell)$, define:
\[\K^D(\rhobar) = \left\{K\subseteq D^\times(\A_{F,f})\middle| \rhobar\sim \rhobar_m \text{ for some } m\subseteq \T^D(K)\right\}\]
(that is, $\K^D(\rhobar)$ is the set of levels $K$ at which the representation $\rhobar$ can occur.)

From now on, fix an absolutely irreducible Galois representation $\rhobar:G_F\to \GL_2(\Fbar_\ell)$ for which $\K^D(\rhobar)\ne\es$ (i.e. $\rhobar$ is ``automorphic for $D$''). In particular, this implies that $\rhobar$ is \emph{odd}, and satisfies the numbered conditions in Section \ref{ssec:notation}. Also assume that $\rhobar|_{G_v}$ is finite flat at all primes $v|\ell$ of $F$.

We will say that a \emph{minimal level} of $\rhobar$ is an element of $\K^D(\rhobar)$ which is maximal under inclusion. The assumption that $\rhobar|_{G_v}$ is finite flat for all $v|\ell$ implies that we may pick a minimal level $K^{\min} = \prod_{v\subseteq\O_F}K^{\min}_v$ of $\rhobar$ for which $K^{\min}_v = D^\times(\O_v)$ for all $v|\ell$. From now on fix such a $K^{\min}$.

Given any level $\ds K = \prod_{v\subseteq\O_F}K_v\subseteq K^{\min}$, we say that $K$ is of \emph{minimal level} at some $v\subseteq \O_F$ if $K_v = K^{\min}_v$.

Now given $K\in \K^D(\rhobar)$ and $m\subseteq \T^D(K)$ for which $\rhobar \sim \rhobar_m$ we define the number:
\[\nu_{\rhobar}(K) = 
\begin{cases}
\dim_{\T^D(K)/m} \overline{S^D}(K)[m]& \text{ if } D \text{ is totally definite}\\
\frac{1}{2}\dim_{\T^D(K)/m} \overline{S^D}(K)[m]& \text{ if } D \text{ is indefinite}
\end{cases}\]
called the \emph{multiplicity} of $\rhobar$ at level $K$. This number is closely related to the mod $\ell$ local-global compatibility conjectures given in \cite{BDJ}. Note that $\nu_{\rhobar}(K)$ does not depend on the choice of coefficient ring $\O$.

Theorem \ref{Mult 2^k} is precisely the assertion that $\nu_{\rhobar}(K^{\min}) = 2^k$. 

\subsection{Endomorphisms of Hecke modules}

We close this section by stating another result of our work:

\begin{thm}\label{R=T}
Let $\rhobar$ satisfy the conditions of Theorem \ref{Mult 2^k}.	If $D$ is totally definite then trace map $S^D(K^{\min})_m\otimes_{\T^{D}(K^{\min})_m}S^D(K^{\min})_m\to \omega_{\T^{D}(K^{\min})_m}$, induced by the self-duality of $S^D(K^{\min})_m$ is surjective (where $\omega_{\T^{D}(K^{\min})_m}$ is the dualizing sheaf\footnote{Which exists as $\T^{D}(K^{\min})_m$ is a $\lambda$-torsion free local $\O$-module of Krull dimension 1, by definition.} of $\T^{D}(K^{\min})_m$), and moreover the natural map $\T^{D}(K^{\min})_m\to \End_{\T^{D}(K^{\min})_m}(S^D(K^{\min})_m)$ is an isomorphism. If $D$ is indefinite, then the natural map $\T^{D}(K^{\min})_m\to \End_{\T^{D}(K^{\min})_m[G_F]}(S^D(K^{\min})_m)$ is an isomorphism.
\end{thm}

As explained in \cite{EmTheta}, this statement has applications towards the integral Eichler basis problem, so can likely be used to strengthen the results of Emerton \cite{EmTheta}. Also the work of \cite{BKM} shows that this statement has an important application to the study of congruence ideals, where it can serve as something of a substitute for a multiplicity one statement.

\section*{Acknowledgments}

I would like to thank Matt Emerton for suggesting this problem, and for all of his advice. I would also like to specifically thank him for pointing out that Theorem \ref{R=T} followed from my work. I also thank Jack Shotton, Frank Calegari, Toby Gee and Florian Herzig for their comments on earlier drafts of this paper, and for many helpful discussions. I would also like to thank an anonymous reviewer for their helpful suggestions. 

\section{Galois Deformation Rings}\label{sec:deformation}

In this section we will define the various Galois deformation rings which we will consider in the rest of the paper, and review their relevant properties.

\subsection{Local Deformation Rings}\label{ssec:local def}
Fix a finite place $v$ of $F$ and a representation $\rbar:G_v\to \GL_2(\F)$.

Let $\Cc_\O$ (resp. $\Cc_\O^\wedge$) be the category of Artinian (resp. complete Noetherian) local $\O$-algebras with residue field $\F$. Consider the (framed) deformation functor $\Dc^\square(\rbar):\Cc_\O\to \Set$ defined by
\begin{align*}
\Dc^\square(\rbar)(A) 
&= \{r:G_v\to \GL_2(A), \text{ continuous lift of }\rbar \}\\
&= \left\{(M,r,e_1,e_2)\middle|\parbox{3.5in}{$M$ is a free rank 2 $A$-module with a basis $(e_1,e_2)$ and $r:G_v\to \Aut_A(M)$ such that the induced map $G\to \Aut_A(M)= \GL_2(A)\to \GL_2(\F)$ is $\rbar$}
\right\}_{/\sim}
\end{align*}
It is well-known that this functor is \emph{pro-representable} by some $R^\square(\rbar)\in \Cc^\wedge_\O$, in the sense that $\Dc^\square(\rbar) \equiv \Hom_\O(R^\square(\rbar),-)$. In particular, $\rbar$ admits a universal lift $r^\square:G_v\to \GL_2(R^\square(\rbar))$.

For any continuous homomorphism, $x:R^{\square}(\rbar)\to \Ebar$, we obtain a Galois representation $r_x:G_v\to \GL_2(\Ebar)$ lifting $\rbar$, from the composition $G_v\xrightarrow{r^\square} \GL_2(R^\square(\rbar))\xrightarrow{x} \GL_2(\Ebar)$.

Now for any character $\psi: G_v\to \O^\times$ with $\psi\equiv \det \rbar \pmod{\lambda}$ define $R^{\square,\psi}(\rbar)$ to be the quotient of $R^\square(\rbar)$ on which $\det r^{\square}(g) = \psi(g)$ for all $g\in G_v$. Equivalently, $R^{\square,\psi}(\rbar)$ is the ring pro-representing the functor of deformations of $\rbar$ with determinant $\psi$.

Given any two characters $\psi_1,\psi_2:G_v\to \O^\times$ with $\det \rbar\equiv \psi_1\equiv \psi_2 \pmod{\lambda}$ we have $\psi_1\psi_2^{-1} \equiv 1 \pmod{\lambda}$, and so (as $1+\lambda\O$ is pro-$\ell$ and $\ell\ne 2$) there is a unique $\chi:G_v\to \O^\times$ with $\psi_1 = \psi_2\chi^2$. But now the map $r\mapsto r\otimes \chi$ is an automorphism of the functor $D^\square(\rbar)$ which can be shown to induce a natural isomorphism $R^{\square,\psi_1}(\rbar)\cong R^{\square,\psi_2}(\rbar)$. Thus, up to isomorphism, the ring $R^{\square,\psi}(\rbar)$ does not depend on the choice of $\psi$.

We call $R^{\square}(\rbar)$ (respectively $R^{\square,\psi}(\rbar)$) the \emph{deformation ring} (respectively the \emph{fixed determinant deformation ring}) of $\rbar$.

In order to prove our main results, we will also need to consider various deformation rings with \emph{fixed type}. Instead of defining these in general, we will consider only the specific examples which will appear in our arguments.

If $v|\ell$ and $\rbar$ and $\psi$ are both \emph{flat}, define $R^{\square,\fl,\psi}(\rbar)$ to be the ring pro-representing the functor of (framed) flat deformations of $\rbar$ with determinant $\psi$. We will refrain from giving a precise definition of this, as it is not relevant to our discussion. We will refer the reader to \cite{Kisin}, \cite{FontLaff}, \cite{Ramakrishna} and \cite{CHT} for more details, and use only the following result from  \cite[Section 2.4]{CHT}:

\begin{prop}
	If $F_v/\Q_\ell$ is unramified, then $R^{\square,\fl,\psi}(\rbar)\cong \O[[X_1,\ldots,X_{3+[F_v:\Q_\ell]}]]$.
\end{prop}

Also if $v\nmid \ell$, we recall the notion of a ``minimally ramified'' deformation given in \cite[Definition 2.4.14]{CHT} (which we again refrain from precisely defining). If $\psi$ is a minimally ramified deformation of $\psibar$, we will let $R^{\square,\min,\psi}(\rbar)$ be the maximal reduced $\lambda$-torsion free quotient of $R^{\square,\psi}(\rbar)$ with the property that if $x:R^{\square,\min,\psi}(\rbar)\to \Ebar$ is a continuous homomorphism, then the corresponding lift $r_x:G_v\to \GL_2(\Ebar)$ of $\rbar$ is minimally ramified. Again, we will refrain from giving a more detailed description of this, and instead we will use only the following well-known result (cf \cite{Shotton,CHT}):
\begin{prop}\label{prop:R^min}
	$R^{\square,\min,\psi}(\rbar)\cong \O[[X_1,X_2,X_3]]$.
\end{prop}

\begin{rem}
It is a well-known result (see for instance \cite{DiamondVexing}) that if $\rhobar:G_F\to \GL_2(\Fbar_\ell)$ is absolutely irreducible and automorphic for $D$ and $v\nmid\ell,\D$ is a prime of $F$ satisfying condition (\ref{vexing}) of Theorem \ref{Mult 2^k} (i.e. either $\Nm(v)\not\equiv -1\pmod{\ell}$, $\rhobar|_{I_v}$ is irreducible or $\rhobar|_{G_v}$ is absolutely reducible) then an \emph{automorphic} lift $\rho:G_F\to \GL_2(\Ebar)$ is minimally ramified at $v$ if and only if $\rho$ is automorphic of level $K$ for some level $K$ which is minimal at $v$.

This was the primary reason for including condition (\ref{vexing}) in Theorem \ref{Mult 2^k}.
\end{rem}

Now assume that $v\nmid \ell$ and $\rbar$ is \emph{Steinberg} (in the sense of Section \ref{ssec:notation}). We define $R^{\square,\st}(\rbar)$ (called the \emph{Steinberg deformation ring}) to be the maximal reduced $\lambda$-torsion free quotient of $R^{\square}(\rbar)$ for which $r_x:G_v\to \GL_2(\Ebar)$ is Steinberg for every continuous homomorphism $x:R^{\square,\st}(\rbar)\to \Ebar$.

Similarly if $\psi:G_v\to \O^\times$ is an \emph{unramified} character with $\psi\equiv \det\rbar\pmod{\lambda}$ (by assumption, $\rbar$ is Steinberg, and hence $\det\rbar$ is unramified), we  define $R^{\square,\st,\psi}(\rbar)$ (called the \emph{fixed determinant Steinberg deformation ring}) to be the maximal reduced $\lambda$-torsion free quotient of $R^{\square,\psi}(\rbar)$ for which $r_x:G_v\to \GL_2(\Ebar)$ is Steinberg for every continuous homomorphism $x:R^{\square,\st,\psi}(\rbar)\to \Ebar$.

It follows from our definitions that $R^{\square,\st,\psi}(\rbar)$ is the maximal reduced $\lambda$-torsion free quotient of $R^{\square,\st}(\rbar)$ on which $\det\rho^{\square}(g) = \psi(g)$ for all $g\in G_v$.

\subsection{Global Deformation Rings}\label{ssec:global def}

Now take a representation $\rhobar:G_F\to \GL_2(\F)$ satisfying:
\begin{enumerate}
	\item $\rhobar$ is absolutely irreducible.
	\item $\rhobar$ is odd.
	\item For each $v|\ell$, $\rhobar|_{G_v}$ is finite flat.
	\item For each $v|\D$, $\rhobar$ is Steinberg at $v$.
	\item $\K^D(\rhobar)\ne \es$.
\end{enumerate}
Let $\Sigma_\ell^D$ be a set of finite places of $F$ containing:
\begin{itemize}
	\item All places $v$ at which $\rhobar$ is ramified
	\item All places $v|\D$ (i.e. places at which $D$ is ramified)
	\item All places $v|\ell$
\end{itemize}
(we allow $\Sigma_\ell^D$ to contain some other places in addition to these), and let $\Sigma\subseteq \Sigma_\ell^D$ consist of those $v\in \Sigma_\ell^D$ with $v\nmid \ell,\D$.

Now as in \cite{Kisin} define $R^{\square}_{F,S}(\rhobar)$ (where $\Sigma_\ell^D\subseteq S$) to be the $\O$-algebra pro-representing the functor $\Dc^{\square}_{F,S}(\rhobar):\Cc_\O\to \Set$ which sends $A$ to the set of tuples $\big(\rho:G_{F,S}\to \End_A(M), \{(e^v_1,e^v_2)\}_{v\in\Sigma_\ell^D} \big)$, where $M$ is a free rank $2$ $A$-module with an identification $M/m_A = \F^2$ sending $\rho$ to $\rhobar$, and for each $v\in \Sigma_\ell^D$, $(e^v_1,e^v_2)$ is a basis for $M$, lifting the standard basis for $M/m_A= \F^2$, up to equivalence.

Also define the \emph{unframed} deformation ring $R_{F,S}(\rhobar)$ to be the $\O$-algebra pro-representing the functor $\Dc_{F,S}(\rhobar):\Cc_\O\to \Set$ which sends $A$ to the set of free rank $2$ $A$ modules $M$ with action $\rho:G_{F,S}\to \End_A(M)$ for which $\rho\equiv\rhobar\pmod{m_A}$, up to equivalence. This exists because $\rhobar$ is absolutely irreducible. We will let $\rho^{\univ}:G_F\to \GL_2(R_{F,S}(\rhobar))$ denote the universal lift of $\rhobar$.

Now take any character $\psi:G_F\to \O^\times$ for which:
\begin{enumerate}
	\item $\psi \equiv \det \rhobar \pmod\lambda$.
	\item $\psi$ is unramified at all places outside of $\Sigma_\ell^D$, \emph{and} all places dividing $\D$.
	\item $\psi$ is flat at all places dividing $\ell$.
	\item $\psi\e_\ell^{-1}$ has finite image.
\end{enumerate}
Note that as $\psi\e_\ell^{-1}$ has finite image, condition (3) is equivalent to the assertion that $\psi\e_\ell^{-1}$ is unramified at all places dividing $\ell$.

Let $\Dc^{\square,\psi}_{F,S}\subseteq \Dc^{\square}_{F,S}$ be the subfunctor of $\Dc^{\square}_{F,S}$ which sends $A$ to the set of tuples $\big(\rho:G_{F,S}\to \End_A(M), \{(e^v_1,e^v_2)\}_{v\in\Sigma_\ell^D} \big)$ (up to equivalence) in $\Dc^{\square}_{F,S}(A)$ for which $\det \rho = \psi$. Define $\Dc^{\psi}_{F,S}\subseteq \Dc_{F,S}$ similarly. Let $R^{\square,\psi}_{F,S}(\rhobar)$ and $R^{\psi}_{F,S}(\rhobar)$ to be the rings pro-representing $\Dc^{\square,\psi}_{F,S}$ and $\Dc^{\psi}_{F,S}$. Equivalently, these are the quotients of $R^{\square}_{F,S}(\rhobar)$ and $R_{F,S}(\rhobar)$, respectively, on which $\det\rho = \psi$.

Now note that the morphism of functors
\[
\left(\rho,\{(e^v_1,e^v_2)\}_{v\in \Sigma_\ell^D}\right) \mapsto \bigg(\rho|_{G_v}:G_v\to \End_A(M),(e^v_1,e^v_2)\bigg)_{v\in\Sigma_\ell^D}
\]
induces a map:
\[\pi:R_{\loc}=\hat{\bigotimes_{v\in \Sigma_\ell^D}}R^{\square}(\rhobar|_{G_v}) \to  R^{\square}_{F,S}.\]
Now consider the ring
\[R^{\square}_{\Sigma,\D,\ell} = \left[\hat{\bigotimes_{v|\ell}}R^{\square,\fl}(\rhobar|_{G_v})\right]\widehat{\otimes}
\left[
\hat{\bigotimes_{v\in\Sigma}}R^{\square,\min}(\rhobar|_{G_v})
\right]\widehat{\otimes}
\left[\hat{\bigotimes_{v|\D}}R^{\square,\st}(\rhobar|_{G_v})\right],\]
so that $R^{\square}_{\Sigma,\D,\ell}$ is a quotient of $R_{\loc}$. Using the map $\pi$ above, we may now define $R^{\square,D}_{F,S}(\rhobar)= R^{\square}_{F,S}\otimes_{R_{\loc}}R^{\square}_{\Sigma,\D,\ell}$. We may also define $R_{\loc}^\psi$, $R^{\square,\psi}_{\Sigma,\D,\ell}$ and $R^{\square,D,\psi}_{F,S}(\rhobar)$ analogously, by adding superscripts of $\psi$ to all of the rings used in the definitions.

Also note that the morphism of functors $\big(\rho, \{(e^v_1,e^v_2)\}_{v\in\Sigma_\ell^D} \big)\mapsto \rho$ induces a map $R^{\psi}_{F,S}(\rhobar) \to R^{\square,\psi}_{F,S}(\rhobar)$. As in \cite[(3.4.11)]{Kisin} this maps is formally smooth of dimension $j = 4|\Sigma_\ell^D|-1$, and so we may identify $R^{\square,\psi}_{F,S}(\rhobar) = R^\psi_{F,S}(\rhobar)[[w_1,\ldots,w_j]]$.

We can now define a map $R^{\square,\psi}_{F,S}(\rhobar)\to R^{\psi}_{F,S}(\rhobar)$ by sending each $w_i$ to $0$. Using this map, we may now define a unframed version of $R^{\square,D,\psi}_{F,S}(\rhobar)$ via 
\[R^{D,\psi}_{F,S}(\rhobar)= R^{\square,D,\psi}_{F,S}(\rhobar)\otimes_{R^{\square,\psi}_{F,S}(\rhobar)} R^{\psi}_{F,S}(\rhobar) = R^{\square,D,\psi}_{F,S}(\rhobar)/(w_1,\ldots,w_j).\footnote{
This definition may seem somewhat strange and ad-hoc. We have defined it this way because it will be convenient for our argument to view $R^{D,\psi}_{F,S}(\rhobar)$ as a quotient of $R^{\square,D,\psi}_{F,S}(\rhobar)$. Fortunately we will see later (in the remark following Lemma \ref{gen free}) that $R^{D,\psi}_{F,S}(\rhobar)$ is actually a fairly nicely behaved ring in the cases relevant to us, and so this definition will likely coincide with most ``natural'' ones.}\]
It follows from these definitions that the maps $x:R^\psi_{F,S}(\rhobar)\to \Ebar$ that factor through $R^{D,\psi}_{F,S}(\rhobar)$ are precisely those for which the induced representation $\rho_x:G_F\to \GL_2(R^{\psi}_{F,S}(\rhobar))\to \GL_2(\Ebar)$ satisfies:
\begin{itemize}
	\item $\rhobar_x|_{G_v}$ is flat at all $v|\ell$
	\item $\rhobar_x|_{G_v}$ is Steinberg at all $v|\D$
	\item $\rhobar_x|_{G_v}$ is minimally ramified at all $v\in\Sigma$.
\end{itemize}
(This is simply because the definition of $R^{D,\psi}_{F,S}(\rhobar)$ is such that any map $x:R^\psi_{F,S}(\rhobar)\to \Ebar$ factors through $R^{D,\psi}_{F,S}(\rhobar)$ if and only if the corresponding map $x:R^{\square,\psi}\surjection R^\psi_{F,S}(\rhobar)\to \Ebar$ factors through $R^{\square,D,\psi}_{F,S}(\rhobar)$.)

In order to prove Theorem \ref{R=T} we will need slightly more refined information about the relationship between $R_{F,S}(\rhobar)$ and $R^{\psi}_{F,S}(\rhobar)$.

Let $\psibar = \det\rhobar:G_{F}\to \F^\times$ be the reduction of $\psi$. Let $\Dc_{F,S}(\psibar):\Cc_\O\to\Set$ be the functor which sends $A\in\Cc_\O$ to the set of maps $\chi:G_{F,S}\to A^\times$ satisfying $\chi\equiv \psibar\pmod{m_A}$, up to equivalence. Let $R_{F,S}(\psibar)$ be the ring pro-representing $\Dc_{F,S}(\psibar)$.

Now for any $A\in\Cc_\O$, $A$ is a finite ring of $\ell$-power order, and so $m_A\subseteq A$ also has $\ell$-power order. 
It follows that $(1+m_A,\times)$ is an abelian multiplicative group of $\ell$-power order. In particular, as $\ell$ is odd, the map $x\mapsto x^2$ is an automorphism of $(1+m_A,\times)$, and hence it has an inverse $\sqrt{\cdot}: (1+m_A,\times)\to (1+m_A,\times)$. It is easy to see that $x\mapsto x^2$, and hence $x\mapsto \sqrt{x}$, commutes with morphisms in $\Cc_\O$, and is thus an automorphism of the functor $A\mapsto (1+m_A,\times)$ from $\Cc_\O$ to $\Ab$.

Now consider any $\rho:G_{F,S}\to \End(M)$ in $\Dc_{F,S}(\rhobar)(A)$. By definition we have $\det \rho\equiv \det\rhobar \equiv \psi \pmod{m_A}$, and so $(\det \rho)^{-1}\psi\equiv 1\pmod{m_A}$. That is, the image of $(\det \rho)^{-1}\psi:G_{F,S}\to A^\times$ lands in $1+m_A$. By the above work, it follows that there is a unique character $\sqrt{(\det \rho)^{-1}\psi}:G_{F,S}\to 1+m_A\subseteq A^\times$ with $(\sqrt{(\det \rho)^{-1}\psi})^2 = (\det \rho)^{-1}\psi$. Thus we may define a representation $\rho^{\psi}:G_{F,S}\to \End(M)$ by $\rho^{\psi}= (\sqrt{(\det \rho)^{-1}\psi})\rho$. Notice that $\rho^\psi\in\Dc_{F,S}(\rhobar)(A)$ and we have $\det\rho^\psi = \psi$, so that $\rho^\psi\in\Dc^\psi_{F,S}(\rhobar)(A)$. It is easy to see that the map $\rho\mapsto \rho^\psi$ is a natural transformation $\Dc_{F,S}\to \Dc^\psi_{F,S}$.

We now claim that the map $\Dc_{F,S}(\rhobar)\to \Dc_{F,S}(\psibar)\times \Dc^{\psi}_{F,S}(\rhobar)$ given by $\rho\mapsto (\det \rho,\rho^\psi)$ is an isomorphism of functors. Indeed, it has an inverse given by $(\chi,\rho)\mapsto\sqrt{\chi\psi^{-1}}\rho$. Looking at the rings these functors represent gives the following:

\begin{lemma}\label{lem:fixed determinant}
There is a natural isomorphism $R_{F,S}(\psibar)\widehat{\otimes}_\O R^\psi_{F,S}(\rhobar)\xrightarrow{\sim} R_{F,S}(\rhobar)$ of $\O$-algebras, induced by the natural transformation $\rho\mapsto (\det\rho,\rho^{\psi})$.
\end{lemma}

Lemma \ref{lem:fixed determinant} may be though of as giving a natural way of separating the determinant of a representation $\rho:G_{F,S}\to \GL_2(A)$ from the rest of the representation.

\subsection{Two Lemmas about Deformation Rings}

We finish this section by stating two standard results (cf. \cite{Kisin}) which will be essential for our discussion of Taylor--Wiles--Kisin patching in Section \ref{sec:patching}.

The first concerns the existence of an ``$R\to \T$'' map:

\begin{lemma}\label{R->T}
Assume that $\rhobar$ satisfies all of the numbered conditions listed in Section \ref{ssec:global def}. Take $K=\prod_vK_v\in \K^D(\rhobar)$ and let $S$ be a set of finite places of $F$ containing $\Sigma_\ell^D$ such that $K$ is unramified outside of $S$. Then there is a surjective map $R_{F,S}(\rhobar)\surjection \T^{D}(K)_m$, which induces a surjective map $R^{\psi}_{F,S}(\rhobar)\surjection \T^{D}_{\psi}(K)_m$ for any character $\psi:G_F\to \O^\times$ lifting $\det\rhobar$ for which $m$ is the support of $\T^{D}_{\psi}(K)$. If $K_v$ is maximal for all $v|\D$ and all $v|\ell$ and $K_v = K_v^{\min}$ for all $v\in \Sigma$, then this map factors through $R^{\psi}_{F,S}(\rhobar)\surjection R^{D,\psi}_{F,S}(\rhobar)$.
\end{lemma}

Note that (as mentioned in the remark following Proposition \ref{prop:R^min}) condition (\ref{vexing}) of Theorem \ref{Mult 2^k} is needed to ensure that $R_{F,S}(\rhobar)\surjection \T^{D}(K)_m$ does indeed factor through $R^{\psi}_{F,S}(\rhobar)\surjection R^{D,\psi}_{F,S}(\rhobar)$.

The second concerns the existence of ``Taylor--Wiles'' primes:

\begin{lemma}\label{TW primes}
Assume that $\rhobar$ satisfies all of the numbered conditions listed in Section \ref{ssec:global def} \emph{and} condition (\ref{TW conditions}) of Theorem \ref{Mult 2^k}. Let $S$ be a set of finite places of $F$ containing $\Sigma_\ell^D$, such for any prime $v\in S\sm \Sigma^D_\ell$, $\Nm(v)\not\equiv1\pmod{\ell}$ and the ratio of the eigenvalues of $\rhobar(\Frob_v)$ is not equal to $\Nm(v)^{\pm 1}$ in $\Fbar_\ell^\times$.
	
Then there exist integers $r,g\ge 1$ such that for any $n\ge 1$, there is a finite set $Q_n$ of primes of $F$ for which:
\begin{itemize}
	\item $Q_n\cap S=\es$.
	\item $\# Q_n = r$.
	\item For any $v\in Q_n$, $\Nm(v)\equiv 1\pmod{\ell^n}$.
	\item For any $v\in Q_n$, $\rhobar(\Frob_v)$ has distinct eigenvalues.
	\item There is a surjection $R^{\square,\psi}_{\Sigma,\D,\ell}[[x_1,\ldots,x_g]]\surjection R^{\square,D,\psi}_{F,S\cup Q_n}(\rhobar)$.
\end{itemize}
Moreover, we have $\dim R^{\square,\psi}_{\Sigma,\D,\ell} = r+j-g+1$.
\end{lemma}

From now on we will write $R_\infty$ to denote $R^{\square,\psi}_{\Sigma,\D,\ell}[[x_1,\ldots,x_g]]$ so that $\dim R_\infty = r+j+1$. By the results of Section \ref{ssec:local def} we have

\[R_\infty= \left[\hat{\bigotimes_{v|\D}}R^{\square,\st,\psi}(\rhobar|_{G_v})\right][[x_1,\ldots,x_{g'}]] \]
for some integer $g'$. In Section \ref{sec:toric} below, we will use the results of \cite{Shotton} to explicitly compute the ring $R_\infty$, and then use the theory of toric varieties to study modules over $R_\infty$. 

In Chapter \ref{sec:patching}, we will use Lemma \ref{R->T} and \ref{TW primes} to construct a particular module $M_\infty$ over $R_\infty$ out of a system of modules over the rings $\T^D(K)$, and then use the results of Chapter \ref{sec:toric} to deduce the structure of $M_\infty$. This will allow us to prove Theorems \ref{Mult 2^k} and \ref{R=T}.

\section{Class Groups of Local Deformation Rings}\label{sec:toric}

In our situation, all of the local deformation rings which will be relevant to us were computed in \cite{Shotton}. In this section, we will use this description to explicitly describe the ring $R_\infty$, and to study its class group (or rather, the class group of a related ring).

We first introduce some notation which we will use for the rest of this paper. If $R$ is any Noetherian local ring, we will always use $m_R$ to denote its maximal ideal.

If $M$ is a (not necessarily free) finitely generated $R$-module, we will say that the \emph{rank} of $M$, denoted by $\rank_RM$ is the cardinality of its minimal generating set.

If $R$ is a domain we will write $K(R)$ for its fraction field. If $M$ is a finitely generated $R$-module, then we we will say that the \emph{generic rank} of $M$, denoted $\genrank_RM$ is the $K(R)$-dimension of $M\otimes_RK(R)$ (that is, the rank of $M$ at the generic point of $R$).

From now on assume that $R$ is a (not necessarily local) normal Cohen--Macaulay \emph{domain} with a dualizing sheaf\footnote{Which will be the case for all Cohen--Macaulay rings we will consider.}, we will use $\omega_R$ to denote the dualizing sheaf of $R$.

For any finitely generated $R$-module $M$, we will let $M^* = \Hom_R(M,\omega_R)$. We say that $M$ is \emph{reflexive} if the natural map\footnote{As is it fairly easy to show that the dual of a finitely generated $R$-module is reflexive (cf \cite[\href{http://stacks.math.columbia.edu/tag/0AV2}{Tag 0AV2}]{stacks-project}) this definition is equivalent to simply requiring that there is \emph{some} isomorphism $M\xrightarrow{\sim}M^{**}$. In particular, if $M\cong M^*$ then $M$ is automatically reflexive.} $M\to M^{**}$ is an isomorphism.

We will let $\Cl(R)$ denote the \emph{Weil divisor class group} of $R$, which is isomorphic (cf \cite[\href{http://stacks.math.columbia.edu/tag/0EBM}{Tag 0EBM}]{stacks-project}) to the group of generic rank $1$ reflexive modules over $R$. For any generic rank $1$ reflexive sheaf $M$, let $[M]\in \Cl(R)$ denote the corresponding element of the class group. The group operation is then defined by $[M]+[N] = [(M{\otimes_R}N)^{**}]$.  Note that $[\omega_R]\in\Cl(R)$ and we have $[M^*] = [\omega_R]-[M]$ for any $[M]\in\Cl(R)$. 

Lastly, given any reflexive module $M$, the natural perfect pairing $M^*\times M\to\omega_R$ gives rise to a natural map $\tau_M:M^*\otimes_RM\to \omega_R$ (defined by $\tau_M(\varphi\otimes x) = \varphi(x)$) called the \emph{trace map}.

Also we will let 
\[k = \#\bigg\{v|\D \bigg|\rhobar \text{ is unramified at } v \text{ and, }
\rhobar(\Frob_v) \text{ is a scalar} \bigg\}.\]
as in Theorem \ref{Mult 2^k}.

Our main result of this section is the following:

\begin{thm}\label{self-dual}
	If $M_\infty$ is a finitely-generated module over $R_\infty$ satisfying:
	\begin{enumerate}
		\item $M_\infty$ is maximal Cohen--Macaulay over $R_\infty$;
		\item we have $M_\infty^*\cong M_\infty$ (and hence $M_\infty$ is reflexive);
		\item $\genrank_{R_\infty}M_\infty = 1$;
	\end{enumerate}
	then $\dim_\F M_\infty/m_{R_\infty} = 2^k$. Moreover, the trace map $\tau_{M_\infty}:M_\infty\otimes_{R_\infty}M_\infty\to \omega_{R_{\infty}}$ is surjective.
\end{thm}

Thus, to prove Theorem \ref{Mult 2^k}, it will suffice to construct a module $M_\infty$ over $R_\infty$ satisfying the conditions of Theorem \ref{self-dual} with $\dim_\F M_\infty/m_{R_\infty} = \nu_{\rhobar}(K^{\min})$. The last statement, that $\tau_{M_\infty}$ is surjective, will be used to prove Theorem \ref{R=T} (see the end of Section \ref{sec:patching}).

Our primary strategy for proving Theorem \ref{self-dual} is to note that conditions (1) and (3) imply that $M_\infty$ is the module corresponding to a Weil divisor on $R_\infty$, and condition (2) implies that we have $2[M_\infty] = [\omega_{R_\infty}]$ in $\Cl(R_\infty)$. Provided that $\Cl(R_\infty)$ is $2$-torsion free, this means that conditions (1), (2) and (3) uniquely characterize the module $M_\infty$. Proving the theorem would thus simply be a matter of computing the unique module $M_\infty$ satisfying the conditions of the theorem explicitly enough. 

Unfortunately, while we can give a precise description of the ring $R_\infty$ in our situation, it is difficult to directly compute $\Cl(R_\infty)$ from that description. Instead, we will first reduce the statement of Theorem \ref{self-dual} to a similar statement over the ring $\Rbar=R_\infty/\lambda$, and then to a statement over a finitely generated graded $\F$-algebra $\Rc$ with $\widehat{\Rc}\cong \Rbar$ (see Theorems \ref{self-dual mod l} and \ref{self-dual decompleted} below). We will then be able to directly compute $\Cl(\Rc)$, and the unique module $\Mc$ with $2[\Mc] = [\omega_{\Rc}]$ in $\Cl(\Rc)$, by using the theory of toric varieties.

In Section \ref{ssec:R_infty} we summarize the computations in \cite{Shotton} to explicitly describe the rings $R_\infty$ and $\Rbar$, and reduce Theorem \ref{self-dual} to the corresponding statement over $\Rbar$ (Theorem \ref{self-dual mod l}). In Section \ref{ssec:toric varieties} we introduce the ring $\Rc$, and show that it is the coordinate ring of an affine toric variety. Finally in Section \ref{ssec:class group} we use the theory of toric varieties to compute $\Cl(\Rc)$, which allows us to prove a ``de-completed'' mod $\lambda$ version of Theorem \ref{self-dual}. In Section \ref{ssec:decompletion} we adapt the method of Danilov \cite{Danilov} for computing the class groups of completions of graded rings to show that $\Cl(\Rc)\cong \Cl(\Rbar)$, from which we deduce Theorem \ref{self-dual mod l} and hence Theorem \ref{self-dual}.

\subsection{Explicit Calculations of Local Deformation Rings}\label{ssec:R_infty}

In order to prove Theorem \ref{self-dual}, it will be necessary to first compute the ring $R_\infty$, or equivalently to compute $R^{\square,\st,\psi}(\rhobar|_{G_v})$ for all $v|\D$. 

These computations were essentially done by Shotton \cite{Shotton}, except that he considers the non fixed determinant version, $R^{\st,\square}(\rhobar|_{G_v})$ instead of $R^{\st,\square,\psi}(\rhobar|_{G_v})$. Fortunately, it is fairly straightforward to recover $R^{\st,\square,\psi}(\rhobar|_{G_v})$ from $R^{\st,\square}(\rhobar|_{G_v})$. Specifically, we get:

\begin{thm}\label{R^st}
	Take any place $v|\D$. Recall that we have assumed that $\Nm(v)\not\equiv -1\pmod{\ell}$. If the residual representation $\rhobar|_{G_v}:G_v\to \GL_2(\F)$ is not scalar, then $R^{\square,\st,\psi}(\rhobar|_{G_v})\cong\O[[X_1,X_2,X_3]]$.
	
	If $\rhobar|_{G_v}:G_v\to \GL_2(\F)$ is scalar then \[R^{\square,\st,\psi}(\rhobar|_{G_v})\cong S_v= \O[[A,B,C,X,Y,Z]]/\I_v\] where $\I_v$ is the ideal generated by the $2\times 2$ minors of the matrix 
	\[\begin{pmatrix}
	A& B&X&Y\\
	C&A&Z&X+2\frac{\Nm(v)-1}{\Nm(v)+1}
	\end{pmatrix}.\]
	The ring $S_v$ is a Cohen--Macaulay and non-Gorenstein domain of relative dimension $3$ over $\O$. $(\lambda,C,Y,B-Z)$ is a regular sequence for $S_v$. Moreover, $S_v[1/\lambda]$ is formally smooth of dimension $3$ over $E$.
\end{thm}
\begin{proof}
For convenience, let $R_{\st} = R^{\st,\square}(\rhobar|_{G_v})$ and $R_{\st}^{\psi} =  R^{\st,\square,\psi}(\rhobar|_{G_v})$. By definition, $R_{\st}^{\psi}$ is the maximal reduced $\lambda$-torsion free quotient of $R_{\st}$ on which $\det\rho^{\square}(g) = \psi(g)$ for all $g\in G_v$.
	
Now let $I_v/\tilde{P}_v\cong \Z_\ell$ be the maximal pro-$\ell$ quotient of $I_v$, so that $\tilde{P}_v\unlhd G_v$ and $T_v=G_v/\tilde{P}_v\cong \Z_\ell\rtimes\widehat{\Z}$. Now let $\sigma,\phi\in T_v$ be topological generators for $\Z_\ell$ and $\widehat{\Z}$, respectively (chosen so that $\phi$ is a lift of arithmetic Frobenius, so that $\phi\sigma\phi^{-1} = \sigma^{\Nm(v)}$).
	
Now as in \cite{Shotton}, we may assume that the universal representation $\rho^\square:G_v\to\GL_2(R_{\st})$ factors through $T_v$. As we already have $\det\rho^\square(\sigma) = 1 = \psi(\sigma)$, it follows that $R_{\st}^\psi$ is the maximal reduced $\lambda$-torsion free quotient of $R$ on which $\det\rho^\square(\phi) = \psi(\phi)$.
	
As explained in Section \ref{sec:deformation}, up to isomorphism the ring $R^\psi_{\st}$ is unaffected by the choice of $\psi$, so it will suffices to prove the claim for a particular choice of $\psi$. Thus from now on we will assume that $\psi$ is unramified and $\ds\psi(\phi) = \frac{\Nm(v)}{(\Nm(v)+1)^2}t^2$ where
\[
t = 
\begin{cases}
\Nm(v)+1 & \Nm(v)\not\equiv \pm 1\pmod{\ell}\\
2 & \Nm(v)\equiv 1\pmod \ell
\end{cases}
\]
so that $t\equiv \Nm(v)+1\equiv \tr\rhobar(\phi)\pmod\ell$ (this particular choice of $t$ is made to agree with the computations of \cite{Shotton}).
	
But now by the definition of $R_{\st} = R^{\st,\square}(\rhobar|_{G_v})$ we have that $\Nm(v)\left(\tr\rho^\square(\phi)\right)^2 = (\Nm(v)+1)^2\det \rho^\square(\phi)$ and so
\[\det\rho^\square(\phi) = \frac{\Nm(v)}{(\Nm(v)+1)^2}\left(\tr \rho^\square(\phi)\right)^2\]
(where we have used the fact that $\Nm(v)\not\equiv -1\pmod \ell$, and so $\Nm(v)+1$ is a unit in $\O$).
	
It follows that
\begin{align*}
\det\rho^\square(\phi)-\psi(\phi) 
&= \frac{\Nm(v)}{(\Nm(v)+1)^2}\left(\tr \rho^\square(\phi)\right)-\frac{\Nm(v)}{(\Nm(v)+1)^2}t^2\\
&= \frac{\Nm(v)}{(\Nm(v)+1)^2}(\tr \rho^\square(\phi)+t)(\tr \rho^\square(\phi)-t).
\end{align*}
But now
\[\frac{\Nm(v)}{(\Nm(v)+1)^2}\left(\tr \rho^\square(\phi)+t\right) \equiv \frac{2\Nm(v)}{(\Nm(v)+1)^2}\tr\rhobar(\phi) \equiv \frac{2\Nm(v)}{\Nm(v)+1}\pmod{m_R}\]
and so as $\ell\nmid 2,\Nm(v),\Nm(v)+1$ we get that $\ds \frac{\Nm(v)}{(\Nm(v)+1)^2}\left(\tr \rho^\square(\phi)+t\right)$ is a unit in $R_{\st}$. It follows that $R_{\st}^\psi$ is the maximal reduced $\lambda$-torsion free quotient of
\[R_{\st}^{\psi,\circ} = \frac{R_{\st}}{\left(\det\rho^\square(\phi)-\psi(\phi) \right)} = \frac{R_{\st}}{\tr \rho^\square(\phi)-t}\]
It now follows immediately from Shotton's computations that in each case $R_{\st}^{\psi,\circ}$ is already reduced and $\lambda$-torsion free (and so $R_{\st}^\psi = R_{\st}^{\psi,\circ}$) and has the form described in the statement of Theorem \ref{R^st} above.
	
Indeed, first assume that $\Nm(v)\not\equiv \pm 1\pmod{\ell}$. By \cite[Proposition 5.5]{Shotton} we may write $R_{\st} = \O[[B,P,X,Y]]$ with
\begin{align*}
\rho^\square(\sigma) &=
\begin{pmatrix}1&X\\y&1\end{pmatrix}^{-1}
\begin{pmatrix}1&x+B\\0&1\end{pmatrix}
\begin{pmatrix}1&X\\y&1\end{pmatrix}\\
\rho^\square(\phi) &=
\begin{pmatrix}1&X\\y&1\end{pmatrix}^{-1}
\begin{pmatrix}\Nm(v)(1+P)&0\\0&1+P\end{pmatrix}
\begin{pmatrix}1&X\\y&1\end{pmatrix},
\end{align*}
for some $x\in \O$. Thus we have 
\[\tr \rho^\square(\phi) = (\Nm(v)+1)(1+P) = t+tP\]
and so (as $t=\Nm(v)+1\in\O$ is a unit), $R_{\st}^{\psi,\circ} = R_{\st}/(tP) \cong \O[[B,P,X,Y]]/(P)\cong \O[[B,X,Y]]$, as desired.
	
Now assume that $\Nm(v)\equiv 1\pmod{\ell}$. Again, following the computations of \cite[Proposition 5.8]{Shotton} we can write
\begin{align*}
\rho^\square(\sigma) &=
\begin{pmatrix}
1+A&x+B\\C&1-A
\end{pmatrix}\\
\rho^\square(\phi) &=
\begin{pmatrix}
1+P&y+R\\S&1+Q
\end{pmatrix}
\end{align*}
(for $x,y\in \O$) where $A,B,C,P,Q,R$ and $S$ topologically generate $R_{\st}$. Now following Shotton's notation, let $T = P+Q$, so that $\tr \rho^\square(\phi) = 2+T= t+T$ and thus $R_{\st}^{\psi,\circ} = R_{\st}/(T)$. In both cases ($\rhobar|_{G_v}$ non-scalar and scalar) Shotton's computations immediately give the desired description of $R_{\st}^\psi$.\footnote{In Shotton's notation, when $\rhobar|_{G_v}$ is scalar $R_{\st}^\psi$ would be cut out by the $2\times 2$ minors of the matrix $\begin{pmatrix}X_1&X_2&X_3&X_4\\Y_1&-X_1&Y_3&X_3+2\frac{\Nm(v)-1}{\Nm(v)+1}\end{pmatrix}$. This is equivalent to the form stated in Theorem \ref{R^st} via the variable substitutions $A = X_1$, $B = -X_2$, $C = Y_1$, $X = X_3$, $Y = X_4$ and $Z = Y_3$.}
	
Moreover, Shotton shows that $S_v$ is indeed Cohen--Macaulay and non Gorenstein of relative dimension $3$ over $\O$, and that $S_v[1/\lambda]$ is formally smooth of dimension $3$ over $E$ . As $S_v$ is Cohen--Macaulay, the claim that $(\lambda,C,Y,B-Z)$ is a regular sequence simply follows by noting that
\[S_v/(\lambda,C,Y,B-Z)\cong \F[[A,B,X]]/(A^2,AB,AX,B^2,BX,X^2) = \F\oplus\F A\oplus \F B\oplus \F X\]
is a zero dimensional ring, and so $(\lambda,C,Y,B-Z)$ is a system of parameters.
\end{proof}
Thus letting $\D_1|\D$ be the product of the places $v|\D$ at which $\rhobar|_{G_v}:G_v\to \GL_2(\F)$ is scalar, we have
\[R_\infty= \left[\hat{\bigotimes_{v|\D_1}}S_v\right][[x_1,\ldots,x_{s}]] \]
for some integer $s$. As the rings $S_v$ are all Cohen--Macaulay by Theorem \ref{R^st}, it follows that $R_\infty$ is as well.

Now note that the description of $S_v$ in Theorem \ref{R^st} becomes much simpler if we work in characteristic $\ell$. Indeed, if $\rhobar|_{G_v}$ is scalar then $\Nm(v)\equiv 1\pmod{\ell}$ and so, $\Sbar=S_v/\lambda$ is an explicit graded ring not depending on $v$. Specifically, we have $\Sbar = \F[[A,B,C,X,Y,Z]]/\Ibar$ where $\Ibar$ is the (homogeneous) ideal generated by the $2\times 2$ minors of the matrix 
$\begin{pmatrix}
A&B&X&Y\\
C&A&Z&X
\end{pmatrix}$.

It thus follows that
\[\Rbar = R_\infty/\lambda \cong \Sbar^{\hat{\otimes} k}[[x_1,\ldots,x_{s}]],\]
which will be much easier to work with than $R_\infty$. In particular, note that $\Rbar$ is still Cohen--Macaulay, as $R_\infty$ is $\lambda$-torsion free by definition.

It will thus be useful to reduce Theorem \ref{self-dual} the following ``mod $\lambda$'' version:

\begin{thm}\label{self-dual mod l}
	If $\Mbar$ is a finitely-generated module over $\Rbar$ satisfying:
	\begin{enumerate}
		\item $\Mbar$ is maximal Cohen--Macaulay over $\Rbar$
		\item We have $\Mbar^*\cong \Mbar$.
		\item $\genrank_{\Rbar}\Mbar=1$.
	\end{enumerate}
	then $\dim_\F \Mbar/m_{\Rbar} = 2^k$. Moreover, the trace map $\tau_{\Mbar}:\Mbar\otimes_{\Rbar}\Mbar\xrightarrow{\sim} \omega_{\Rbar}$ is surjective.
\end{thm}

\begin{proof}[Proof that Theorem \ref{self-dual mod l} implies \ref{self-dual}]
	Assume that Theorem \ref{self-dual mod l} holds, and that $M_\infty$ satisfies the hypotheses of Theorem \ref{self-dual}. As $R_\infty$ is flat over $\O$, it is $\lambda$-torsion free and thus $\lambda$ is not a zero divisor on $M_\infty$ (by condition (1)). It follows that $\Mbar= M_\infty/\lambda$ is maximal Cohen--Macaulay over $\Rbar$ and $\genrank_{\Rbar}\Mbar = \genrank_{R_\infty}M_\infty = 1$. (In general, if $R$ is Cohen--Macaulay and $M$ is maximal Cohen--Macaulay over $R$, then for any regular element $x\in R$, $\genrank_{R/x}(M/x) = \genrank_RM$, provided $R/x$ is also a domain.)
	
	Moreover, as $M_\infty$ is maximal Cohen--Macaulay over $R_\infty$, \cite[Proposition 21.12b]{Eisenbud} gives
	\[\Mbar = M_\infty/\lambda = \Hom_{R_\infty}(M_\infty,\omega_{R_\infty})/\lambda = \Hom_{R_\infty/\lambda}(M_\infty/\lambda,\omega_{R_\infty}/\lambda) = \Hom_{\Rbar}(\Mbar,\omega_{\Rbar}),\]
	where we have used the fact that $\omega_{R_\infty}/\lambda\cong \omega_{\Rbar}$, by \cite[Chapter 21.3]{Eisenbud}. Thus $\Mbar$ is self-dual. Thus $\Mbar$ satisfies all of the hypotheses of Theorem \ref{self-dual mod l}, and so $\dim_\F \Mbar/m_{\Rbar} = 2^k$ and $\tau_{\Mbar}$ is surjective.
	
	Now we obviously have that $\Mbar/m_{\Rbar}\cong M_\infty/m_{R_\infty}$, so the first conclusion of Theorem \ref{self-dual} follows.
	
	Also, the trace map $\tau_{\Mbar}:\Mbar\otimes_{\Rbar}\Mbar\to \omega_{\Rbar}$ is just the mod-$\lambda$ reduction of the map $\tau_{M_\infty}:M_\infty\otimes M_\infty\to \omega_{R_\infty}$, so it follows that $\tau_{\Mbar}$ is surjective if and only if $\tau_{M_\infty}$ is. Thus the second conclusion of Theorem \ref{self-dual} follows.
\end{proof}

As hinted above, we will prove Theorem \ref{self-dual mod l} by computing the class group of $\Rbar$.

We finish this section by proving the following lemma, which will make the second conclusion of Theorem \ref{self-dual mod l} easier to prove (and will also be useful in the proof of Theorem \ref{R=T}):

\begin{lemma}\label{trace}
	If $R$ is a Cohen--Macaulay ring with a dualizing sheaf $\omega_R$, and $M$ is a reflexive $R$-module, then the trace map $\tau_M:M^*\otimes_RM\to \omega_R$ is surjective if and only if there \emph{exists} an $R$-module surjection $M^*\otimes_RM\surjection \omega_R$.
\end{lemma}
\begin{proof}
	Assume that $f:M^*\otimes_RM\to \omega_R$ is a surjection. Take any $\alpha\in\omega_R$. Then we can write
	\[\alpha = f\left(\sum_{i\in I}b_i\otimes c_i\right) = \sum_{i\in I}f(b_i\otimes c_i)\]
	for some finite index set $I$ and some $b_i\in M^*$ and $c_i\in M$. For each $i\in I$, consider the $R$-linear map $\varphi_i:M\to \omega_R$ defined by $\varphi_i(c) = f(b_i\otimes c)$. Then we have $\varphi_i\in M^*$ for all $i$ and so
	\begin{align*}
	\alpha &= \sum_{i\in I}f(b_i\otimes c_i) = \sum_{i\in I}\varphi_i(c_i) = \sum_{i\in I}\tau_M(\varphi_i\otimes c_i) = \tau_M\left(\sum_{i\in I}\varphi_i\otimes c_i \right).
	\end{align*}
	Thus $\tau_M$ is surjective.
\end{proof}

\subsection{Toric Varieties}\label{ssec:toric varieties}

For the remainder of this section we will consider the rings $\Sc = \F[A,B,C,X,Y,Z]/\Ibar$, where again $\Ibar$ is the ideal generated by the $2\times 2$ minors of the matrix $\begin{pmatrix}
A&B&X&Y\\
C&A&Z&X
\end{pmatrix}$, and $\Rc = \Sc^{\otimes k}[x_1,\ldots,x_s]$. Note that $\Sc$ and $\Rc$ are naturally finitely generated graded $\F$-algebras. Let $m_{\Sc}$ and $m_{\Rc}$ denote their irrelevant ideals, and note $\Sbar$ and $\Rbar$ are the completions of $\Sc$ and $\Rc$ at these ideals.

The goal of this subsection and the next one is to prove the following ``de-completed'' version of Theorem \ref{self-dual mod l}. In Section \ref{ssec:decompletion} we will show that this implies Theorem \ref{self-dual mod l}, and hence Theorem \ref{self-dual}.

\begin{thm}\label{self-dual decompleted}
	If $\Mc$ is a finitely-generated module over $\Rc$ satisfying:
	\begin{enumerate}
		\item $\Mc$ is maximal Cohen--Macaulay over $\Rc$
		\item We have $\Mc^*\cong \Mc$.
		\item $\genrank_{\Rc}\Mc=1$.
	\end{enumerate}
	then $\dim_\F \Mc/m_{\Rc} = 2^k$. Moreover, the trace map $\tau_{\Mc}:\Mc\otimes_{\Rc}\Mc\xrightarrow{\sim} \omega_{\Rc}$ is surjective.
\end{thm}

As outlined above, we will prove this theorem by computing $\Cl(\Rc)$. The key insight that allows us to preform this computation is that $\Sc$, and hence $\Rc$, is the coordinate ring of an affine \emph{toric} variety.

In this section, we review the basic theory of toric varieties and show that $\Sc$ and $\Rc$ indeed correspond to toric varieties. We shall primarily follow the presentation of toric varieties from \cite{CLS}. Unfortunately \cite{CLS} works exclusively with toric varieties over $\C$, whereas we are working in positive characteristic. All of the results we will rely on work over arbitrary base field, usually with identical proofs, so we will freely cite the results of \cite{CLS} as if they were stated over arbitrary fields. We refer the reader to \cite{MS} and \cite{DanilovToric} for a discussion of toric varieties over arbitrary fields.

We recall the following definitions. For any integer $d\ge 1$, let $T_d = \G_m^d = (\F^\times)^d$, thought of as a group variety. Define the two lattices
\begin{align*}
M &= \Hom(T_d,\G_m) &
N &= \Hom(\G_m,T_d),
\end{align*}
called the \emph{character lattice} and the \emph{lattice of one-parameter subgroups}, respectively. Note that $M\cong N\cong \Z^d$. We shall write $M$ and $N$ additively. For $m\in M$ and $u\in N$, we will write $\chi^m:T_d\to \G_m$ and $\lambda^u:\G_m\to T_d$ to denote the corresponding morphisms. 

First note that there is a perfect pairing $\langle\ ,\ \rangle:M\times N\to \Z$ given by $t^{\langle m,u\rangle} = \chi^m(\lambda^u(t))$. We shall write $M_\R$ and $N_\R$ for $M\otimes_\Z\R$ and $N\otimes_\Z\R$, which are each $d$-dimensional real vector spaces. We will extend the pairing $\langle\ ,\ \rangle$ to a perfect pairing $\langle\ ,\ \rangle:M_\R\times N_\R\to \R$.

For the rest of this section, we will (arbitrarily) fix a choice of basis $e_1,\ldots,e_d$ for $M$, and so identify $M$ with $\Z^d$. We will also identify $N$ with $\Z^d$ via the dual basis to $e_1,\ldots,e_d$. Under these identifications, $\langle\ ,\ \rangle$ is simply the usual (Euclidean) inner product on $\Z^d$.

We can now define:

\begin{defn}\label{def:toric}
	An \emph{(affine) toric variety} of dimension $d$ is a pair $(X,\iota)$, where $X$ is an affine variety $X/\F$ of dimension $d$ and $\iota$ is an open embedding $\iota:T_d\injection X$ such that the natural action of $T_d$ on itself extends to a group variety action of $T_d$ on $X$. We will usually write $X$ instead of the pair $(X,\iota)$.
	
	For such an $X$, we define the \emph{semigroup} of $X$ to be
	\[\semi_X=\{m\in M|\chi^m:T_d\to \G_m \text{ extends to a morphism }X\to \A^1\}\subseteq M\]
	
	For convenience, we will also say that a finitely generated $\F$-algebra $R$ (together with an inclusion $R\injection \F[M]$) is \emph{toric} if $\Spec R$ is toric.
\end{defn}

The primary significance of  affine toric varieties is that they are classified by their semigroups. Specifically:

\begin{prop}\label{semigroup}
	If $X$ is an affine toric variety of dimension $d$, then $X = \Spec \F[\semi_X]$, and the embedding $\iota:T_d\injection X$ is induced by $\F[\semi_X]\injection \F[M]$ (using the fact that $T_d = \Spec\F[M]$). Moreover we have
	\begin{enumerate}
		\item The semigroup $\semi_X$ spans $M$ (that is, $\Z\semi_X$ has rank $d$).
		\item If $\semi_X$ is \emph{saturated} in $M$ (in the sense that $km\in\semi_X$ implies that $m\in \semi_X$ for all $k> 0$ and $m\in M$) then $X$ is a \emph{normal} variety.
	\end{enumerate}
	Conversely, if $\semi\subseteq M$ is a finitely generated semigroup spanning $M$ then the inclusion $\F[\semi]\injection \F[M]$ gives $\Spec \F[\semi]$ the structure of a $d$-dimensional affine toric variety.
\end{prop}
\begin{proof}
	cf. \cite{CLS} Proposition 1.1.14 and Theorems 1.1.17 and 1.3.5.
\end{proof}

If $R$ is a toric $\F$-algebra, we will write $\semi_R$ to mean $\semi_{\Spec R}$.

While it can be difficult to recognize toric varieties directly from Definition \ref{def:toric}, the following Proposition makes it fairly easy to identify toric varieties in $\A^s$.

\begin{prop}\label{toric ideal}
	Fix an integer $h\ge 1$ and let $\Phi:\Z^h\to M$ be any homomorphism with finite cokernel, and let $L = \ker \Phi$. Let $\semi\subseteq M$ be the semigroup generated by $\Phi(e_1),\ldots,\Phi(e_h)\in M$. Then we have an isomorphism $\F[z_1,\ldots,z_h]/I_L\cong\F[\semi]$ given by $z_i\mapsto \Phi(e_i)$, where
	\[I_L = \left(z^{\alpha}-z^{\beta}\middle|\alpha,\beta\in \Z_{\ge 0}^h \text{ such that } \alpha-\beta\in L\right)\subseteq \F[z_1,\ldots,z_h].\]
	(Where, for any $\alpha = (\alpha_1,\ldots,\alpha_h)\in \Z_{\ge 0}^h$, we write $z^\alpha = z_1^{\alpha_1}\cdots z_h^{\alpha_h}\in \F[z_1,\ldots,z_h]$.)
	
	Moreover $I_L$ can be explicitly computed as follows: Assume that $\L = (\ell^1,...,\ell^r)$ is a $\Z$-basis for $L$, with $\ell^i = (\ell^i_1,\ldots,\ell^i_h)\in \Z^h$. Write each $\ell^i$ as $\ell^i = \ell^i_+-\ell^i_i$ where 
	\begin{align*}
	\ell_+^i &= (\max\{\ell^i_1,0\},\ldots,\max\{\ell^i_h,0\})\in \Z_{\ge0}^h\\
	\ell_-^i &= (\max\{-\ell^i_1,0\},\ldots,\max\{-\ell^i_h,0\})\in \Z_{\ge0}^h.
	\end{align*}
	Then if $I_{\L} = \left(z^{\ell^1_+}-z^{\ell^1_-},\ldots,z^{\ell^h_+}-z^{\ell^h_-}\right)\subseteq \F[z_1,\ldots,z_h]$, $I_L$ is the \emph{saturation} of $I_{\L}$ with respect to $z_1\cdots z_h$, that is:
	\begin{align*}
	I_L &= \left(I_{\L}:(z_1\cdots z_h)^\infty \right) = \bigg\{f\in \F[z_1,\ldots,z_h]\bigg|(z_1\cdots z_h)^mf\in I_{\L}\text{ for some } m\ge 0\bigg\}\\
	&= \bigg(z^{\alpha}-z^{\beta}\bigg|\alpha,\beta\in \Z^h_{\ge0}\text{ such that }(z_1\cdots z_h)^m\left(z^{\alpha}-z^{\beta}\right)\in I_{\L}\text{ for some } m\ge 0\bigg).
	\end{align*}
	Conversely, if $I\subseteq \F[z_1,\ldots,z_h]$ is any \emph{prime} ideal which can be written in the form $I = \left(z^{\alpha_i}-z^{\beta_i}\middle| i\in \As \right)$ for a finite index set $\As$ and $\alpha_i,\beta_i\in\Z_{\ge 0}^h$, then $I = I_L$ for some $L$ and so $\F[z_1,\ldots,z_h]/I$ can be given the structure of a toric $\F$-algebra.
\end{prop}
\begin{proof}
	This mostly follows from \cite[Propositions 1.1.8, 1.1.9 and 1.1.11]{CLS}. The statement that $I_L = \left(I_{\L}:(z_1\cdots z_h)^\infty \right)$ is \cite[Exercise 1.1.3]{CLS} or \cite[Lemma 7.6]{MS}
\end{proof}

Now applying the above results to the $\F$-algebra $\Sc$, we get:

\begin{prop}\label{S toric}
	$\Sc$ may be given the structure of a $3$-dimensional toric $\F$-algebra, with semigroup 
	\[\semi_\Sc=\{(a,b,c)\in \Z^3|a,b,c\ge 0, 2a+2b\ge c\}\subseteq M\]
	under some choice of basis $e_1,e_2,e_3$ for $M$. Moreover:
	\begin{enumerate}
		\item $\Spec\Sc$ is the affine cone over a surface $\Vc\subseteq\P^5$ isomorphic to $\P^1\times \P^1$. The embedding $\P^1\times\P^1\xrightarrow{\sim}\Vc\subseteq \P^5$ corresponds to the very ample line bundle $\O_{\P^1}(2)\boxtimes\O_{\P^1}(1)$ on $\P^1\times\P^1$.
		\item $\Sc$ is isomorphic to the ring $\F[x,xz,xz^2,y,yz,yz^2]\subseteq \F[x,y,z]$.
		\item $\Spec \Sc$ is a \emph{normal} variety.
		\item $\Rc$ is toric of dimension $3k+s$.
	\end{enumerate}
\end{prop}
\begin{proof}
	Take $d=3$ in the above discussion, and fix an isomorphism $M\cong \Z^3$.
	
	Write $\semi = \{(a,b,c)\in \Z^3|a,b,c\ge 0, 2a+2b\ge c\}$. We will first show that $\Sc\cong \F[\semi]$. Note that $\semi$ is generated by the (transposes of) the columns of the matrix
	\[\Phi = \begin{pmatrix}
	1&1&1&0&0&0\\
	0&0&0&1&1&1\\
	1&2&0&1&2&0
	\end{pmatrix},\]
	which in particular gives an isomorphism $\F[\semi]\cong \F[x,xz,xz^2,y,yz,yz^2]\subseteq \F[x,y,z]$.
	
	Let $L = \ker \Phi$. By Proposition \ref{toric ideal} it follows that $\F[\semi] \cong \F[A,B,C,X,Y,Z]/I_L$ (where we have identified the ring $\F[z_1,z_2,z_3,z_4,z_5,z_6]$ with $\F[A,B,C,X,Y,Z]$ in the obvious way, in order to keep are notation consistent).
	
	But now note that $L$ is a rank $3$ lattice with basis $\L = (\ell^1,\ell^2,\ell^3)$ given by the vectors:
	\[\ell^1=
	\begin{pmatrix}
	1\\0\\-1\\-1\\0\\1
	\end{pmatrix},
	\ell^2=
	\begin{pmatrix}
	1\\0\\-1\\1\\-1\\0
	\end{pmatrix},
	\ell^3=
	\begin{pmatrix}
	1\\-1\\0\\1\\0\\-1
	\end{pmatrix}.\]
	It follows that
	\[I_{\L} = (AZ-CX,AX-CY,AX-BZ)\]
	whence it is straightforward to compute that
	\begin{align*}
	I_L &= \left(I_{\L}:(ABCXYZ)^\infty \right)\\
	&=\left(A^2-BC,AZ-CX,AX-CY,AX-BZ,AY-BX,X^2-YZ \right) = \Ibar.
	\end{align*}
	Thus $\Sc = \F[A,B,C,X,Y,Z]/\Ibar$ is indeed toric with $\semi_{\Sc} = \semi$, and we have 
	\[\Sc\cong \F[\semi_\Sc]\cong \F[x,xz,xz^2,y,yz,yz^2].\]
	Moreover, this isomorphism sends the ideal $(A,B,C,X,Y,Z)\subseteq \Sc$ to the ideal\\ $(x,xz,xz^2,y,yz,yz^2)$, from which (2) easily follows.
	
	Now the semigroup $\semi_\Sc = \{(a,b,c)\in \Z^3|a,b,c\ge 0,2a+2b\ge c\ge 0\}$ is clearly saturated in $M$, and so $\Spec \Sc$ is indeed normal, proving (3).

Now note that $\O_{\P^1}(2)\boxtimes\O_{\P^1}(1)$ is indeed a very ample line bundle on $\P^1\times\P^1$ and corresponds to the (injective) morphism $f:\P^1\times \P^1 \to \P^5$ defined by
\[f([s:t],[u:v]) = [s^2u:stu:t^2u:s^2v:stv:t^2v].\]
It thus follows that the coordinate ring on the cone over the image of $f$ is isomorphic to
\[\F[s^2u,stu,t^2u,s^2v,stv,t^2v]\cong\F[x,xz,xz^2,y,yz,yz^2],\]
proving (1).
	
	Lastly, recalling that $M = \Z^3$, let $M' = M = M^{\oplus k}\oplus \Z^s = \Z^{3k+s}$ and define:
	\[\semi' = (\semi)^{\otimes k}\otimes \Z_{\ge 0}^s\]
	so that
	\[\F[\semi'] = \F[\semi]^{\otimes k}\otimes \F[x]^{\otimes s} \cong \Sc^{\otimes k}[x_1,\ldots,x_s] = \Rc\]
	so that $\Rc$ is indeed toric of dimension $3k+s$, proving (4).
\end{proof}

We will now restrict our attention to \emph{normal} affine toric varieties. The advantage to doing this is that Proposition \ref{semigroup} has a refinement (see Proposition \ref{toric cone} below) that allows us to characterize normal toric varieties much more simply, using cones instead of semigroups.

We now make the following definitions:

\begin{defn}\label{def:cone}
	A \emph{convex rational polyhedral cone} in $N_\R$ is a set of the form:
	\[\sigma = \Cone(S) = \left\{\sum_{\lambda\in S}x_\lambda\lambda\middle|x_\lambda\ge 0 \text{ for all }\lambda\in S \right\}\subseteq N_\R\]
	for some finite subset $S\subseteq N$. 
	
	A \emph{face} of $\sigma$ is a subset $\tau\subseteq \sigma$ which can be written as $\tau = \sigma\cap H$ for some hyperplane $H\subseteq N_\R$ which does not intersect the interior of $\sigma$. We write $\tau\face\sigma$ to say that $\tau$ is a face of $\sigma$. It is clear that any face of $\sigma$ is also a convex rational polyhedral cone. We say that $\sigma$ is \emph{strongly convex} if $\{0\}$ is a face of $\sigma$.
	
	We write $\R\sigma$ for the subspace of $N_\R$ spanned by $\sigma$, and we will let the dimension of $\sigma$ be $\dim \sigma = \dim_\R \R\sigma$.
	
	We make analogous definitions for cones in $M_\R$.
	
	For a convex rational polyhedral cone $\sigma\subseteq N_\R$ (or similarly for $\sigma\subseteq M_\R$), we define its \emph{dual} cone to be:
	\[\sigma^\dual=\{m\in M_\R|\langle m,u\rangle\ge 0 \text{ for all } u\in \sigma\}.\]
\end{defn}

It is easy to see that $\sigma^\dual$ is also convex rational polyhedral cone. If $\sigma$ is strongly convex and $\dim\sigma = d$ then the same is true of $\sigma^\dual$. Moreover, for any $\sigma$ we have $\sigma^{\dual\dual} = \sigma$.

We now have the following:

\begin{prop}\label{toric cone}
	If $X$ is a normal affine toric variety of dimension $d$, then there is a (uniquely determined) strongly convex rational polyhedral cone $\sigma_X\subseteq N_\R$ for which $\sigma_X^\dual\cap M = \semi_X$ and 
	\begin{align*}
	\sigma_X\cap N&=\{u\in N|\lambda^u:\G_m\to T_d \text{ extends to a morphism }\A^1\to X\}\\
	&=\left\{u\in N\middle|\lim_{t\to 0}\lambda^u(t)\text{ exists in }X\right\}.
	\end{align*}
	We call $\sigma_X$ the \emph{cone associated to }$X$. Again, if $R$ is a toric $\F$-algebra, then we write $\sigma_R$ for $\sigma_{\Spec R}$.
\end{prop}
\begin{proof}
	This follows from \cite{CLS} Theorem 1.3.5 (for $\sigma_{X}^\dual\cap M$) and Proposition 3.2.2 (for $\sigma_X\cap N$). Note that it is clear from our definitions that a convex rational polyhedral cone $\sigma\subseteq N_\R$ is uniquely determined by $\sigma\cap N$.
\end{proof}

\begin{rem}
	Based on the statement of Proposition \ref{toric cone}, it would seem more natural to simply define the cone associated to $\Xc$ to be $\sigma_\Xc^\dual$, and not mention the lattice $N$ at all. The primary reason for making this choice in the literature is to simplify the description of non-affine toric varieties, which is not relevant to our applications. Nevertheless we shall use the convention established in Proposition \ref{toric cone} to keep our treatment compatible with existing literature, and specifically to avoid having to reformulate Theorem \ref{class group}, below.
\end{rem}

Rephrasing the statement of Proposition \ref{S toric} in terms of cones, we get:

\begin{cor}\label{S cone}
	We have $\sigma_{\Sc} = \Cone(e_1,e_2,e_3,2e_1+2e_2-e_3)$.
\end{cor}
\begin{proof}
	The description of $\semi_\Xc$ in Proposition \ref{S toric} immediately implies that
	\[\sigma_\Xc^\dual = \{(x,y,z)\in \R^3|x,y,x\ge 0, 2x+2y\ge z\} = \Cone(e_1,e_2,e_1+2e_3,e_2+2e_3).\]
	Thus we get
	\begin{align*}
	\sigma_\Xc &= \sigma_\Xc^{\dual\dual} = \{u\in N|\langle m,u\rangle\ge 0 \text{ for all } m\in\sigma_\Xc^\dual\}\\
	&= \{u\in N|\langle m,e_1\rangle,\langle m,e_2\rangle,\langle m,e_1+2e_3\rangle,\langle m,e_2+2e_3\rangle\ge 0\}\\
	&= \{(x,y,z)\in \R^3|x,y\ge 0,x+2z\ge 0, y+2z\ge 0\}\\
	&=\Cone(e_1,e_2,e_1+2e_3,e_2+2e_3).
	\end{align*}
\end{proof}

\subsection{Class Groups of Toric Varieties}\label{ssec:class group}

The benefit of this entire discussion is that Weil divisors on toric varieties are much easier work with than they are for general varieties. In order to explain this, we first introduce a few more definitions.

For any variety $X$, we will let $\Div(X)$ denote the group of Weil divisors of $X$. If $X = \Spec R$ is normal, affine and toric of dimension $d$, then the torus $T_d$ acts on $X$, and hence acts on $\Div(X)$. We say that a divisor $D\in \Div(X)$ is \emph{torus-invariant} if it is preserved by this action. We will write $\Div_{T_d}(X)\subseteq \Div(X)$ for the group of torus invariant divisors.

Now consider the (strongly convex, rational polyhedral) cone $\sigma_R\subseteq N_\R$. We will let $\sigma_R(1)$ denote the set of \emph{edges} (1 dimensional faces) of $\sigma_R$. For any $\rho\in\sigma_R(1)$, note that $\rho\cap N$ is a semigroup isomorphic to $\Z_{\ge 0}$, and so there is a unique choice of generator $u_\rho\in \rho\cap N$ (called a minimal generator). By Proposition \ref{toric cone}, the limit 
$\ds\gamma_\rho = \lim_{t\to 0}\lambda^{u_\rho}(t)\in X$ exists. Thus we may consider its orbit closure $D_\rho= \overline{T_d\cdot \gamma_\rho}\subseteq X$.

The following theorem allows us to characterize $\Cl(R)$, and $[\omega_R]\in\Cl(R)$, entirely in terms of the set $\sigma_R(1)$.

\begin{thm}\label{class group}
	Let $X = \Spec R$ be a normal affine toric variety, with cone $\sigma_R\subseteq N_\R$. We have the following:
	\begin{enumerate}
		\item For any $\rho\in\sigma_R(1)$, $D_\rho\subseteq X$ is a torus-invariant prime divisor. Moreover, $\ds\Div_{T_d}(X) = \bigoplus_{\rho\in\sigma_X(1)}\Z D_\rho$.
		\item Any divisor $D\in \Div(X)$ is rationally equivalent to a torus-invariant divisor.
		\item For any $m\in M$, the rational function $\chi^m\in K(X)$ has divisor $\ds\div(\chi^m) = \sum_{\rho\in\sigma_X(1)}\langle m,u_\rho\rangle D_\rho$.
		\item For any torus-invariant divisor $D$,
		\[\O(D) = \{f\in K(X)| \div(f)+D\ge 0\} = \bigoplus_{\chi^m\in\O(D)}\F \chi^m = \bigoplus_{\div(\chi^m)+D\ge 0}\F \chi^m \subseteq K(X)\]
		\item There is an exact sequence
		\[M\to\Div_{T_d}(X)\to \Cl(R)\to 0\]
		where the first map is $m\mapsto \div(\chi^m)$ and the second map is $D\mapsto \O(D)$.
		\item $R$ is Cohen--Macaulay and we have $\ds\omega_R\cong \O\left(-\sum_{\rho\in\sigma_X(1)}D_\rho\right)$
	\end{enumerate}
\end{thm}
\begin{proof}
	By the orbit cone correspondence (\cite{CLS} Theorem 3.2.6), it follows that each $D_\rho$ is a torus-invariant prime divisor, and moreover that these are the only torus-invariant prime divisors. The rest of (1) follows easily from this (cf. \cite{CLS} Exercise 4.1.1).
	
	(3) is just \cite{CLS} Proposition 4.1.2. (4) follows from \cite{CLS} Proposition 4.3.2. (5) is \cite{CLS} Theorem 4.1.3, and (2) is an immediate corollary of (5). Lastly (6) is \cite{CLS} Theorems 8.2.3 and 9.2.9.
\end{proof}

But now Corollary \ref{S cone} and Theorem \ref{class group} make it straight-forward to compute $\Cl(\Sc)$ and $[\omega_{\Sc}]$:

\begin{prop}\label{Cl(S)}
	Let $e_0 = 2e_1+2e_2-e_3$, so that $\sigma_{\Xc} = \Cone(e_0,e_1,e_2,e_3)$. For each $i$, let $\rho_i = \R_{\ge 0}e_i$ and $D_i = D_{\rho_i}$, so that $u_{\rho_i} = e_i$ and $\sigma_{\Xc}(1) = \{\rho_0,\rho_1,\rho_2,\rho_3\}$. Then:
	\begin{enumerate}
		\item We have an isomorphism $\Cl(\Sc)\cong \Z$ given by $k\mapsto \O(kD_0)$.
		\item $\omega_{\Sc}\cong \O(2D_0)$.
		\item If $\Mc$ is a generic rank 1 reflexive, self-dual module over $\Sc$, then $\Mc\cong \O(D_0)$.
		\item Identifying $\Sc$ with $\F[x,xz,xz^2,y,yz,yz^2]\subseteq \F[x,y,z]$ as in Proposition \ref{S toric} we get
		\begin{align*}		
		\O(D_0) &\cong \Sc\cap xz\F[x,y,z] = (xz,xz^2)\subseteq\Sc\\
		\omega_{\Sc} = \O(2D_0) &\cong \Sc\cap x\F[x,y,z] = (x,xz,xz^2)\subseteq\Sc,
		\end{align*}
		so in particular, $\dim_\F\O(D_0)/m_{\Sc}=2$.
		\item There is a surjection $\O(D_0)\otimes_{\Sc}\O(D_0)\surjection \omega_{\Sc}$.
	\end{enumerate}
\end{prop} 
\begin{proof}
	Write $x = \chi^{e_1}, y = \chi^{e_2}$ and $z = \chi^{e_3}$, so that $\F[M] = \F[x^{\pm 1},y^{\pm 1},z^{\pm 1}]$. By Theorem \ref{class group}(3) we get that
	\begin{align*}
	\div(x) &= 2D_0+D_1,&
	\div(y) &= 2D_0+D_2,&
	\div(z) &= D_3-D_0.
	\end{align*}
	It follows that $D_1\sim -2 D_0$, $D_2\sim -2 D_0$ and $D_3\sim D_0$, and so from Theorem \ref{class group}(5) we get that $\Cl(\Sc)$ is generated by $\O(D_0)$. Moreover, the exactness in Theorem \ref{class group}(5) gives that the above relations are the only ones between the $D_i$'s, and so $\O(D_0)$ is non-torsion in $\Cl(\Sc)$, indeed giving the isomorphism $\Cl(\Sc)\cong \Z$. (Alternatively, Theorem \ref{class group}(5) implies that the $\Z$-rank of $\Cl(\Sc)$ is at least $\rank \div_{T_3}(\Xc)-\rank M = 4-3=1$.) This proves (1).
	
	For (2), we simply use Theorem \ref{class group}(6):
	\[\omega_{\Sc}\cong \O(-D_0-D_1-D_2-D_3) \cong \O\big(-D_0-(-2D_0)-(-2D_0)-D_0\big) = \O(2D_0).\]
	Now by (1), any generic rank 1 reflexive $\Sc$-module is in the form $\O(kD_0)$ for some $k\in \Z$, and by (2) $\O(kD_0)^*\cong \O((2-k)D_0)$. Thus if $\O(kD_0)$ is self-dual then $k=1$, giving (3).
	
	By the above computations, we get that
	\[\div(x^ay^bz^c) = (2a+2b-c)D_0+2aD_1+2bD_2+cD_3\]
	for any $(a,b,c)\in \Z^3$. Now note that $\O(D_0) \cong \O(-D_1-D_3)$ and $\O(2D_0) \cong \O(-D_1)$ which are both ideals of $\Sc$. But now by Theorem \ref{class group}(4) as ideals of $\Sc$ we have
	\begin{align*}
	\O(D_0) &\cong \O(-D_1-D_3)
	= \left(x^ay^bz^c\middle| 2a+2b-c\ge 0, 2a\ge 1,2b\ge 0, c\ge 1 \right)\\
	&= \left(x^ay^bz^c\middle| 2a+2b-c\ge 0, a\ge 1,b\ge 0, c\ge 1 \right)\\
	&= \left(x^ay^bz^c\middle|x^ay^bz^c\in\Sc, xz|x^ay^bz^c\right)\\
	&=\Sc\cap xz\F[x,y,z] = (xz,xz^2)\\
	\O(2D_0) &\cong \O(-D_1)
	= \left(x^ay^bz^c\middle| 2a+2b-c\ge 0, 2a\ge 1,2b\ge 0, c\ge 0 \right)\\
	&= \left(x^ay^bz^c\middle| 2a+2b-c\ge 0, a\ge 1,b\ge 0, c\ge 0 \right)\\
	&= \left(x^ay^bz^c\middle|x^ay^bz^c\in\Sc, x|x^ay^bz^c\right)\\
	&=\Sc\cap x\F[x,y,z] = (x,xz,xz^2),
	\end{align*}
	proving (4).
	
	Now identify $\O(D_0)$ with $(xz,xz^2)\subseteq \Sc$ and $\omega_{\Sc}$ with $(x,xz,xz^2)\subseteq \Sc$. Notice that
	\[\O(D_0)\O(D_0) = (xz,xz^2)(xz,xz^2) = (x^2z^2,x^2z^3,x^2z^4) = xz^2(x,xz,xz^2) = xz^2\omega_{\Sc}.\]
	Thus we can define a surjection $f:\O(D_0)\otimes_{\Sc}\O(D_0)\surjection \omega_{\Sc}$ by $\ds f(\alpha\otimes \beta) = \frac{1}{xz^2}\alpha\beta$, proving (5).
\end{proof}

We can now compute $\Cl(\Rc)$ and $\omega_{\Rc}$, by using the following lemma:

\begin{lemma}\label{toric product}
	For any normal, affine toric varieties $X$ and $Y$ the natural map $\Cl(X)\oplus\Cl(Y)\to \Cl(X\times Y)$ given by $([A],[B])\mapsto [A\boxtimes B]$ is an isomorphism which sends $([\omega_X],[\omega_Y])$ to $\omega_{X\times Y}$.
\end{lemma}
\begin{proof}
	By \cite{CLS} Proposition 3.1.14, $X\times Y$ is a toric variety with cone $\sigma_{X\times Y} = \sigma_X\times \sigma_Y$. It follows that $\sigma_{X\times Y}(1) = \sigma_X(1)\sqcup \sigma_Y(1)$. The claim now follows immediately from Theorem \ref{class group}.
\end{proof}

Thus we have:

\begin{cor}\label{Cl(R)}
	The map $\varphi:\Cl(\Sc)^k\to \Cl(\Rc)$ given by
	\[([A_1],\ldots,[A_k])\mapsto \big[(A_1\boxtimes A_2\boxtimes \cdots\boxtimes A_a)[x_1,\ldots,x_s]\big]\]
	is an isomorphism which sends $([\omega_\Sc],\ldots,[\omega_\Sc])$ to $[\omega_{\Rc}]$.
	
	Consequently there is a unique self-dual generic rank $1$ reflexive module $\Mc$ over $\Rc$, which is the image of $([\O(D_0)],\ldots,[\O(D_0)])$. We have that $\dim _\F\Mc/m_{\Rc} = 2^k$ and there is a surjection $\Mc\otimes_{\Rc}\Mc\surjection \omega_{\Rc}$.
\end{cor} 
\begin{proof}
	The isomorphism follows immediately from Corollary \ref{toric product} (noting that $\A^1$ is a toric variety with $\Cl(\A^1) = 0$ and $\omega_{\A^1} = \A^1$).
	
	Now for any self-dual generic rank 1 reflexive module $\Mc$ over $\Rc$, it follows that $[\Mc] = \varphi([A_1],\ldots,[A_a])$ where each $A_i$ is self-dual. Proposition \ref{Cl(S)} implies that each $A_i$ is isomorphic to $\O(D_0)$, as claimed.
	
	For this $\Mc$ we indeed have
	\begin{align*}
	\Mc/m_{\Rc} &= \frac{\O(D_0)^{\boxtimes k}[x_1,\ldots,x_s]}{m_{\Sc}^{\boxtimes k}\boxtimes(x_1,\ldots,x_s)} \cong \left(\frac{\O(D_0)}{m_{\Sc}}\right)^{\boxtimes k} = \left(\F^2\right)^{\boxtimes k} = \F^{2^k}.
	\end{align*}
	Also, the surjection $\O(D_0)\otimes_{\Sc}\O(D_0)\surjection \omega_{\Sc}$ from Proposition \ref{Cl(S)} indeed gives a surjection
	\begin{align*}
	\Mc\otimes_{\Rc}\Mc 
	&=\left(\O(D_0)^{\boxtimes k}[x_1,\ldots,x_s] \right)\otimes_{\Rc}\left(\O(D_0)^{\boxtimes k}[x_1,\ldots,x_s]\right)\\
	&\cong\left(\O(D_0)\otimes\O(D_0)\right)^{\boxtimes k}[x_1,\ldots,x_s]\surjection\omega_{\Sc}^{\boxtimes k}[x_1,\ldots,x_s]\cong \omega_{\Rc}.
	\end{align*}
\end{proof}

Which completes the proof of Theorem \ref{self-dual decompleted}.

\begin{rem}
In our proof of Theorem \ref{self-dual decompleted}, we never actually used the first condition, namely that $\Mc$ was maximal Cohen--Macaulay over $\Rc$. We only used the (strictly weaker) assumption that $\Mc$ was reflexive, which, combined with the fact that $\Mc$ was self-dual, was enough to uniquely determine the structure of $\Mc$.
	
	In most situations, the modules $M_\infty$ produced by the patching method will be maximal Cohen--Macaulay, but it is possible that they might fail to be self-dual (e.g. if they arise from the cohomology of a non self-dual local system).
	
	In this situation it is possible to formulate a weaker version of Theorem \ref{self-dual}, where one drops the self-duality assumption. Specifically one can show (in the notation of Proposition \ref{Cl(S)}) that the only Cohen--Macaulay generic rank one modules over the ring $\Sc$ are the 5 modules:
	\begin{align*}
	\O(-D_0) &= (xz,xz^2,yz,yz^2)\\
	\O &= \Sbar\\
	\O(D_0) &= (xz,xz^2)\\
	\O(2D_0) &= (x,xz,xz^2)\\
	\O(3D_0) &= (x^2,x^2z,x^2z^2,x^2z^3).
	\end{align*}
	This can be done quite simply by first completing at $m_{\Sc}$, and noting that if $(t_1,t_2,t_3)$ is a regular sequence for $\Sbar$ then it must also be a regular sequence for $\Mc_{m_{\Sc}}$ over $\Sbar$ where $\Mc$ is any maximal Cohen--Macaulay module over $\Sc$, which implies that $\Sbar$ and $\Mc_{m_{\Sc}}$ are both finite free $\F[[t_1,t_2,t_3]]$-modules. Moreover if $\genrank_{\Sc}\Mc = 1$ then $\Sbar$ and $\Mc_{m_{\Sc}}$ have the same rank over $\F[[t_1,t_2,t_3]]$, and so
	\[\dim_{\F}\Mc/m_{\Sc}\Mc = \dim_{\F}\Mc_{m_{\Sc}}/m_{\Sc}\Mc_{m_{\Sc}} \le \dim_{\F}\Mc/(t_1,t_2,t_3) = \dim_{\F}\Sc/(t_1,t_2,t_3).\]
	Thus using the regular sequence $(x,yz^2,y-xz^2)$ for $\Sc$, we see that if $\Mc$ is a maximal Cohen--Macaulay $\Sc$-module of generic rank $1$, then $\dim_{\F}\Mc/m_{\Sc}\Mc\le \dim_{\F}\Sc/(x,yz^2,y-xz^2) = 4$.
	
	It is easy to verify from the description of $\Cl(\Sc)$ given in Proposition \ref{Cl(S)}, and the description of $\O(D)$ from Theorem \ref{class group} that the five modules listed above are the only generic rank $1$ reflexive $\Sc$-modules $\Mc$ with $\dim_{\F}\Mc/m_{\Sc}\Mc\le 4$, and all of these can directly be shown to be maximal Cohen--Macaulay.
	
	This unfortunately does not allow us to uniquely deduce the structure of $\Mc$ and hence of $M_\infty$, but it does give us the bound $\dim_\F M_\infty/m_{R_\infty} \le 4^k$, and could potentially lead to more refined information about $M_\infty$, which may be of independent interest.
\end{rem}

\subsection{Class groups of completed rings}\label{ssec:decompletion}

The goal of this section is to prove that Theorem \ref{self-dual decompleted} implies Theorem \ref{self-dual mod l}. We shall do this by proving that the natural map $\Cl(\Rc)\to \Cl(\Rbar)$ given by $[\Mc]\mapsto [\Mc\otimes_{\Rc}\Rbar]$ is an isomorphism.

First note that the Theorem \ref{self-dual mod l} will indeed follow from this. Assume that $\Mbar$ is an $\Rbar$-module satisfying the conditions of Theorem \ref{self-dual mod l}. Then in particular it corresponds to an element of $\Cl(\Rbar)$, and so there is some reflexive generic rank $1$ $\Rc$-module $\Mc$ with $\Mbar\cong \Mc\otimes_{\Rc}\Rbar = \invlim \Mc/m_{\Rc}^n\Mc$. We claim that $\Mc$ satisfies conditions (1) and (2) of Theorem \ref{self-dual decompleted} (by assumption, it satisfies condition (3)).

Showing that $\Mc$ is self-dual is equivalent to showing that $2[\Mc] = [\omega_{\Rc}]$ in $\Cl(\Rc)$, which follows from the fact that $2[\Mc\otimes_{\Rc}\Rbar] = 2[\Mbar] = [\omega_{\Rbar}]$ in $\Cl(\Rbar)$ and the fact that $\omega_{\Rbar}\cong \omega_{\Rc}\otimes_{\Rc}\Rbar$ (cf \cite[Corollaries 21.17 and 21.18]{Eisenbud}).

We now observe that $\Mc$ is maximal Cohen--Macaulay over $\Rc$.\footnote{Strictly speaking it is not necessary to show this, as condition (1) was never used in the proof of Theorem \ref{self-dual decompleted}, but we will still show it for the sake of completeness.} By Theorem \ref{R^st}, $(C,Y,B-Z)$ is a regular sequence for $\Sbar$ consisting of homogeneous elements. It follows that this is also a regular sequence for $\Sc$, and so $\Rc$ also has a regular sequence $(z_1,\ldots,z_{3k+s})$ consisting entirely of homogeneous elements. Now it follows that this regular sequence is also regular for $\Rbar$, and hence for $\Mbar$. But now as the $z_i$'s are all homogeneous it follows that $\Mc/(z_1,\ldots,z_i)\injection \Mbar/(z_1,\ldots,z_i)$ for all $i$ and so $(z_1,\ldots,z_{3k+s})$ is also a regular sequence for $\Mc$. Hence $\Mc$ is maximal Cohen--Macaulay over $\Rc$.

Hence $\Mc$ satisfies the conditions of Theorem \ref{self-dual decompleted}, so we get that $\dim_{\F}\Mc/m_{\Rc}\Mc = 2^k$ and $\tau_{\Mc}:\Mc\otimes_{\Rc}\Mc\to\omega_{\Rc}$ is surjective.

Now as $\Mbar/m_{\Rbar}\Mbar\cong \Mc/m_{\Rc}\Mc$, we indeed get $\dim_{\F}\Mbar/m_{\Rbar}\Mbar = 2^k$. 

To show that $\tau_{\Mbar}$ is surjective, note that
\[(\Mc\otimes_{\Rc}\Mc)\otimes_{\Rc}\Rbar \cong (\Mc\otimes_{\Rc}\Rbar)\otimes_{\Rbar}(\Mc\otimes_{\Rc}\Rbar) \cong \Mbar\otimes_{\Rbar}\Mbar\]
so as $\omega_{\Rc}$ is a quotient of $\Mc\otimes_{\Rc}\Mc$ it follows that $\omega_{\Rbar}\cong \omega_{\Rc}\otimes_{\Rc}\Rbar$ is a quotient of $\Mbar\otimes_{\Rbar}\Mbar$ and so $\tau_{\Mbar}$ is indeed surjective by Lemma \ref{trace}, which completes the proof of Theorem \ref{self-dual mod l}.

Unfortunately, it is not true in general that if $R$ is a graded $\F$-algebra and $\Rhat$ is the completion at the irrelevant ideal then the map $\Cl(R)\to \Cl(\Rhat)$ is an isomorphism. However Danilov \cite{Danilov} has shown that this is true in certain cases:

\begin{thm}[Danilov]\label{Danilov}
Let $V$ be a smooth projective variety with a very ample line bundle $\Lc$ giving an injection $V\injection \P^N$. Let $\Spec S\subseteq \A^{N+1}$ be the affine cone on $V$, so that $S$ is a graded $\F$-algebra, and let $\Sh$ be the completion of $R$ at the irrelevant ideal. Then:
\begin{enumerate}
\item The natural map $\Cl(S)\to \Cl(\Sh)$ is an isomorphism if and only if $H^1(V,\Lc^{\otimes i}) = 0$ for all $i\ge 1$.
\item For $s>0$, the natural map $\Cl(S)\to \Cl(\Sh[[x_1,\ldots,x_s]])$ is an isomorphism if and only if $H^1(V,\Lc^{\otimes i}) = 0$ for all $i\ge 0$.
\end{enumerate}
\end{thm}

We now make the following observation:

\begin{lemma}\label{H^d(S)=0}
There exists a smooth projective variety $\Vc$ and an ample line bundle $\Lc$ on $\Vc$ such that $\Spec \Sc$ is the affine cone over $\Vc$, under the projective embedding induced by $\Lc$. We have $H^d(\Vc,\Lc^{\otimes i}) = 0$ for all $d\ge 1$ and $i\ge 0$.
\end{lemma}
\begin{proof}
This is largely a restatement of Proposition \ref{S toric}(1). Specifically we have $\Vc = \P^1\times\P^1$ and $\Lc = \O_{\P^1}(2)\boxtimes\O_{\P^1}(1)$. To prove the vanishing of cohomology, we simply note that $H^d(\P^1,\O_{\P^1}(i)) = 0$ for all $d\ge 1$ and $i\ge 0$, and so
\[H^d(\P^1\times\P^1,\O_{\P^1}(2i)\boxtimes\O_{\P^1}(i)) = \bigoplus_{e=0}^dH^e(\P^1,\O_{\P^1}(2i))\otimes H^{d-e}(\P^1,\O_{\P^1}(i)) = 0\]
for any $d\ge 1$ and $i\ge 0$.
\end{proof}

\begin{rem}
By \cite[Excercise 18.16]{Eisenbud}, the conclusion of Lemma \ref{H^d(S)=0} about the vanishing of the cohomology groups $H^d(\Vc,\Lc^{\otimes i})$ holds whenever $\Sc$ is Cohen--Macaulay and $d<\dim \Vc$. This means the results of this section will be applicable in most cases where the ring $R_\infty$ is Cohen--Macaulay, and so this does not impose a fundamental limitation on our method.
\end{rem}

It follows that $\Cl(\Sc)\cong \Cl(\Sbar)$. In fact (as the natural map $\Cl(\Sc)\to \Cl(\Sc[x_1,\ldots,x_s])$ is an isomorphism) it follows that the natural map $\Cl(\Sc[x_1,\ldots,x_s])\to \Cl(\Sbar[[x_1,\ldots,x_s]])$ is an isomorphism, and so Theorem \ref{self-dual mod l} follows in the case when $k=1$.

When $k>1$ however, we cannot directly appeal to Theorem \ref{Danilov}, as $\Spec \Rc$ is no longer the cone over a smooth projective variety, and in fact $\Spec \Rc$ does not have isolated singularities. Fortunately it is fairly straightforward to adapt the method of \cite{Danilov} to our situation. Specifically, we will prove the following (which obviously applies to the ring $\Rc$):

\begin{prop}\label{Danilov for products}
Let $\Vc_1,\ldots,\Vc_k$ be a collection of smooth projective varieties of dimension at least $1$, and for each $j$, let $\Lc_j$ be a very ample line bundle on $\Vc_j$ giving an injection $\Vc_j\injection \P^{N_j}$. Let $\Spec S_j\subseteq \A^{N_j+1}$ be the affine cone on $V_j$. Let 
\[R = \left[\bigotimes_{j=1}^kS_j\right][x_1,\ldots,x_s]\]
(for some $s\ge 0$), so that $R$ is a graded $\F$-algebra. Let $\Rhat$ be the completion of $R$ at the irrelevant ideal.

If we have that $H^d(\Vc_j,\Lc_j^{\otimes i}) = 0$ for all $d = 1,2$, $j=1,\ldots,k$ and $i\ge 0$, then the natural map $\Cl(R)\to \Cl(\Rhat)$ is an isomorphism.
\end{prop}
\begin{proof}
For simplicity, we first reduce to the case $s=0$. If $s\ge 2$, then we may simply let $\Vc_{k+1} = \P^{s-1}$, $\Lc = \O_{\P^{s-1}}(1)$ and note that we still have the cohomology condition $H^d(\Vc_{k+1},\Lc_{k+1}^{\otimes i}) = H^d(\P^{s-1},\O_{\P^{s-1}}(i)) = 0$ for all $d\ge 1$ and $i\ge 0$. So the $s\ge 2$ case follows from the $s=0$ case.

The $s=1$ case now follows from the $s=0$ and $s\ge 2$ cases by letting $\ds R_0 = \bigotimes_{j=1}^kS_j$ and considering the commutative diagram:

\begin{center}
	\begin{tikzpicture}
	\node(11) at (0,2) {$\Cl(R_0)$};
	\node(21) at (4,2) {$\Cl(R_0[x_1])$};
	\node(31) at (8,2) {$\Cl(R_0[x_1,x_2])$};
	
	\node(10) at (0,0) {$\Cl(\Rhat_0)$};
	\node(20) at (4,0) {$\Cl(\Rhat_0[[x_1]])$};
	\node(30) at (8,0) {$\Cl(\Rhat_0[[x_1,x_2]])$};
	
	\draw[->] (11)--(10);
	\draw[->] (21)--(20);
	\draw[->] (31)--(30);
	
	\draw[->] (11)--(21);
	\draw[->] (21)--(31);
	
	\draw[->] (10)--(20);
	\draw[->] (20)--(30);
	\end{tikzpicture}
\end{center}
and noting that the maps on the top row are isomorphisms by standard properties of the class groups of varieties, and the maps on the bottom row are injective (since if $M$ is a reflexive $\Rhat_0$ module and $M[[x_1]] = M\otimes_{\Rhat_0}\Rhat_0[[x_1]]$ is a free $\Rhat_0[[x_1]]$-module, then $M/m_{\Rhat_0}\cong M[[x_1]]/m_{\Rhat_0[[x_1]]}\cong \F$, and so $M$ is a cyclic, and thus a free $\Rhat$-module). So from now on, we shall assume $s=0$.

We first introduce some notation.

For each $j$, let $Y_j = \Spec S_j$. Let $\ds X = \prod_{j=1}^k\Vc_j$ and $\ds Y = \prod_{j=1}^kY_j = \Spec R$. Also let $\ds Z_j = \prod_{j'\ne j}Y_{j'}\subseteq Y$ and $Z = Z_1\cup Z_2\cdots\cup Z_k\subseteq Y$. Note that each $Z_j$ is irreducible subscheme of $Y$ of codimension at least $2$.

Write $Z_j = \Spec R/I_j$ and $Z = \Spec R/I$. Note that $I_j = m_jR$, where $m_j$ is the irrelevant ideal of $S_j$, and $I = I_1I_2\cdots I_k$. In particular, $I_j$ and $I$ are homogeneous ideals of $R$. Now let $\Yhat = \Spec \Rhat$, $\Ihat_j = I_j\Rhat$, $\Ihat = I\Rhat$, $\Zhat_j = \Spec \Rhat/\Ihat_j$ and $\Zhat = \Spec \Rhat/\Ihat$. Note that the $\Zhat_j$'s are still irreducible, and we have $\Zhat = \Zhat_1\cup\Zhat_2\cup\cdots\cup\Zhat_k$.

Now let $\ds C = \Proj\left(\bigoplus_{n=0}^\infty I^n\right)$ and $\ds \Ct = \Proj\left(\bigoplus_{n=0}^\infty \Ihat^n\right)$ be the blowups of $Y$ and $\Yhat$ along $Z$ and $\Zhat$ and let $p:C\to Y$ and $\pt:\Ct\to \Yhat$ be the projection maps. Let $E_j = p^{-1}(Z_j)$, $\Et_j = \pt^{-1}(\Zhat_j)$, $E = p^{-1}(Z)$ and $\Et = \pt^{-1}(\Zhat)$. Note that the $E_j$'s and $\Et_j$'s are irreducible and we have $E = E_1\cup E_2\cup\cdots\cup E_k$ and $\Et = \Et_1\cup \Et_2\cup\cdots\cup \Et_k$.

Let $m_R\subseteq R$ denote the irrelevant ideal and let $\mhat_R = m_R\Rhat\subseteq R$ be its completion. Notice that we have natural isomorphisms $p^{-1}(\{m_R\}) = E_1\cap E_2\cap\cdots \cap E_k \cong X$ and $\pt^{-1}(\{\mhat_R\}) = \Et_1\cap \Et_2\cap\cdots \cap \Et_k \cong X$. Identify $X$ with its images in both $C$ and $\Ct$. We will let $\Chat$ and $\Cthat$ denote the formal completions of $C$ and $\Ct$ along the subvariety $X$.

Lastly, we define a rank $k$ vector bundle $\xi$ on $X$ as follows. For each $j$, let $\pi_jX\to \Vc_j$ be the projection map, so that $\pi_j^* \Lc_j = \O_{\Vc_1}\boxtimes\cdots\boxtimes \Lc_j\boxtimes\cdots\boxtimes \O_{\Vc_k}$ is a line bundle on $X$. We will let $\xi = \pi_1^* \Lc_1\oplus \pi_2^* \Lc_2\oplus\cdots\oplus \pi_k^* \Lc_k$.

We first observe the following:

\begin{lemma}\label{V(xi)}
There is an isomorphism $C\cong V(\xi)$, where $V(\xi)$ is the total space of the vector bundle $\xi$ over $X$. This isomorphism is compatible with the inclusions $X\injection C$ and $X\injection V(\xi)$.

Moreover we have isomorphisms of formal schemes $\Chat\cong \Vhat(\xi)\cong \Cthat$, where $\Vhat(\xi)$ is the completion of $V(\xi)$ along the zero section $X\injection V(\xi)$. These isomorphisms are again compatible with the natural inclusions of $X$.
\end{lemma}
\begin{proof}
Letting $V(\Lc_j)$ be the total space of $\Lc_j$ over $\Vc_j$ we see that
\begin{align*}
V(\xi) &= V(\pi_1^*\Lc_1\oplus \pi_2^*\Lc_2\oplus\cdots\pi_k^*\Lc_k) \cong V(\pi_1^*\Lc_1)\times_XV(\pi_2^*\Lc_2)\times_X\cdots\times_XV(\pi_k^*\Lc_k)\\
& \cong V(\Lc_1)\times V(\Lc_2)\times\cdots V(\Lc_k).
\end{align*}
Now as in \cite[Lemma 1(3)]{Danilov}, each $V(\Lc_j)$ is the blowup of $\Spec S_j$ at the point $m_j$. Now using this and the fact that $I = I_1I_2\cdots I_k = m_1\otimes m_2\otimes \cdots \otimes m_k$ we indeed get
\begin{align*}
V(\xi) &\cong \prod_{j=1}^kV(\Lc_j) \cong \prod_{j=1}^k\Proj\left(\bigoplus_{n=0}^\infty m_j^n \right) \cong \Proj\left(\bigoplus_{n=0}^\infty \left(m_1^n\otimes m_2^n\otimes\cdots\otimes m_k^n\right)\right) = \Proj\left(\bigoplus_{n=0}^\infty I^n\right) = C,
\end{align*}
where we used the fact that $\ds\Proj\left(\bigoplus_{n=0}^\infty A_n \right)\times \Proj\left(\bigoplus_{n=0}^\infty B_n \right) \cong \Proj\left(\bigoplus_{n=0}^\infty A_n\otimes B_n \right)$ for finitely generated graded $R$-algebras $\ds\bigoplus_{n=0}^\infty A_n$ and $\ds\bigoplus_{n=0}^\infty B_n$. (See for instance, \cite[Exercise 9.6.D]{RisingSea})

It is easy to check that these isomorphisms are compatible with the embeddings $X\injection C, V(\xi)$. This automatically gives $\Chat \cong \Vhat(\xi)$.

Now notice that the subscheme $X = p^{-1}(\{m_R\})\subseteq C$ is cut out by the ideal sheaf $\Is=p^*(m_R)$ and similarly the subscheme $X = \pt^{-1}(\{\mhat_R\})\subseteq \Ct$ is cut out by the ideal sheaf $\Ish = \pt^*(\mhat_R)$. But now using the fact that $\mhat^a_R/\mhat^b_R = m^a_R/m^b_R$ for all $a>b$, as in \cite[Section 4]{Danilov} we get that
\begin{align*}
\Cthat 
&= \dirlim_n \left(X,\O_{\Ct}/\Ish^{n+1}\right) 
 = \dirlim_n \pt^{-1}\left(\Spec\left(\Rhat/\mhat_R^{n+1}\right) \right)
 = \dirlim_n \Proj\left(\bigoplus_{i=0}^\infty \mhat_R^i/\mhat_R^{n+i+1}\right)\\
&= \dirlim_n \Proj\left(\bigoplus_{i=0}^\infty m_R^i/m_R^{n+i+1}\right)
 = \dirlim_n p^{-1}\left(\Spec\left(R/m_R^{n+1}\right)\right)
 = \dirlim_n \left(X,\O_{C}/\Is^{n+1}\right) = \Chat,
\end{align*}
completing the proof of the lemma.
\end{proof}

We next note the following analogue of \cite[Section 2]{Danilov}:

\begin{lemma}\label{Cl to Pic}
Assume the setup of Proposition \ref{Danilov for products}, but without the assumption about the vanishing of the cohomology groups. There is a commutative diagram with exact rows:
\begin{center}
	\begin{tikzpicture}
	\node(01) at (0,2) {$0$};
	\node(11) at (2,2) {$\Z^k$};
	\node(21) at (4,2) {$\Pic(X)$};
	\node(31) at (6,2) {$\Cl(R)$};
	\node(41) at (8,2) {$0$};
	
	\node(00) at (0,0) {$0$};
	\node(10) at (2,0) {$\Z^k$};
	\node(20) at (4,0) {$\Pic(\Vhat(\xi))$};
	\node(30) at (6,0) {$\Cl(\Rhat)$};
	\node(40) at (8,0) {$0$};
	
	\draw[->] (11)--(10) node[midway,left]{$=$};
	\draw[->] (21)--(20);
	\draw[->] (31)--(30);
	
	\draw[->] (01)--(11);
	\draw[->] (11)--(21);
	\draw[->] (21)--(31);
	\draw[->] (31)--(41);
	
	\draw[->] (00)--(10);
	\draw[->] (10)--(20);
	\draw[->] (20)--(30);
	\draw[->] (30)--(40);
	\end{tikzpicture}
\end{center}
Where the map $\Cl(R)\to \Cl(\Rhat)$ is the natural completion map, and the map $\Pic(X)\to \Pic(\Vhat(\xi))$ is the pullback along the projection map $\Vhat(\xi)\to X$.
\end{lemma}
\begin{proof}
Let $U= Y\sm Z$ and $\Ut=\Yhat\sm \Zhat$, and note that $p$ and $\pt$ induce isomorphisms $C\sm E\cong U$ and $\Ct\sm\Et\cong\Ut$. Thus we will also regard $U$ and $\Ut$ as being open subschemes of $C$ and $\Ct$. As in \cite[Lemma 3]{Danilov} we get that $U$ and $\Ut$ are both regular.

Now as $Z\subseteq Y$ and $\Zhat\subseteq \Yhat$ have codimension at least two, the restriction maps $\Cl(R) = \Cl(Y)\to \Cl(U)$ and $\Cl(\Rhat)=\Cl(\Yhat)\to \Cl(\Ut)$ are isomorphisms.

Now as each $E_j\subseteq C$ is an irreducible subvariety of codimension 1, and $E = E_1\cup E_2\cup \cdots\cup E_k$, we get the the restriction map $\Cl(C)\to \Cl(C\sm E) = \Cl(U)$ is a surjection, with kernel equal to the $\Z$-span of $[E_1],\ldots,[E_k]$ (cf \cite[Proposition II.6.5]{Hartshorne}).

We claim that $[E_1],\ldots,[E_k]\in \Cl(C)$ are $\Z$-linearly independent. Assume not. Then there exists some non-unit rational function $g$ on $C$ for which $\div g = n_1[E_1]+\cdots+n_k[E_k]$, and so in particular, $\supp g\subseteq E$. But then writing $g = p^*(g')$ for some rational function on $Y$, we get that $\supp g' \subseteq p(E) = Z$, which implies that $g'$, and hence $g$, is a unit as $Z\subseteq Y$ has codimension at least two, a contradiction.

Thus we have an exact sequence $0\to \Z^k\to \Cl(C)\to \Cl(R)\to 0$.

Similarly we have a surjection $\Cl(\Ct)\to \Cl(\Rhat)$ with kernel spanned by $[\Et_1],\ldots,[\Et_k]\in \Cl(\Ct)$, which are also $\Z$-linearly independent. This gives the exact sequence $0\to \Z^k\to \Cl(\Ct)\to \Cl(\Rhat)\to 0$.

It remains to give isomorphisms $\Pic(X)\cong \Cl(C)$ and $\Pic(\Vhat(\xi))\cong \Cl(\Ct)$ compatible with the other maps.

First, as $C$ and $\Ct$ are locally factorial, we get that $\Cl(C)\cong \Pic(C)$ and $\Cl(\Ct)\cong \Pic(\Ct)$. By \cite[Proposition 3]{Danilov}, the zero section $X\injection V(\xi)$ gives an isomorphism $\Pic(V(\xi))\cong \Pic(X)$, so by Lemma \ref{V(xi)}, $\Pic(X)\cong \Pic(V(\xi))\cong \Pic(C)\cong \Cl(C)$.

Now as $\Rhat$ is an \emph{adic} Noetherian ring with ideal of definition $\mhat_R$, the morphism $\pt:\Ct\to \Yhat=\Spec \Rhat$ is projective, and $\Cthat$ is the completion of $\Ct$ along $X = \pt^{-1}(\{m_R\})$, the argument of \cite[Proposition 4]{Danilov} implies that $\Pic(\Ct)\cong \Pic(\Cthat)$ is an isomorphism.

Thus Lemma \ref{V(xi)} gives $\Pic(\Vhat(\xi))\cong\Pic(\Cthat)\cong \Pic(\Ct)$, establishing the desired commutative diagram.
\end{proof}

Thus it will suffice to show that the map $\Pic(X)\to \Pic(\Vhat(\xi))$ induced by the projection $\Vhat(\xi)\to X$ is an isomorphism.

Now write $\ds C_n = \rSpec_X\left(\bigoplus_{i=0}^n\xi^{\otimes i}\right)$ (where $\rSpec_X$ denotes the relative $\Spec$ over $X$), so that $\ds\Vhat(\xi) = \dirlim_n C_n$. As in \cite[Proposition 5]{Danilov} we have $\ds\Pic(\Vhat(\xi))\cong \invlim_n\Pic(C_n)$.

Now for each $n$, let $\pr_n:C_n\to X$ be the projection, and let $i_n:X\to C_n$ be the zero section. Note that we canonically have $C_0 = X$ and $i_0$ and $\pr_0$ are just the identity map.

We have that $\pr_n\circ i_n = \id_X$ and so $i_n^*\circ \pr_n^* = \id_{\Pic(X)}$. Hence $\pr_n^*:\Pic(X)\to \Pic(C_n)$ is an injection (and in fact, $\Pic(X)$ is a direct summand of $\Pic(C_n)$). It follows that the map $\ds\pr^* = (\pr_n^*):\Pic(X)\to \invlim_n\Pic(C_n)\cong \Pic(\Vhat(\xi))$ is injective. In particular this means that $\Cl(R)\to\Cl(\Rhat)$ is injective.

Now for each $n$ we have $\Pic(C_n) = H^1(X,\O^*_{C_n})$. As in \cite[Section 3]{Danilov}, we consider the exact sequence of sheaves on $X$:
\[0\to \xi^{\otimes (n+1)} \to \O^*_{C_{n+1}} \to \O^*_{C_n} \to 1,\]
where the first map sends $s\in \Gamma(W,\xi^{\otimes (n+1)})$ to $1+s\in \Gamma(W,\O^*_{C_{n+1}})$. Then the long exact sequence of cohomology gives an exact sequence:
\[H^1(X,\xi^{\otimes(n+1)}) \to \Pic(C_{n+1})\to \Pic(C_n)\to H^2(X,\xi^{\otimes(n+1)}).\]
We now claim that $H^d(X,\xi^{\otimes i}) = 0$ for all $d=1,2$ and $i\ge 0$. First note that
\begin{align*}
\xi^{\otimes i} 
&= \left(\bigoplus_{j=1}^k\pi_j^* \Lc_j \right)^{\otimes i} 
 = \bigoplus_{\substack{i_1+\cdots+i_k = i\\ i_1,\ldots,i_k\ge 0}}\left(\pi_1^*\Lc_1^{\otimes i_1}\otimes\cdots\otimes \pi_k^*\Lc_k^{\otimes i_k}\right)
 = \bigoplus_{\substack{i_1+\cdots+i_k = i\\ i_1,\ldots,i_k\ge 0}}\Lc_1^{\otimes i_1}\boxtimes\cdots\boxtimes \Lc_k^{\otimes i_k}
\end{align*}
but now for any $i_1,\ldots,i_k\ge 0$ and any $d=1,2$ we get:
\begin{align*}
H^d(X,\Lc_1^{\otimes i_1}\boxtimes\cdots\boxtimes \Lc_k^{\otimes i_k}) 
&=H^d(\Vc_1\times\cdots\times \Vc_k,\Lc_1^{\otimes i_1}\boxtimes\cdots\boxtimes \Lc_k^{\otimes i_k})\\
&= \bigoplus_{\substack{d_1+\cdots+d_k = d\\ d_1,\ldots,d_k\ge 0}}\ \bigotimes_{j=1}^k H^{d_j}(\Vc_j,\Lc_j^{\otimes i_j}) = 0,
\end{align*}
since for any $k$-tuple $(d_1,\ldots,d_k)$ with $d_1+\cdots+d_k=d\in\{1,2\}$ and $d_1,\ldots,d_k\ge 0$, there must be some index $j$ for which $d_j \in \{1,2\}$, and so $H^{d_j}(\Vc_j,\Lc^{\otimes i_j}) = 0$ by assumption.

Thus for any $n\ge 0$, we indeed get that $H^1(X,\xi^{\otimes (n+1)}) = H^2(X,\xi^{\otimes (n+1)}) = 0$, and so we have $\Pic(C_{n+1})\cong \Pic(C_n)$. Thus as $\pr_0^*:\Pic(X)\to \Pic(C_0)$ is an isomorphism, it follows by induction that $\pr_n^*:\Pic(X)\to \Pic(C_n)$ is an isomorphism for all $n$, and so $\ds\pr:\Pic(X)\to \dirlim_n \Pic(C_n) = \Pic(\Vhat(\xi))$ is an isomorphism.

Hence the completion map $\Cl(R)\to \Cl(\Rhat)$ is indeed an isomorphism, completing the proof.
\end{proof}

So indeed, $\Cl(\Rc)\to \Cl(\Rbar)$ is an isomorphism. As noted above, this completes the proof of Theorem \ref{self-dual mod l}, and hence of Theorem \ref{self-dual}.

\section{The construction of $M_\infty$}\label{sec:patching}
From now on assume that $\rhobar:G_F\to \GL_2(\F)$ satisfies condition (\ref{TW conditions}) of Theorem \ref{Mult 2^k} (i.e. the ``Taylor-Wiles'' condition). The goal of this section is to construct a module $M_\infty$ over $R_\infty$ satisfying the conditions of Theorem \ref{self-dual}.

We shall construct $M_\infty$ by applying the Taylor--Wiles--Kisin patching method \cite{Wiles,TaylorWiles,Kisin} to a natural system of modules over the rings $\T^D(K)$. For convenience we will follow the ``Ultrapatching'' construction introduced by Scholze in \cite{Scholze}. The primary advantage to doing this is that Scholze's construction is somewhat more ``natural'' than the classical construction, and thus it will be easier to show that $M_\infty$ satisfies additional properties (in our case, that it is self-dual).

\subsection{Ultrapatching}\label{ssec:ultrapatching}

In this subsection, we briefly recall Scholze's construction (while introducing our own notation).

From now on, fix a nonprincipal ultrafilter $\uf$ on the natural numbers $\N$ (it is well known that such an $\uf$ must exist, provided we assume the axiom of choice). For convenience, we will say that a property $\Pc(n)$ holds for \emph{$\uf$-many $i$} if there is some $I\in \uf$ such that $\Pc(i)$ holds for all $i\in I$.

For any collection of sets $\As =\{A_n\}_{n\ge 1}$, we define their \emph{ultraproduct} to be the quotient \[\uprod{\As} = \left(\prod_{n=1}^\infty A_n\right)/\sim\] where we define the equivalence relation $\sim$ by $(a_n)_n\sim (a_n')_n$ if $a_i = a_i'$ for $\uf$-many $i$.

If the $A_n$'s are sets with an algebraic structure (eg. groups, rings, $R$-modules, $R$-algebras, etc.) then $\uprod{\As}$ naturally inherits the same structure. 

Also if each $A_n$ is a finite set, and the cardinalities of the $A_n$'s are bounded (this is the only situation we will consider in this paper), then $\uprod{\As}$ is also a finite set and there are bijections $\uprod{\As}\xrightarrow{\sim} A_i$ for $\uf$-many $i$. Moreover if the $A_n$'s are sets with an algebraic structure, such that there are only finitely many distinct isomorphism classes appearing in $\{A_n\}_{n\ge 1}$ (which happens automatically if the structure is defined by \emph{finitely} many operations, eg. groups, rings or $R$-modules or $R$-algebras over a \emph{finite} ring $R$) then these bijections may be taken to be isomorphisms. This is merely because our conditions imply that there is some $A$ such that $A\cong A_i$ for $\uf$-many $i$ and hence $\uprod{\As}$ is isomorphic to the ``constant'' ultraproduct $\uprod{\{A\}_{n\ge 1}}$ which is easily seen to be isomorphic to $A$, provided that $A$ is finite.

Lastly, in the case when each $A_n$ is a module over a \emph{finite} local ring $R$, there is a simple algebraic description of $\uprod{\As}$. Specifically, the ring $\ds\Rc = \prod_{n=1}^\infty R$ contains a unique maximal ideal $\prF_\uf\in \Spec\Rc$ for which $\Rc_{\prF_\uf}\cong R$ and $\ds\left(\prod_{n=1}^\infty A_n \right)_{\prF_\uf}\cong \uprod{\As}$ as $R$-modules. This shows that $\uprod{-}$ is a particularly well-behaved functor in our situation. In particular, it is exact.

For the rest of this section, fix a power series ring $S_\infty = \O[[z_1,\ldots,z_t]]$ and consider the ideal $\n = (z_1,\ldots,z_n)$.

We can now make our main definitions:

\begin{defn}\label{def:patching}
	Let $\Ms = \{M_n\}_{n\ge 1}$ be a sequence of finite type ${S_\infty}$-modules.
	\begin{itemize}
		\item We say that $\Ms$ is a \emph{weak patching system} if the $S_\infty$-ranks of the $M_n$'s are uniformly bounded.
		\item We say that $\Ms$ is a \emph{patching system} if it is a weak patching system, and for any open ideal $\a\subseteq S_\infty$, we have $\Ann_{S_\infty}(M_i)\subseteq \a$ for all but finitely many $n$.
		\item We say that $\Ms$ is \emph{free} if $M_n$ is free over ${S_\infty}/\Ann_{S_\infty}(M_n)$ for all but finitely many $n$.
	\end{itemize}
	Furthermore, assume that $\Rs = \{R_n\}_{n\ge1}$ is a sequence of finite type local ${S_\infty}$-\emph{algebras}.
	\begin{itemize}
		\item We say that $\Rs = \{R_n\}_{n\ge1}$ is a \emph{(weak) patching algebra}, if it is a (weak) patching system.
		\item If $M_n$ is an $R_n$-module (viewed as an ${S_\infty}$-module via the ${S_\infty}$-algebra structure on $R_n$) for all $n$ we say that $\Ms = \{M_n\}_{n\ge1}$ is a \emph{(weak) patching $\Rs$-module} if it is a (weak) patching system.
	\end{itemize}
	Now for any weak-patching system $\Ms$, we define its patched module to be the $S_\infty$-module
	\[\patch(\Ms) = \invlim_{\a}\uprod{\Ms/\a},\]
	where the inverse limit is taken over all open ideals of $S_\infty$.
	
	If $\Rs$ is a (weak) patching algebra and $\Ms$ is a (weak) patching $\Rs$-module, then $\patch(\Rs)$ inherits a natural $S_\infty$-algebra structure, and $\patch(\Ms)$ inherits a natural $\patch(\Rs)$-module structure.
\end{defn}

In the above definition, the ultraproduct essentially plays the role of pigeonhole principal in the classical Taylor-Wiles construction, with the simplification that is is not necessary to explicitly define a ``patching datum'' before making the construction. Indeed, if one were to define patching data for the $M_n/\a$'s (essentially, imposing extra structure on each of the modules $M_n/\a$) then the machinery of ultraproducts would ensure that the patching data for $\uprod{\Ms/\a}$ would agree with that of $M_n/\a$ for infinitely many $n$. It is thus easy to see that our definition agrees with the classical construction (cf. \cite{Scholze}).

Thus the standard results about patching (cf \cite{Kisin}) may be rephrased as follows:

\begin{thm}\label{patching}
	Let $\Rs$ be a weak patching algebra, and let $\Ms$ be a \emph{free} patching $\Rs$-module. Then:
	\begin{enumerate}
		\item $\patch(\Rs)$ is a finite type $S_\infty$-algebra. $\patch(\Ms)$ is a finitely generated \emph{free} $S_\infty$-module.
		\item The structure map $S_\infty\to \patch(\Rs)$ (defining the $S_\infty$-algebra structure) is injective, and thus $\dim \patch(\Rs) = \dim S_\infty$.
		\item The module $\patch(\Ms)$ is maximal Cohen--Macaulay over $\patch(\Rs)$. $(\lambda,z_1,\ldots,z_t)$ is a regular sequence for $\patch(\Ms)$.
		\item Let $\n = (z_1,\ldots,z_t)\subseteq S_\infty$, as above. Let $R_0$ be a finite type local $\O$-algebra, and let $M_0$ be a finitely generated $R_0$-module. If, for each $n\ge 1$, there are isomorphisms $R_n/\n\cong R_0$ of $\O$-algebras and $M_n/\n\cong M_0$ of $R_n/\n\cong R_0$-modules, then we have $\patch(\Rs)/\n\cong R_0$ as $\O$-algebras and $\patch(\Ms)/\n\cong M_0$ as $\patch(\Rs)/\n\cong R_0$-modules.
	\end{enumerate}
\end{thm}

From the set up of Theorem \ref{patching} there is very little we can directly conclude about the ring $\patch(\Rs)$. However in practice one generally takes the rings $R_n$ to be quotients of a fixed ring $R_\infty$ (which in our case will be a result of Lemma \ref{TW primes}) of the same dimension as $S_\infty$ (and thus as $\patch(\Rs)$). Thus we define a \emph{cover} of a  weak patching algebra $\Rs = \{R_n\}_{n\ge 1}$ to be a pair $(R_\infty,\{\varphi_n\}_{n\ge 1})$ (which we will denote by $R_\infty$ when the $\varphi_n$'s are clear from context), where $R_\infty$ is a complete, local, topologically finitely generated $\O$-algebra of Krull dimension $\dim S_\infty$ and $\varphi_n:R_\infty\to R_n$ is a surjective $\O$-algebra homomorphism for each $n$. We have the following:

\begin{thm}\label{surjective cover}
	If $(R_\infty,\{\varphi_n\})$ is a cover of a weak patching algebra $\Rs$, then the $\varphi_n$'s induce a natural continuous surjection $\varphi_\infty:R_\infty\surjection \patch(\Rs)$. If $R_\infty$ is a domain then $\varphi_\infty$ is an isomorphism.
\end{thm}
\begin{proof}
	The $\varphi_n$'s induce a continuous map $\ds\Phi = \prod_{n\ge 1}\varphi_n:R_\infty\to \prod_{n\ge 1}R_n$, and thus induce continuous maps 
	\[\Phi_\a:R_\infty\xrightarrow{\Phi}\prod_{n\ge 1}R_n\surjection \prod_{n\ge 1}(R_n/\a)\surjection \uprod{\Rs/\a}\]
	for all open $\a\subseteq {S_\infty}$. Hence they indeed induce a continuous map \[\varphi_\infty=(\Phi_\a)_\a:R_\infty\to \invlim_\a \uprod{\Rs/\a} = \patch(\Rs).\]
	Now as $R_\infty$ is complete and topologically finitely generated, it is compact, and thus $\varphi_\infty(R_\infty)\subseteq \patch(\Rs)$ is closed. So to show that $\varphi_\infty$ is surjective, it suffices to show that $\varphi_\infty(R_\infty)$ is dense, or equivalently that each $\Phi_\a$ is surjective.
	
	Now for any $n$ and any open $\a\subseteq$, $R_n/\a$ is a finite set with the structure of a continuous $R_\infty$ algebra (defined by the continuous surjection $\varphi_n:R_\infty\surjection R_n\surjection R_n/\a$) and the cardinalities of the $R_n/\a$'s are bounded. As noted above, this implies that $\uprod{\Rs/\a}$ also has the structure of an $R_\infty$-algebra (which is just the structure induced by $\Phi_\a$). As $R_\infty$ is \emph{topologically} finitely generated, there are only finitely many distinct isomorphism classes of $R_\infty$-algebras in $\{R_n/\a\}_{n\ge 1}$. By the above discussion of ultraproducts, this implies that $R_i/\a\cong \uprod{\Rs/\a}$ as $R_\infty$-algebras for $\uf$-many $i$. But now taking any such $i$, as the structure map $R_\infty\to R_i/\a$ is surjective, and so the structure map $\Phi_\a:R_\infty\to \uprod{\Rs/\a}$ is as well.
	
	The final claim simply follows by noting that if $R_\infty$ is a domain and $\varphi_\infty$ is \emph{not} injective, then $\patch(\Rs)\cong R_\infty/\ker\varphi_\infty$ would have Krull dimension strictly smaller than $R_\infty$, contradicting our assumption that $\dim R_\infty = \dim S_\infty = \dim\patch(\Rs)$.
\end{proof}

In order to construct the desired module $M_\infty$ over $R_\infty$ satisfying the conditions of Theorem \ref{self-dual}, we will construct a weak patching algebra $\Rs^\square$ covered by $R_\infty$, and a free patching $\Rs^\square$-module $\Ms^\square$, and then define $M_\infty = \patch(\Ms^\square)$.

\subsection{Spaces of automorphic forms}\label{ssec:M_psi}

In this section, we will construct the spaces of automorphic forms $M(K)$ and $M_\psi(K)$ that will be used in Section \ref{ssec:Ms^sqaure} to construct the patching system $\Ms^\square$, producing $M_\infty$.

Recall that $\rhobar:G_F\to \GL_2(\Fbar_\ell)$ is assumed to be a Galois representation satisfying all of the conditions of Theorem \ref{Mult 2^k}. In particular $\K^D(\rhobar)\ne\es$, so that $\rhobar = \rhobar_m$ for some $K\in\K^D(\rhobar)$ and some $m\subseteq \T^D(K)$. By enlarging $\O$ if necessary, assume that $\T^D(K)/m = \F$ and $\F$ contains all eigenvalues of $\rhobar(\sigma)$ for all $\sigma\in G_F$.

Since the results of Theorems \ref{Mult 2^k} and \ref{R=T} are known classically in the case when $F = \Q$ and $D = \GL_2$, we will exclude this case for convenience. Thus we will assume that $D(F)$ is a division algebra. 

For any $K\in \K^D(\rhobar)$, define $M(K)= S^D(K)_m^\dual$ if $D$ is totally definite and \[M(K)= \Hom_{R_{F,S}(\rhobar)[G_F]}(\rho^{\univ},S^D(K)_m^\dual)\] if $D$ is indefinite. Note that this definition depends only on the $\T^D(K)_m$-module structure of $S^D(K)^\dual_m$, and not on the specific choice of $S$ in $R_{F,S}(\rhobar)$. Give $M(K)$ its natural $\T^D(K)_m$-module structure.

\begin{rem}
	The purpose of the definition of $M(K)$ in the indefinite case is to ``factor out'' the Galois action on $S^D(K)^\dual$. This construction was described by Carayol in \cite{Carayol2}. As in \cite{Carayol2} we have that the natural evaluation map $M(K)\otimes_{R_{F,S}(\rhobar)}\rho^{\univ}\to S^D(K)_m^\dual$ is an isomorphism, and so $S^D(K)_m^\dual\cong M(K)^{\oplus 2}$ as $\T^D(K)_m$-modules.
	
	 If we did not do this, and only worked with $S^D(K)^\dual_m$, then the module $M_\infty$ we will construct would have generic rank 2 instead of generic rank 1, and so we would not be able to directly apply Theorem \ref{self-dual}.
\end{rem}

Note that it follows from the definitions that $\dim_{\F}M(K)/m = \nu_{\rhobar}(K)$ for all $K$.

For technical reasons (related to the proof of Lemma \ref{TW primes}) we cannot directly apply the patching construction to the modules $M(K)$. Instead, it will be necessary to introduce ``fixed-determinant'' versions of these spaces, $M_\psi(K)$.

We now make the following definition: If $D$ is definite, a level $K\subseteq D^\times(\A_{F,f})$ is \emph{sufficiently small} if for all $t\in D^\times(\A_{F,f})$ we have $K\A_{F,f}^\times\cap (t^{-1}D^\times(F)t) = F^\times$. This is condition (2.1.2) in \cite{KisinFM}.

If $D$ is indefinite, a level $K\subseteq D^\times(\A_{F,f})$ is \emph{sufficiently small} if 
for all $t\in D^\times(\A_{F,f})$ the action of $(K\A_{F,f}^\times\cap (t^{-1}D^\times(F)t)/F^\times$ on $\half$ is free. Note that this implies that the Shimura variety $X_K$ does not contain any elliptic points.

The importance of considering sufficiently small levels is the following standard lemma:

\begin{lemma}\label{sufficiently small}
Let $K\subseteq D^\times(\A_{F,f})$ be a level, and let $K'\unlhd K$ be level which is a normal subgroup of $K$. Consider a finite subgroup $G \le K\A_{F,f}^\times/K'F^\times$, and let $\O[G]$ be its group ring and $\a_G\subseteq \O[G]$ be the augmentation ideal. Also let $m$ be a non-Eisenstein maximal ideal of $\T^D_K$. If $K$ is sufficiently small then:
\begin{enumerate}
\item If $D$ is totally definite (resp. indefinite) then $G$ acts freely on the double quotient $D^\times(F)\backslash D^\times(\A_{F,f})/K'$ (resp. the Riemann surface $X^D(K') = D^\times(F)\backslash \left(D^\times(\A_{F,f})\times \half\right)/K'$) by right multiplication.
\item $S^D(K')_m^\dual$ is a finite projective $\O[G]$-module.
\item If $G = KF^\times/K'F^\times$ then the operator $\ds\sum_{g\in G}g$ induces an isomorphism 
$S^D(K')_m^\dual/\a_GS^D(K')_m^\dual\xrightarrow{\sim} S^D(K)_m^\dual$.
\end{enumerate}
\end{lemma}
\begin{proof}
The definite case essentially follows from the argument of \cite[Lemma (2.1.4)]{KisinFM}. In the indefinite case, (1) is true by definition, and the argument of \cite[Lemme 3.6.2]{BD} shows that (2) and (3) follow from (1).
\end{proof}

This lemma will allow us to construct the desired free patching system $\Ms^\square$. However, in order to use this lemma it will be necessary to first restrict our attention to sufficiently small levels $K$. First by the conditions on $\rhobar$ and \cite[Lemma 4.11]{DDT} we may pick a prime $w\not\in \Sigma^D_\ell$ satisfying
\begin{itemize}
\item $\Nm(w)\not\equiv 1\pmod{\ell}$
\item The ratio of the eigenvalues of $\rhobar(\Frob_w)$ is not equal to $\Nm(w)^{\pm 1}$ in $\Fbar_\ell^\times$.
\item For any nontrivial root of unity $\zeta$ for which $[F(\zeta):F]\le 2$, $\zeta+\zeta^{-1} \not\equiv 2\pmod{w}$.
\end{itemize}
Define 
\begin{align*}
U_w &= \left\{\begin{pmatrix}a&b\\c&d\end{pmatrix}\in \GL_2(\O_{F,w})\middle| \begin{pmatrix}a&b\\c&d\end{pmatrix} \equiv\begin{pmatrix}1&*\\0&1\end{pmatrix}\pmod{w}\right\}\subseteq \GL_2(F_w)\\
U_w^{-} &= \left\{\begin{pmatrix}a&b\\c&d\end{pmatrix}\in \GL_2(\O_{F,w})\middle| \begin{pmatrix}a&b\\c&d\end{pmatrix} \equiv\begin{pmatrix}*&*\\0&*\end{pmatrix}\pmod{w}\right\}\subseteq \GL_2(F_w)\\
\end{align*}
Let $K_0$ (resp. $K_0^-$) be the preimage of $U_w$ (resp. $U_w^-$)under the map $K^{\min}\injection D^{\times}(\A_{F,f})\surjection \GL_2(F_w)$. We then have the following:

\begin{lemma}\label{lem:aux prime}
$K_0$ is sufficiently small, and we have compatible isomorphisms $\T^D(K^{\min})_m\cong \T^D(K_0)_m$, $S^D(K^{\min})_m\cong S^D(K_0)_m$ and $M(K^{\min})\cong M(K_0)$.
\end{lemma}
\begin{proof}
The fact that $K_0$ is sufficiently small follows easily from last hypothesis on $w$ (cf \cite[(2.1.1)]{Kisin}). As in \cite[Section 4.3]{DDT}, the first two conditions on $w$ imply that $w$ is not a level-raising prime for $\rhobar$ and so we obtain natural isomorphisms $\T^D(K^{\min})_m\cong \T^D(K_0^{-})_m\cong \T^D(K_0)_m$. By the definition of $M(K)$, the isomorphism $M(K^{\min})\cong M(K_0)$ will follow from $S^D(K^{\min})_m\cong S^D(K_0)_m$, so it suffices to prove this isomorphism.

It follows from the argument of \cite[Lemma 2.2]{TaylorMero}\footnote{Note that this proof does not rely on Taylor's assumption that $\Nm(w)\equiv 1$, only on the assumption that the ratio of the eigenvalues of $\rhobar(\Frob_w)$ is not $\Nm(w)^{\pm1}$. Also by the definition of $U_w^-$, no lift of $\rhobar$ occurring in $S^D(K_0^-)$ can have determinant ramified at $w$, and so the fact that we have not yet fixed determinants does not affect the argument.} that there is an isomorphism $S^D(K^{\min})_m\cong S^D(K_0^-)_m$. Now the argument of \cite[Lemma 4.11]{BDJ} implies that the map $S^D(K_0^-)_m\to S^D(K_0)^{K_0/K_0^-}_m$ is an isomorphism (as the assumptions that $\ell>2$ and $\rhobar|_{G_{F(\zeta_\ell)}}$ is absolutely irreducible imply that $\rhobar$ is not ``badly dihedral'', in the sense defined in that argument). Finally, as $K_0/K_0^-\cong ((\O_F/w)^\times)^2$ has prime to $\ell$ order, we get that $S^D(K_0)^{K_0/K_0^-}_m\cong S^D(K_0)_m$, giving the desired isomorphism.
\end{proof}

It now follows that $\nu_{\rhobar}(K^{\min}) = \nu_{\rhobar}(K_0)$. Also be definition, $K_0$ and $K^{\min}$ agree at all places besides $w$, and hence at all places in $\Sigma^D_\ell$. Thus Lemma \ref{R->T} gives a surjection $R^{D,\psi}_{F,S}(\rhobar)\surjection \T^{D}(K_0)_m$ (note that we are using condition (\ref{vexing}) of Theorem \ref{Mult 2^k} here). We will now restrict our attention to levels contained in $K_0$.

For any level $K\subseteq K_0$, let $C_K = F^\times\backslash \A_{F,f}^\times/(K\cap \A_{F,f}^\times)$ denote the image of $\A_{F,f}^\times$ in the double quotient $D^\times(F)\backslash D^\times(\A_{F,f})/K$. Note that this is a finite abelian group. For any finite place $v$ of $F$, let $\varpibar_v$ denote the image of the uniformizer $\varpi_v\in F_v^\times\subseteq\A_{F,f}^\times$ in $C_K$.

By the definition of $S_v$, we see that $\varpibar_v$ acts on $S^D(K)^\dual$ as $S_v$ for all $v\not\in S$, and so we may identify $\O[C_K]$ with a subring of $\T^D(K)$. Specifically, it is the $\O$-subalgebra generated by the Hecke operators $S_v$ for $v\not\in S$.

Now the action of $C_K$ on $D^\times\backslash D^\times(\A_{F,f})/K$ induces an action of $\O[C_K]$ on $M(K)$. By Lemma \ref{sufficiently small} (with $K'=K$ and $G=C_K\injection K\A_{F,f}^\times/K$) $S^D(K)_m^\dual$ is a finite projective $\O[C_K]$-module. Let $m' = m\cap \O[C_K]$, so that $m'$ is a maximal ideal of $\O[C_K]$. It follows that $S^D(K)^\dual_m$ is a finite free $\O[C_K]_{m'}$-module.

Let $C_{K,\ell}\le C_K$ be the Sylow $\ell$-subgroup. Since $C_K$ is abelian, we have $C_K\cong C_{K,\ell}\times (C_K/C_{K,\ell})$ and so $\O[C_K] \cong \O[C_{K,\ell}]\otimes_\O\O[C_K/C_{K,\ell}]$. Now as $C_K/C_{K,\ell}$ has prime to $\ell$ order, by enlarging $\O$ if necessary, we may assume that $\O[C_K/C_{K,\ell}]\cong \O^{\oplus \#(C_K/C_{K,\ell})}$ as an $\O$-algebra, and so $\O[C_K] \cong \O[C_{K,\ell}]^{\oplus \#(C_K/C_{K,\ell})}$. But now as $\O[C_{K,\ell}]$ is a complete local $\O$-algebra (as $\O[\Z/\ell^n\Z]\cong \O[x]/((1+x)^{\ell^n}-1)$ is for any $n$, and $C_{K,\ell}$ is a finite abelian $\ell$-group), it follows that $\O[C_K]_{m'}\cong \O[C_{K,\ell}]$ for any maximal ideal $m'$. Hence there is an embedding $\O[C_{K,\ell}]\injection \T^D(K)_m$ which makes $S^D(K)^\dual_m$ into a finite projective (and hence free) $\O[C_{K,\ell}]$-module.

It follows that $M(K)$ is also a finite free $\O[C_{K,\ell}]$-module. Indeed, this is simply by definition in the case when $D$ is definite. If $D$ is indefinite, this follows from the fact that $M(K)^{\oplus 2} \cong S^D(K)^\dual_m$ is free, and direct summands of free $\O[C_{K,\ell}]$-modules are projective and hence free.

Now fix a character $\psi:G_F\to\O^\times$ for which $m$ is in the support of $\T^D_\psi(K^{\min})$. For any level $K\subseteq K_0$, define an ideal $\If_\psi = \left(\Nm(v)[\varpibar_v]-\psi(\Frob_v)\middle| v\not\in S \right)\subseteq \O[C_{K,\ell}]$. As $m$ is also in the support of $\T^D_\psi(K)$, it follows that $\If_\psi$ contained in the kernel of some map $\T^D_\psi(K)\to \O$ (corresponding to some lift of $\rho:G_F\to \GL_2(\O)$ of $\rhobar$ which is modular of level $K$ and has $\det \rho = \psi$), and so we can deduce that $\O[C_{K,\ell}]/\If_\psi\cong \O$. We may now define $M_\psi(K)=M(K)/\If_\psi M(K)$. It follows that $M_\psi(K)$ is a finite free $\O$-module. Moreover, by definition it follows that $\T^D_\psi(K)_m$ is exactly the image of $\T^D(K)_m$ in $\End_\O(M_\psi(K))$, and $M_\psi(K)= M(K)\otimes_{\T^D(K)_m}\T^D_\psi(K)_m$. 

It is necessary to consider the modules $M_\psi(K)$ instead of $M(K)$, because the patching argument requires us to work with fixed-determinant deformation rings. Fortunately, as $\If_\psi\subseteq m$, we get $\dim_\F M_\psi(K_0)/m = \dim_\F M(K_0)/m=\nu_{\rhobar}(K_0) = \nu_{\rhobar}(K^{\min})$, and so considering the $M_\psi(K)$'s instead of the $M(K)$'s will still be sufficient to prove Theorem \ref{Mult 2^k}.

\subsection{A Patching System Producing $M_\infty$}\label{ssec:Ms^sqaure}

For the rest of this paper, we will take the ring $S_\infty$ from the Section \ref{ssec:ultrapatching} to be $\O[[y_1,\ldots,y_r,w_1,\ldots,w_j]]$, where $r$ is an in Lemma \ref{TW primes} and $j = 4|\Sigma_\ell^D| - 1$ is as in Section \ref{ssec:global def}, and let $\n = (y_1,\ldots,y_r,w_1,\ldots,w_j)$ as before. Note that $\dim S_\infty = r+j+1 = \dim R_\infty$ by Lemma \ref{TW primes}.

We will construct a weak patching algebra $\Rs^\square$ covered by $R_\infty$ using the deformation rings $R^{\square,D,\psi}_{F,S\cup Q_n}(\rhobar)$, and construct a free patching $\Rs^\square$-module $\Ms^\square$ using the spaces $M_\psi(K)$ constructed above. We then take $M_\infty = \patch(\Ms^\square)$. By Theorems \ref{patching} and \ref{surjective cover} it will then follow that $M_\infty$ is maximal Cohen--Macaulay over $R_\infty$. In Section \ref{ssec:prop M_infty}, we will show that $M_\infty$ satisfies the remaining conditions of Theorem \ref{self-dual}.

From now on, fix $S=\Sigma^D_\ell\cup\{w\}$, where $w$ is the prime chosen in Section \ref{ssec:M_psi} above, and fix a collection of \emph{sets} of primes $\Qc = \{Q_n\}_{n\ge 1}$ satisfying the conclusion of Lemma \ref{TW primes}. For any $n$, let $\Delta_n$ be the maximal $\ell$-power quotient of $\ds\prod_{v\in Q_n}(\O_F/v)^\times$. Consider the ring $\Lambda_n=\O[\Delta_n]$, and note that:
\[\Lambda_n\cong\frac{\O[[y_1,\ldots,y_r]]}{\left((1+y_1)^{\ell^{e(n,1)}}-1,\ldots, (1+y_r)^{\ell^{e(n,r)}}-1\right)}\]
where $\ell^{e(n,i)}$ is the $\ell$-part of $\Nm(v)-1 = \# (\O_F/v)^\times$, so that $e(n,i)\ge n$ by assumption. Let $\a_n = (y_1,\ldots,y_r)\subseteq \Lambda_n$ be the augmentation ideal.

Also let $\ds H_n = \ker\left(\prod_{v\in Q_n}(\O_F/v)^\times\surjection\Delta_n\right)$.
For any finite place $v$ of $F$, consider the compact open subgroups of $\GL_2(\O_{F,v})$:
\begin{align*}
\Gamma_0(v) &=\left\{\begin{pmatrix}a&b\\c&d\end{pmatrix}\in \GL_2(\O_{F,v})\middle| \begin{pmatrix}a&b\\c&d\end{pmatrix} \equiv \begin{pmatrix}*&*\\0&*\end{pmatrix}\pmod{v}\right\}\subseteq \GL_2(\O_{F,v})\\
\Gamma_1(v) &=\left\{\begin{pmatrix}a&b\\c&d\end{pmatrix}\in \Gamma_0(v)\middle| a \equiv d\pmod{v}\right\}\subseteq \Gamma_0(v)\subseteq\GL_2(\O_{F,v}).
\end{align*}
Notice that $\Gamma_1(v)\unlhd \Gamma_0(v)$ and we have group isomorphisms
\[\Gamma_0(v)/\Gamma_1(v)\cong \Gamma_0(v)F^\times/\Gamma_1(v)F^\times\xrightarrow{\sim} (\O_F/v)^\times\]
given by $\begin{pmatrix}a&b\\c&d\end{pmatrix}\mapsto ad^{-1}\pmod{v}$. Now let $\ds\Gamma_H(Q_n)\subseteq \prod_{v\in Q_n}\Gamma_0(v)$ be the preimage of $\ds H_n\subseteq \prod_{v\in Q_n}(\O_F/v)^\times$ under the map
\[\prod_{v\in Q_n}\Gamma_0(v)\surjection \prod_{v\in Q_n}\Gamma_0(v)/\Gamma_1(v) \cong \ds\prod_{v\in Q_n}(\O_F/v)^\times\]
finally let $K_n\subseteq K_0$ be the preimage of $\Gamma_H(Q_n)$ under
\[K_0\injection \prod_{v\subseteq \O_F}D^\times(\O_{F,v})\surjection \prod_{v\in Q_n}D^\times(\O_{F,v}).\]
Now for any $n\ge 1$, any $v\in Q_n$ and any $\delta\in (\O_F/v)^\times$, consider the double coset operators $U_v,\langle \delta\rangle_v:S^D(K_n)\to S^D(K_n)$ defined by
\begin{align*}
U_v &= \left[K\begin{pmatrix}\varpi_v&0\\0&1\end{pmatrix}K\right],&
\langle\delta\rangle_v &= \left[K\begin{pmatrix}d&0\\0&1\end{pmatrix}K\right],
\end{align*}
where $d\in \O_F$ is a lift of $\delta\in (\O_F/v)^\times$. Note that $\langle\delta\rangle_v$ does not depend on the choice of $d$. In fact, if $\delta,\delta'\in (\O_F/v)^\times$ have the same image under $(\O_F/v)^\times \to \prod_{v\in Q_n}(\O_F/v)^\times\surjection \Delta_n$, then $\langle\delta\rangle_v = \langle\delta'\rangle_v$. Define
\[\Tbar^D(K_n) = \T^D(K_n)\bigg[U_v,\langle \delta\rangle_v\bigg|v\in Q_n,\delta\in(\O_F/v)^\times \bigg] \subseteq \End_\O(S^D(K_n)),\]
and note that this is a commutative $\O$-algebra extending $\T^D(K_n)$, which is finite free as an $\O$-module. Also for convenience set $\Tbar^D(K_0) = \T^D(K_0)$.

Note that the double coset operators $U_v$ and $\langle \delta\rangle_v$ commute with the action of $\T^D(K_n)$ and (in the case when $D$ is indefinite) $G_F$ on $S^D(K_n)$, and thus they descend to maps $U_v,\langle \delta\rangle_v:M_\psi(K_n)\to M_\psi(K_n)$. Let $\Tbar^D_\psi(K_n)_m$ denote the image of $\Tbar^D(K_n)_m$ in $\End_\O(M_\psi(K_n))$.

As in \cite[(3.4.5)]{Kisin}, $\Tbar^D_\psi(K_n)_m$ has $2^{\#Q_n}$ different maximal ideals (corresponding to the different possible choices of eigenvalue for $U_v$ for $v\in Q_n$). Fix any such maximal ideal $m_{Q_n}\subseteq\Tbar^D_\psi(K_n)_m$. Let $\Tbar_n=(\Tbar^D_\psi(K_n)_m)_{m_{Q_n}}$ and $M_n = M_\psi(K_n)_{m_{Q_n}}$ and $R_n = R^{\psi}_{F,S\cup Q_n}(\rhobar)$. Also define $M_0 = M_\psi(K_0)$, $\Tbar_0 = \T_0= \T^D_\psi(K_0)_m$ and $R_0 = R^{\psi}_{F,S}(\rhobar)$ (note that $M_n$, $\T_n$ and $R_n$ all have fixed determinant $\psi$, but we are suppressing this in our notation).

We now have the following standard result (cf \cite{deShalit,DDT,TaylorMero,Kisin}, and also Lemma \ref{sufficiently small} above):

\begin{prop}\label{aug ideal}
For any $n\ge 1$, there is a surjective map $R_n\surjection \Tbar_n$ giving $M_n$ the structure of a $R_n$-module. There exists an embedding $\Lambda_n\injection R_n$ under which $M_n$ is a finite rank free $\Lambda_n$-module. Moreover, we have $R_n/\a_n\cong R_0$ and $M_n/\a_n\cong M_0$ (so in particular, $\rank_{\Lambda_n}M_n = \rank_\O M_0$)
\end{prop}

Now let $R_n^\square = R^{\square,\psi}_{F,S\cup Q_n}(\rhobar)$, and recall from Section \ref{ssec:global def} that $R_n^\square = R_n[[w_1,\ldots,w_j]]$ for some integer $j$. Using this, we may define framed versions of $\Tbar_n$ and $M_n$. Namely
\begin{align*}
\Tbar^{\square}_n&= R^{\square}_{n}\otimes_{R_{n}}\Tbar_n\cong \Tbar_n[[w_1,\ldots,w_j]]\\
M^{\square}_n&=R^{\square}_{n}\otimes_{R_n}M_n\cong M_n[[w_1,\ldots,w_j]]
\end{align*}
so that $M^{\square}_n$ inherits a natural $\Tbar^{\square}_n$-module structure, and we still have a surjective map $R^{\square}_{n}\surjection \Tbar^{\square}_n$ (and so $M_n^\square$ inherits a $R_n^\square$-module structure). Note that the ring structure of $\Tbar^{\square}_n$ and the $\Tbar^{\square}_n$-module structure of $M^\square_n$ do not depend on the choice of the set $S$, and so we may define this without reference to a specific $S$.

Also for any $n$, consider the ring $\Lambda_n^\square = \Lambda_n[[w_1,\ldots,w_j]] = \O[\Delta_n][[w_1,\ldots,w_j]]$, which we will view as a quotient of the ring $S_\infty = \O[[y_1,\ldots,y_r,w_1,\ldots,w_j]]$ from above. 

Rewriting Proposition \ref{aug ideal} in terms of the framed versions of $R_n$ and $M_n$, we get:

\begin{prop}\label{aug ideal framed}
	There exists an embedding $\Lambda_n^\square\injection R_n^\square$ under which $M_n^\square$ is a finite rank free $\Lambda_n^\square$-module. Moreover, we have $R_n^\square/\n\cong R_0$ and $M_n^\square/\n\cong M_0$ (so in particular, $\rank_{S_\infty}M_n^\square = \rank_{\Lambda_n^\square}M_n^\square = \rank_\O M_0$).
\end{prop}

So in particular, the rings $R_n^\square$ are $S_\infty$-algebras and the modules $M_n^\square$ are $S_\infty$-modules.

Lastly, for any $n\ge 0$, define $R^{\square,D}_n = R^{\square,D,\psi}_{F,S\cup Q_n}(\rhobar)$ and $R_n^D = R^{D,\psi}_{F,S\cup Q_n}(\rhobar)$ (where we take $Q_0=\es$).  Note that the actions of $R_n^\square$ and $R_0$ on $M_n^\square$ and $M_0$ factor through $R^{\square,D}_n$ and $R_0^D$, respectively.

Temporarily writing $R^{\square,\psi}_{\Sigma,\D,\ell} = R_{\loc}^{\psi}/J$, so that $R^{\square,D}_n = R^{\square}_n/J$ for all $n\ge 0$, we also see that
\[R^{\square,D}_n/\n\cong (R^\square_n/\n)/J\cong (R^\square_0/(w_1,\ldots,w_j))/J\cong (R^\square_0/J)/(w_1,\ldots,w_j) \cong R^{\square,D}_0/(w_1,\ldots,w_j)\cong R^D_0.\]

Now we claim that the collections $\Rs^{\square,D}=\{R_n^{\square,D}\}_{n\ge 1}$ and $\Ms^\square = \{M_n^\square\}_{n\ge 1}$ satisfy the necessary conditions to apply Theorem \ref{patching}. Specifically we have:

\begin{lemma}\label{M is patching}
	$\Rs^{\square,D}$ is a weak patching algebra over $R_0^D$ and $\Ms^\square$ is a free patching $\Rs^{\square,D}$-module over $M_0$. Moreover, the surjections $R_\infty\surjection R_n^{\square,D}$ from Theorem \ref{TW primes} induce an isomorphism $R_\infty\cong \patch(\Rs^{\square,D})$.
\end{lemma}
\begin{proof}
	Let $S_\infty' = \O[[y_1,\ldots,y_r]]\subseteq S_\infty$ with ideal $\n' = (y_1,\ldots,y_r) = \n\cap S_\infty'$, so that $R_n$ is a $S_\infty'$-algebra and $M_n$ is a $S_\infty'$-module. By definition, we have $R_n^\square = R_n\otimes_{S_\infty'}S_\infty$ and $M_n^\square = M_n\otimes_{S_\infty'}S_\infty$. Thus by Proposition \ref{aug ideal},
	\begin{align*}
	\rank_{S_\infty}R_n^\square &= \rank_{S_\infty'}R_n = \rank_\O R_0\\
	\rank_{S_\infty}M_n^\square&= \rank_{S_\infty'}M_n = \rank_\O M_0.
	\end{align*}
	Also as $R^{\square,D}_n$ is a quotient of $R^\square_n$, we get that $\rank_{S_\infty}R^{\square,D}_n\le \rank_{S_\infty}R_n^\square = \rank_\O R_0$.
	Thus the $S_\infty$-ranks of the $R_n^\square$'s and $M_n^\square$'s are bounded so $\Rs^{\square,D}$ is a weak patching algebra and $\Ms^\square$ is a weak patching $\Rs$-module.
	
	Also as noted above $R_n^{\square,D}/\n\cong R_0^D$ and $M_n^\square/\n\cong M_0$,
	so $\Rs^{\square,D}$ and $\Ms^\square$ are over $R_0^D$ and $M_0$, respectively.
	
	Now by Proposition \ref{aug ideal framed}, for any $n$ we have, 
	\[I_n=\Ann_{S_\infty} M_n^\square = \Ann_{S_\infty}\Lambda_n^\square = \left(\Ann_{S_\infty'}M_n\right) =  \left((1+y_1)^{\ell^{e(n,1)}},\ldots,(1+y_r)^{\ell^{e(n,r)}}\right)\subseteq S_\infty\]
	(where as above, $e(n,i)\ge n$ for each $i$) and $M_n^\square$ is free over $S_\infty/\Ann_{S_\infty}M_n^\square = \Lambda_n^\square$.
	
	It remains to show that $\Ms^\square$ is a patching system, i.e. that for any open $\a\subseteq S_\infty$, $I_n = \Ann_{S_\infty}M_n^\square\subseteq \a$ for all but finitely many $n$. But as ${S_\infty}/\a$ is finite, and the group $1+m_{S_\infty}$ is pro-$\ell$, the group $(1+m_{S_\infty})/\a = \im(1+m_{S_\infty}\injection {S_\infty}\surjection {S_\infty}/\a)$ is a finite $\ell$-group. Since $1+y_i\in 1+m_{S_\infty}$ for all $i$, there is an integer $K\ge 0$ such that $(1+y_i)^{\ell^K}\equiv 1 \pmod{\a}$ for all $i=1,\ldots,d'$. Then for any $n\ge K$, $e(n,i)\ge n\ge K$ for all $i$, and so indeed $I_n\subseteq \a$ by definition.
	
	The final statement follows from Lemma \ref{surjective cover} after noting that $R_\infty$ is a domain (by Theorem \ref{R^st} and the discussion following it) and $\dim R_\infty = \dim S_\infty$.
\end{proof}

Thus we may define $M_\infty = \patch(\Ms^\square)$. By Theorem \ref{patching} and Lemma \ref{M is patching} we get that $M_\infty$ is maximal Cohen--Macaulay over $\patch(\Rs^\square)\cong R_\infty$ and
\[\dim_\F M_\infty/m_{R_\infty} = \dim_\F (M_\infty/\n)/m_{R_\infty} = \dim_\F M_0/m_{R_0} = \nu_{\rhobar}(K^{\min}).\]

\subsection{The Properties of $M_\infty$}\label{ssec:prop M_infty}

We shall now show that $M_\infty$ satisfies the remaining conditions of Theorem \ref{self-dual}. We start by showing $\genrank_{R_\infty}M_\infty = 1$.

First, the fact that $R_\infty[1/\lambda]$ is formally smooth implies that:

\begin{lemma}\label{gen free}
	$M_0[1/\lambda]$ is free of rank $1$ over $R_0^D[1/\lambda]$. In particular, the natural map $R_0^D[1/\lambda]\surjection \T_0[1/\lambda]$ is an isomorphism.
\end{lemma}
\begin{proof}
	As $M_\infty$ is maximal Cohen--Macaulay over $R_\infty$, $M_\infty[1/\lambda]$ is also maximal Cohen--Macaulay over $R_\infty[1/\lambda]$. Since $R_\infty[1/\lambda]$ is a formally smooth domain, this implies that $M_\infty[1/\lambda]$ is \emph{locally} free over $R_\infty[1/\lambda]$ of some rank, $d$.
	
	Now quotienting by $\n$ we get that $M_\infty[1/\lambda]/\n \cong M_0[1/\lambda]$ is locally free over $R_\infty[1/\lambda]/\n\cong R_0^D[1/\lambda]$ of constant rank $d$. But now $R_0^D[1/\lambda]$ is a finite dimensional commutative $E$-algebra, and hence is a product of local rings. Thus as $M_0[1/\lambda]$ is locally free of rank $d$, it must actually be free of rank $d$.
	
	But now by classical generic strong multiplicity $1$ results we get that $M_0[1/\lambda]$ is free of rank $1$ over $\T_0[1/\lambda]$ (recalling that $\T_0 = \T^D_\psi(K_0)\cong \T^D_\psi(K^{\min})$ and $M_0 = M_\psi(K_0)\cong M_\psi(K^{\min})$), which is a quotient of $R_0^D[1/\lambda]$. Thus $d=1$ and hence $M_0[1/\lambda]\cong R_0^D[1/\lambda]$.
	
	Lastly, as the action of $R_0^D[1/\lambda]$ on $M_0[1/\lambda]$ is free \emph{and} factors through $R_0^D[1/\lambda]\surjection \T_0[1/\lambda]$, we get that $R_0^D[1/\lambda]\surjection \T_0[1/\lambda]$ is an isomorphism.
\end{proof}

It is now straightforward to compute $\genrank_{R_\infty}M_\infty$.

Let $K(R_\infty)$ and $K(S_\infty)$ be the fraction fields of $R_\infty$ and $S_\infty$, respectively. As $R_\infty$ is a finite type free $S_\infty$-algebra, $K(R_\infty)$ is a finite extension of $K(S_\infty)$. It follows that 
\[M_\infty\otimes_{R_\infty}K(R_\infty) \cong M_\infty\otimes_{S_\infty}K(S_\infty).\]
Since $R_\infty$ and $M_\infty$ are both finite free $S_\infty$-modules, we thus get
\begin{align*}
\genrank_{R_\infty}M_\infty &= \dim_{K(R_\infty)}\left[M_\infty\otimes_{R_\infty}K(R_\infty)\right] = \dim_{K(R_\infty)}\left[M_\infty\otimes_{S_\infty}K(S_\infty)\right]\\
&= \frac{\dim_{K(S_\infty)}\left[M_\infty\otimes_{S_\infty}K(S_\infty)\right]}{\dim_{K(S_\infty)}K(R_\infty)}= \frac{\rank_{S_\infty}M_\infty}{\rank_{S_\infty}R_\infty}
\end{align*}
But now for any finite free $S_\infty$ module $A$ we have
\[\rank_{S_\infty}A = \rank_{S_\infty/\n}A/\n = \rank_{\O}A/\n = \dim_E(A/\n)[1/\lambda]\]
and so the fact that $M_\infty[1/\lambda]/\n\cong M_0[1/\lambda]\cong R^D_0[1/\lambda]\cong R_\infty[1/\lambda]/\n$ implies that $\rank_{S_\infty}M_\infty = \rank_{S_\infty}R_\infty$, giving $\genrank_{R_\infty}M_\infty = 1$.

\begin{rem}
	It is worth mentioning here that Shotton's computations of local deformation rings \cite{Shotton} (particularly the fact that $R_\infty$ is Cohen--Macaulay, by Theorem \ref{R^st}) actually imply an integral ``$R=\T$'' theorem (see for instance, Section 5 of \cite{Snowden}).
	
	Specifically one considers the surjection $f:R_0^D\surjection \T_0$. As shown in Lemma \ref{gen free} (see also \cite{Kisin}), $f$ is an isomorphism after inverting $\lambda$ (i.e. $R_0^D[1/\lambda]\cong \T_0[1/\lambda]$). This means that $\ker f\subseteq R_0^D$ is a torsion $\O$-module.
	
	But now $R_\infty$ is Cohen--Macaulay, and $M_\infty$ is a maximal Cohen--Macaulay module over $R_\infty$. Since $(\lambda,y_1,\ldots,y_r,w_1,\ldots,w_j)$ is a regular sequence for $M_\infty$ (by Theorem \ref{patching}(3)) it follows that it is also a regular sequence for $R_\infty$. Thus $R_0^D\cong R_\infty/\n = R_\infty/(y_1,\ldots,y_r,w_1,\ldots,w_j)$ is Cohen--Macaulay and $\lambda$ is a regular element on $R_0^D$ (i.e. a non zero divisor).
	
	But this implies that $R_0^D$ is $\lambda$-torsion free, giving that $\ker f = 0$, so indeed $f:R_0^D\surjection \T_0$ is an isomorphism.
\end{rem}

It remains to show the following:

\begin{prop}\label{M*}
	$M_\infty\cong M_\infty^*= \Hom_{R_\infty}(M_\infty,\omega_{R_\infty})$ as $R_\infty$-modules.
\end{prop}

This will ultimately follow from the fact that the modules $M(K)$ were naturally self-dual:

\begin{lemma}\label{pairing}
For any $n\ge 1$, there is a $\Tbar^D(K_n)_m$-equivariant perfect pairing $M(K_n)\times M(K_n)\to \O$. This induces a $\Tbar^D_\psi(K_n)_m$-equivariant perfect pairing $M_\psi(K_n)\times M_\psi(K_n)\to \O$, and thus a $\Tbar_n$-equivariant perfect pairing $M_n\times M_n\to \O$.
\end{lemma}
\begin{proof}
First note that there is a $\Tbar^D(K_n)$-equivariant perfect pairing $S^D(K_n)\times S^D(K_n) \to \O$. In the totally definite case, this is the monodromy pairing, in the indefinite case it is Poincar\'e duality (although this must be modified slightly in order to make the pairing $\Tbar^D(K_n)$-equivariant, see \cite[3.1.4]{Carayol2}). Completing and dualizing gives a $\T^D(K_n)_m$-equivariant perfect pairing $S^D(K_n)_m^\dual\times S^D(K_n)_m^\dual \to \O$.
	
In the totally definite case, this is already the desired pairing $M(K_n)\times M(K_n)\to \O$. In the indefinite case, it follows from \cite[3.2.3]{Carayol2} that the self-duality on $S^D(K_n)^\dual_m$ implies that $M(K_n)= \Hom_{R_n[G_F]}(\rho^{\univ},S^D(K_n)_m^\dual)$ is also $\T^D(K_n)_m$-equivariantly self-dual.

Now in the notation above we have $\O[C_{K,\ell}]=\O[S_v]_{v\not\in S}\subseteq \T^D(K_n)_m\subseteq \Tbar^D(K_n)_m$, $M(K_n)$ is a finite free $\O[C_{K,\ell}]$-module and $M_\psi(K_n)=M(K_n)/\If_\psi M(K_n)$. To deduce the pairing on $M_\psi(K_n)$, it will suffice to show that $M(K)/\If_\psi M(K)\cong M(K)[\If_\psi]$ as $\T^D(K)_m$-modules.

Consider the character $\phi:C_K\to \O^\times$ defined $\phi(\varpibar_v)=\psi(\Frob_v)/\Nm(v)$ for all $v\not\in S$ (this is well defined by the construction of $C_K$), and note that $\If_\psi$ is generated by the elements $g-\phi(g)$ for all $g\in C_K$. Let $\ds a_\psi = \sum_{g\in C_K}\phi(g)^{-1}g\in \O[C_K]\subseteq \T^D(K)_m$. The standard theory of group rings implies that multiplication by $a_\psi$ induces a short exact sequence
\[0\to \If_\psi\to \O[C_K]\xrightarrow{a_\psi} (\O[C_K])[\If_\psi]\to 0\]
Since $M(K)$ is a free $\O[C_K]$-module, this implies the multiplication by $a_\psi$ induces an isomorphism $a_\psi:M(K)/\If_\psi M(K)\xrightarrow{\sim} M(K)[\If_\psi]$. As $a_\psi\in\O[C_K]\subseteq \T^D(K)_m$, this is the desired isomorphism of $\T^D(K)_m$-modules, and so we indeed have a $\T^D_\psi(K)_m$-equivariant perfect pairing $M_\psi(K)\times M_\psi(K)\to \O$.

The final statement, about the pairing $M_n\times M_n\to \O$ follows by localizing at $m_{Q_n}$.
\end{proof}

To deduce Proposition \ref{M*} from Lemma \ref{pairing}, we shall make use of the following lemma:

\begin{lemma}\label{change of ring}
	If $A$ is a local Cohen--Macaulay ring and $B$ is a local $A$-algebra which is also Cohen--Macaulay with $\dim A = \dim B$, then for any $B$-module $M$,
	\[\Hom_A(M,\omega_A)\cong \Hom_B(M,\omega_B)\]
	as left $\End_B(M)$-modules.
\end{lemma}
\begin{proof}
	By  \cite[\href{http://stacks.math.columbia.edu/tag/08YP}{Tag 08YP}]{stacks-project} there is an isomorphism 
	\[\Hom_A(M,\omega_A)\cong \Hom_B(M,\Hom_A(B,\omega_A))\]
	sending $\alpha:M\to \omega_A$ to $\alpha':m\mapsto (b\mapsto \alpha(bm))$, which clearly preserves the action of $\End_B(M)$ (as $(\alpha\circ \psi)(bm) = \alpha(b\psi(m))$ for any $\psi\in\End_B(M)$). It remains to show that $\Hom_A(B,\omega_A) \cong \omega_B$, which is just Theorem 21.15 from \cite{Eisenbud} in the case $\dim A = \dim B$.
\end{proof}
\begin{proof}[Proof of Proposition \ref{M*}]
By Lemma \ref{pairing}, we have $M_n\cong \Hom_\O(M_n,\O)$ as $R_n^D$-modules for all $n\ge 1$. Now as $\Delta_n$ is a finite group, the ring $\Lambda_n = \O[\Delta_n]$ has Krull dimension 1. Moreover as in the proof of Lemma \ref{M is patching}, $\Lambda_n = \O[[y_1,\ldots,y_r]]/I_n$, where $I_n$ is generated by $r$ elements. Thus $\Lambda_n$ is a complete intersection, and so $\omega_{\Lambda_n} = \Lambda_n$. Thus by Lemma \ref{change of ring} we have $M_n\cong \Hom_{\Lambda_n}(M_n,\Lambda_n)$, again as $R_n^D$-modules.
	
Tensoring with $\O[[w_1,\ldots,w_j]]$, this implies $M_n^\square\cong \Hom_{\Lambda_n^\square}(M_n^\square,\Lambda_n^\square)$ as $R_n^\square$-modules (and hence as $R_\infty$-modules). Moreover, by Lemma \ref{aug ideal}, $M_n^\square$ is finite free over $\Lambda_n^\square$.
	
Now take any open ideal $\a\subseteq S_\infty$. Letting $\Lambda^\square = S_\infty/I_n$ as in the proof of Lemma \ref{M is patching} we have that $I_n\subseteq \a$ for all but finitely many $n$. For any such $n$, we now have:
\begin{align*}
M_n^\square/\a&\cong \Hom_{\Lambda_n^\square}(M_n^\square,\Lambda_n^\square)/\a \cong \Hom_{\Lambda_n^\square}(M_n^\square,\Lambda_n^\square/\a) = \Hom_{\Lambda_n^\square}(M_n^\square,S_\infty/\a)\\
&= \Hom_{S_\infty/\a}(M_n^\square/\a,S_\infty/\a)
\end{align*}
as $R_n^{\square,D}/\a$-modules.
	
Now as noted above, we have that $\uprod{\Rs^{\square,D}/\a}\cong R_i^{\square,D}/\a$ and $\uprod{\Ms^\square/\a}\cong M_i^\square/\a$ for $\uf$-many $i$. Taking any such $i$, the above computation gives that:
	\[\uprod{\Ms^\square/\a} \cong \Hom_{{S_\infty}/\a}\left(\uprod{\Ms^\square/\a},{S_\infty}/\a\right)\]
	as $\uprod{\Rs^{\square,D}/\a}$-modules. Taking inverse limits, it now follows that:
	\[M_\infty = \patch(\Ms^\square)\cong \invlim_\a\Hom_{{S_\infty}/\a}\left(\uprod{\Ms^\square/\a},{S_\infty}/\a\right)\]
	as $\patch(\Rs^{\square,D})$-modules. Now we claim that the right hand side is just $\Hom_{S_\infty}(M_\infty,{S_\infty})$. Using the fact that $M_\infty$, and thus $\Hom_{S_\infty}(M_\infty,{S_\infty})$, is a finite free ${S_\infty}$-module (and thus is $m_{S_\infty}$-adically complete) we get that:
	\[\Hom_{S_\infty}(M_\infty,{S_\infty}) \cong \invlim_\a \Hom_{S_\infty}(M_\infty,{S_\infty})/\a\]
	as $\ds\patch(\Rs^{\square,D}) = \invlim_\a\patch(\Rs^{\square,D})/\a$-modules. But now for any $\a$, as $M_\infty\cong\patch(\Ms^\square)$ is a projective ${S_\infty}$-module:
	\[\Hom_{S_\infty}(M_\infty,{S_\infty})/\a\cong \Hom_{{S_\infty}/\a}(M_\infty/\a,{S_\infty}/\a)\cong \Hom_{{S_\infty}/\a}\left(\uprod{\Ms^\square/\a},{S_\infty}/\a\right)\]
	as $\patch(\Rs^{\square,D})/\a = \uprod{\Rs^{\square,D}/\a}$-modules. So indeed: \[\Hom_{S_\infty}(M_\infty,{S_\infty}) \cong \invlim_\a\Hom_{{S_\infty}/\a}\left(\uprod{\Ms^\square/\a},{S_\infty}/\a\right) \cong M_\infty\] as $\patch(\Rs^{\square,D})$-modules, and hence as $R_\infty$-modules.
	
	But now as $\dim R_\infty = \dim {S_\infty}$ and $S_\infty$ is regular (and thus Gorenstein), Lemma \ref{change of ring} implies that 
	\[M_\infty\cong \Hom_{S_\infty}(M_\infty,{S_\infty}) \cong \Hom_{R_\infty}(M_\infty,\omega_{R_\infty})\]
	as $R_\infty$-modules, as claimed.
\end{proof}

This shows that $M_\infty$ indeed satisfies the conditions of Theorem \ref{self-dual}, and so completes the proof of Theorem \ref{Mult 2^k}.

\subsection{Endomorphisms of Hecke modules}

It remains to show Theorem \ref{R=T}. We first note that Theorem \ref{R=T} can be restated in terms of the objects considered the previous section as follows:
 
\begin{prop}\label{prop:R=T M(K_0)}
The trace map $M(K_0)\otimes_{\T^D(K_0)_m} M(K_0) \to \omega_{\T^D(K_0)_m}$ induced by the perfect pairing from Lemma \ref{pairing} is surjective, and the natural map $\T^D(K_0)_m\to \End_{\T^D(K_0)_m}(M(K_0))$ is an isomorphism.
\end{prop}
\begin{proof}[Proof that Proposition \ref{prop:R=T M(K_0)} implies Theorem \ref{R=T}]
First note that by Lemma \ref{lem:aux prime} it suffices to prove Theorem \ref{R=T} with all of the levels $K^{\min}$ replaced by $K_0$. If $D$ is definite, then self-duality and the definition of $M(K_0)$ implies that $M(K_0)\cong S^D(K_0)_m^\dual\cong S^D(K_0)_m$, and so the statement of Theorem \ref{R=T} is immediate.

Now assume that $D$ is indefinite. To simply notation, write $\T = \T^D(K_0)_m$, $M = M(K_0)$, $S_{K_0} = S^D(K_0)_m$ and $R = R_{F,S}(\rhobar)$. As noted above, we have an isomorphism 
\[M\otimes_{R} \rho^{\univ} = \Hom_{R[G_F]}(\rho^{\univ},S_{K_0}^\dual)\otimes_{R}\rho^{\univ}\xrightarrow{\sim}S_{K_0}^\dual\] by \cite{Carayol2}. It now follows that
\begin{align*}
	\End_{R[G_F]}(S_{K_0})
	&\cong \End_{R[G_F]}(S_{K_0}^\dual)\\
	&\cong \Hom_{R[G_F]}(M\otimes_{R}\rho^{\univ},S_{K_0}^\dual)\\
	&\cong \Hom_{R[G_F]}\left((M\otimes_{R}R[G_F])\otimes_{R[G_F]}\rho^{\univ},S_{K_0}^\dual\right)\\
	&\cong \Hom_{R[G_F]}(M\otimes_{R}R[G_F],\Hom_{R[G_F]}(\rho^{\univ},S_{K_0}^\dual))\\
	&\cong \Hom_{R}(M,\Hom_{R[G_F]}(\rho^{\univ},S_{K_0}^\dual))\\
	&\cong \Hom_{R}(M,M) = \End_{R}(M) = \End_{\T}(M) = \T
\end{align*}
as $\T$-algebras (where we used \cite[\href{http://stacks.math.columbia.edu/tag/00DE}{Tag 00DE}, \href{http://stacks.math.columbia.edu/tag/05DQ}{Tag 05DQ}]{stacks-project}).
\end{proof}

The above work implies the following ``fixed-determinant'' version of Proposition \ref{prop:R=T M(K_0)}:

\begin{prop}\label{prop:R=T M_psi(K_0)}
The trace map $M_\psi(K_0)\otimes_{\T_\psi^D(K_0)_m} M_\psi(K_0) \to \omega_{\T_\psi^D(K_0)_m}$ induced by the perfect pairing from Lemma \ref{pairing} is surjective, and the natural map $\T_\psi^D(K_0)_m\to \End_{\T_\psi^D(K_0)_m}(M_\psi(K_0))$ is an isomorphism. Moreover, the surjection $R^{D,\psi}_{F,S}(\rhobar)\surjection \T_\psi^D(K_0)_m$ from Lemma \ref{R->T} is an isomorphism.
\end{prop}
\begin{proof}
Recall that in the notation of Section \ref{ssec:Ms^sqaure}, $M_\psi(K_0) = M_0$, $\T_\psi^D(K_0)_m = \T_0$ and $R^{D,\psi}_{F,S}(\rhobar) = R_0^D$. As shown above, $M_\infty$ satisfies the hypotheses of Theorem \ref{self-dual}, so by the last conclusion of Theorem \ref{self-dual}, we get that the trace map $\tau_{M_\infty}:M_\infty\otimes_{R_\infty} M_\infty \to \omega_{R_\infty}$ is surjective.

As noted above, $(y_1,\ldots,y_r,w_1,\ldots,w_j)$ is a regular sequence for $M_\infty$, and hence for $R_\infty$. It follows that $R_0^D \cong R_\infty/\n$ is Cohen--Macaulay and $M_0 \cong M_\infty/\n$ is maximal Cohen--Macaulay over $R_0^D$. Moreover, we get that the dualizing module of $R_0^D$ is just $\omega_{R_0^D}\cong \omega_{R_\infty}/\n$.
	
But now quotienting out by $\n$, we thus get a surjective map $M_0\otimes_{R_0^D}M_0\surjection \omega_{R_0^D}$, which (by Lemma \ref{trace}) implies that the trace map $\tau_{M_0}:M_0\otimes_{R_0^D}M_0\to \omega_{R_0^D}$ is also surjective.
	
But now, as in \cite[Lemmas 2.4 and 2.6]{EmTheta}, we have the following commutative algebra result:
	
\begin{lemma}\label{endomorphism}
	Let $B$ be an $\O$-algebra and let $U$ and $V$ be $B$-modules. Assume that $B,U$ and $V$ are all finite rank free $\O$-modules, and we have a $B$-bilinear perfect pairing $\langle\ ,\ \rangle:V\times U\to \O$. Moreover, assume that $U[1/\lambda]$ is free over $B[1/\lambda]$. Define $\phi:U\otimes_BV\to \Hom_{\O}(B,\O)$ by $\phi(u\otimes v)(b) = \langle bu,v\rangle = \langle u,bv\rangle$. Then $\phi$ is surjective if and only if the natural map from $B$ to the center of $\End_B(U)$ is an isomorphism.
\end{lemma}
	
Applying this with $B = R_0^D$, $U = M_0$, $V = M_0^*$ and $\langle\ ,\ \rangle:M_0\times M_0\to \O$ being the perfect pairing from Lemma \ref{pairing} implies that the natural map $R_0^D\to Z(\End_{R_0^D}(M_0))$ is an isomorphism. (Here we have used the fact that $\omega_{R_0^D}\cong \Hom_{\O}(R_0^D,\O)$ as in the proof of Lemma \ref{change of ring}, and $M_0[1/\lambda]\cong R_0^D[1/\lambda]$ by Lemma \ref{gen free}.)
	
But now as $M_0$ is free over $\O$, we get that 
\[\End_{R_0^D}(M_0)\injection \End_{R_0^D[1/\lambda]}(M_0[1/\lambda]) \cong \End_{R_0^D[1/\lambda]}(R_0^D[1/\lambda]) = R_0^D[1/\lambda]\]
and so $\End_{R_0^D}(M_0)$ is commutative. Hence the natural map $R_0^D\to \End_{R_0^D}(M_0)$ is an isomorphism. As the action of $R_0^D$ on $M_0$ factors through $R_0^D\surjection \T_0$, this implies that this map and the map $\T_0\to\End_{\T_0}(M_0)$ are isomorphisms.\footnote{As noted in the remark following Lemma \ref{gen free}, the fact that $R_0^D\surjection\T_0$ is an isomorphism follows more simply from the fact that $R_\infty$ is a Cohen-Macaulay domain, flat over $\O$, and that $R_\infty[1/\lambda]$ is formally smooth (as shown in \cite{Shotton}).}
\end{proof}

It remains to deduce Proposition \ref{prop:R=T M(K_0)} from Proposition \ref{prop:R=T M_psi(K_0)}. As in the proof of Proposition \ref{prop:R=T M(K_0)}, write $\T = \T^D(K_0)_m$, $M = M(K_0)$ and $R = R_{F,S}(\rhobar)$. Also write $\T_\psi = \T_\psi^D(K_0)_m$, $M_\psi = M_\psi(K_0)$ and $R_\psi = R^{\psi}_{F,S}(\rhobar)$ and $R^D_\psi = R^{D,\psi}_{F,S}(\rhobar)\cong \T_\psi$. Lastly recall the group $C_{K_0,\ell}$ from Section \ref{ssec:M_psi} above, and abbreviate this as $C = C_{K_0,\ell}$. Recall that $\O[C]$ is a subalgebra of $\T$ and $M$ is free over $\O[C]$ with we have $M_\psi= M/\If_\psi$.

Our argument will hinge on the following key fact about the structure of $\T$:

\begin{lemma}
There is an isomorphism $\T\cong \T_\psi\otimes_\O\O[C]$.
\end{lemma}
\begin{proof}
As in Lemma \ref{gen free} we have $\T_\psi[1/\lambda]\cong M_\psi[1/\lambda]$ by classical generic multiplicity 1 results. An analogous argument gives $\T[1/\lambda]\cong M[1/\lambda]$. As $\T,\T_\psi,M$ and $M_\psi$ are all finite rank free $\O$-modules, this gives $\rank_\O \T_\psi = \rank_\O M_\psi$ and $\rank_\O \T = \rank_\O M$. Also as $M$ is free over $\O[C]$, $M/\If_\psi\cong M_\psi$ and $\O[C]/\If_\psi\cong \O$, we get that $\rank_\O M = \left(\rank_\O M_\psi \right)\left(\rank_\O\O[C]\right)$. It follows that $\rank_\O \T = \rank_\O(\T_\psi\otimes_\O\O[C])$. As all of these rings are free over $\O$, it will thus suffice to give a surjection $ \T_\psi\otimes_\O\O[C]\surjection \T$.

Now recall the isomorphism $R_{F,S}(\psibar)\widehat{\otimes}_\O R_\psi\xrightarrow{\sim} R$ from Lemma \ref{lem:fixed determinant}. This gives us maps $\alpha:R_{F,S}(\psibar)\to R\surjection \T$ and $\beta:R_\psi\to R\surjection \T$ such that $(\im\alpha)(\im\beta) = \T$. It will thus suffice to show that $\alpha$ and $\beta$ factor through surjections $R_{F,S}(\psibar)\surjection \O[C]$ and $R_\psi\surjection \T_\psi$, respectively.

Note that the map $R\surjection \T$ from Lemma \ref{R->T} is induced by a representation $\rho^{\mod}:G_{F,S}\to \GL_2(\T)$ lifting $\rhobar$, and we have $\tr(\rho^{\mod}(\Frob_v)) = T_v$ and $\det(\rho^{\mod}(\Frob_v)) = \Nm(v)S_v$ for all $v\not\in S$ (cf \cite{Kisin}). 

By Lemma \ref{lem:fixed determinant}, the map $R_{F,S}(\psibar)\to R$ is characterized by the morphism of functors $\rho\mapsto\det\rho$, and so satisfies $\psi^{\univ}(g)\mapsto \det\rho^{\univ}(g)$ for all $g\in G_{F,S}$ (where $\psi^{\univ}$ is the universal lift of $\psibar$). Thus for any $v\not\in S$, we have $\alpha(\psi^{\univ}(\Frob_v)) = \det\rho^{\mod}(\Frob_v) = \Nm(v)S_v$. By Chebotarev density, it follows that $\im\alpha = \O[S_v]_{v\not\in S} = \O[C]\subseteq \T$.

Now for any map $x:\T\to\Ebar$, let $\rho_x:G_{F,S}\to \GL_2(\T)\xrightarrow{x} \GL_2(\Ebar)$ be the modular lift of $\rhobar$ corresponding to $x$. Let $\rho^{\psi}_{x\circ\beta}:G_{F,S}\to \GL_2(R_\psi)\xrightarrow{x\circ\beta}\GL_2(\Ebar)$ be the lift of $\rhobar$ with determinant $\psi$ corresponding to $x\circ\beta:R_\psi\to R\to\T\to\Ebar$. From the construction of the map $R_\psi\to R$, we see that $\rho^{\psi}_{x\circ\beta}$ is the (unique) twist of $\rho_x$ with determinant $\psi$. In particular, $\rho^{\psi}_{x\circ\beta}$ is also modular of level $K_0$. It follows that $\rho^{\psi}_{x\circ\beta}$ is flat at every $v|\ell$, Steinberg at every $v|\D$ and minimally ramified at every $v\in\Sigma^D_\ell$, $v\nmid\ell,\D$. As remarked in Section \ref{ssec:global def}, this implies that $x\circ\beta:R_\psi\to\T\to \Ebar$ factors through $R^D_\psi\cong \T_\psi$ for any $x:\T\to\Ebar$.

But now as $\T$ is a finite free $\O$-module, we have an injection $\T\injection \T\otimes_\O\Ebar$, and as $\T$ is reduced, $\T\otimes_\O\Ebar$ is a finite product of copies of $\Ebar$. Thus by the above, the map $R_\psi\xrightarrow{\beta}\T\injection\T\otimes_\O\Ebar$ factors through $R_\psi\surjection R_\psi^D\cong \T_\psi$, and so $\beta:R_\psi\to \T$ does as well. So indeed $\beta\otimes\alpha$ induces a surjection (and hence an isomorphism) $\T_\psi\otimes\O[C]\to \T$.
\end{proof}

Now as $\O[C]$ is a complete intersection, we have $\omega_{\O[C]}\cong \O[C]$. Thus using Lemma \ref{change of ring} we get
\[\omega_{\T} \cong \Hom_{\O[C]}(\T,\O[C])\cong \Hom_{\O[C]}(\T_\psi\otimes_\O\O[C],\O[C])\cong \Hom_{\O}(\T_\psi,\O)\otimes_\O\O[C]\cong \omega_{\T_\psi}\otimes_{\O}\O[C].\]
In particular, $\omega_\T/\If_\psi\cong \omega_{\T_\psi}$. Now consider the trace map $\tau_M:M\otimes_\T M\to \omega_\T$, and notice that the induced map $\tau_M\otimes_{\O[C]}(\O[C]/\If_\psi):(M/\If_\psi)\otimes_\T(M/\If_\psi)\to \omega_\T/\If_\psi$ can be identified with $\tau_{M_\psi}:M_\psi\otimes_{\T_\psi}M_\psi\surjection \omega_{\T_\psi}$. Since $-\otimes_{\O[C]}(\O[C]/\If_\psi)$ is right-exact and $\tau_{M_\psi}$ is surjective, it follows that $\tau_M$ is surjective as well.

It now follows by the argument in the proof of Proposition \ref{prop:R=T M_psi(K_0)} that the map $\T\to \End_\T(M)$ is an isomorphism. This completes the proof the Proposition \ref{prop:R=T M(K_0)}, and hence of Theorem \ref{R=T}.
\bibliographystyle{amsalpha}
{\footnotesize
\bibliography{refs}}
\end{document}